\documentclass[11pt]{amsart}
\usepackage[T1]{fontenc}
\usepackage[latin1]{inputenc}
\usepackage{a4wide}
\usepackage{setspace}
\usepackage{supertabular}
\usepackage{xypic}
\usepackage{amssymb}
\usepackage{amsthm}
\usepackage[english]{babel}
\usepackage{latexsym}
\date{}
\newtheorem{thm}{Theorem}[section]

\newtheorem{defn}[thm]{Definition}
\newtheorem{rem}[thm]{Remark}
\newtheorem{conj}[thm]{Conjecture}

\newtheorem{Problem}[thm]{Problem}
\newtheorem{prop}[thm]{Proposition}

\newtheorem{lem}[thm]{Lemma}
\newtheorem{cor}[thm]{Corollary}
\newtheorem{notation}[thm]{Notation}
\newenvironment{f-proof}[1][\sc D\'emonstration.]{\begin{trivlist}
\item[\hskip \labelsep {\bfseries #1}]}{\hfill{$\square$}\end{trivlist}}

\newcommand{\Prod}{\displaystyle\prod}
\newcommand{\fonc}[5]{
 \begin{array}{cccc}
 #1: & #2 & \longrightarrow & #3\\
     & #4 & \longmapsto & #5
 \end{array}
}
\newcommand{\appl}[4]{
 \begin{array}{cccc}
  #1 & \longrightarrow & #2\\
  #3 & \longmapsto & #4
 \end{array}
}

\begin{document}

\title{Tannakian twists of quadratic forms and orthogonal Nori motives}

\author{Ph. Cassou-Nogu\`es and B. Morin}

\maketitle

\begin{abstract}
We revisit classical results of Serre, Fr\"ohlich and Saito in the theory of quadratic forms. Given a neutral Tannakian category $(\mathcal{T},\omega)$ over a field $k$ of characteristic $\neq 2$, another fiber functor $\eta$ over a $k$-scheme $X$ and an orthogonal object $(M,q)$ in $\mathcal{T}$, we show formulas relating the torsor $\bf{Isom}^{\otimes}(\omega,\eta)$ to Hasse-Witt invariants of the quadratic space $\omega(M,q)$ and the symmetric bundle $\eta(M,q)$. We apply this result to various neutral Tannakian categories arising in different contexts. We first consider  Nori's Tannakian category of essentially finite bundles over an integral proper $k$-scheme $X$ with a rational point, in order to study an analogue of the Serre-Fr\"ohlich embedding problem for Nori's fundamental group scheme. Then we consider Fontaine's Tannakian categories of $B$-admissible representations, in order to obtain a generalization of both the classical Serre-Fr\"ohlich formula and Saito's analogous result for Hodge-Tate $p$-adic representations. Finally we consider Nori's category of mixed motives over a number field. These last two examples yield formulas relating the torsor of periods of an orthogonal motive to Hasse-Witt invariants of the associated Betti and de Rham quadratic forms and to Stiefel-Withney invariants of the associated local $l$-adic orthogonal representations. We give some computations for Artin motives and for the motive of a smooth hypersurface. 
\end{abstract}

\section{Introduction} 

Let $k$ be a field of characteristic $\neq 2$, let $(\mathcal{T},\omega)$ be a $k$-linear neutralized Tannakian category, and let $\eta$ be another fiber functor with values in the category of vector bundles over a $k$-scheme $X$. We denote by $\mathcal{G}_{\omega}:=\mathbf{Aut}^{\otimes}(\omega)$ the Tannaka dual and by $T_{\omega_X,\eta}:=\bf{Isom}^{\otimes}(\omega,\eta)$ the $\mathcal{G}_{\omega}$-torsor over $X$ of isomorphisms of tensor functors. An orthogonal object $(M,q)$ in $\mathcal{T}$ is an object $M$ endowed with a symmetric map $M\otimes M\rightarrow \mathbf{1}$ inducing an isomorphism $M\stackrel{\sim}{\rightarrow}M^{\vee}$, where $\mathbf{1}$ is the unit object and $M^{\vee}$ the dual of $M$. Such an orthogonal object yields a quadratic $k$-vector space $(\omega(M),q_{\omega})$ and a symmetric bundle $(\eta(M),q_{\eta})$ over $X$. An orthogonal object $(M,q)$ may be seen as an orthogonal representation $\mathcal{G}_{\omega}\rightarrow \mathbf{O}(q_{\omega})$. 
 Composing with the determinant map $\mathbf{O}(q_{\omega})\rightarrow \mathbb{Z}/2\mathbb{Z}$ we obtain a map $\delta_q^1:\mathrm{Tors}(X,\mathcal{G}_{\omega})\rightarrow H^1(X_{\mathrm{et}},\mathbb{Z}/2\mathbb{Z})$. Similarly, the Pin-extension (see \cite{Fr\"ohlich85} Appendix I and \cite{Jardine92} Appendix)
\begin{equation}\label{ext-i}
1\longrightarrow \mathbb{Z}/2\mathbb{Z}\longrightarrow \widetilde{\mathbf{O}}(q_{\omega})\longrightarrow \mathbf{O}(q_{\omega})\longrightarrow 1
\end{equation}
induces a map $\delta_q^2:\mathrm{Tors}(X,\mathcal{G}_{\omega})\rightarrow H^2(X_{\mathrm{et}},\mathbb{Z}/2\mathbb{Z})$. Here $\mathrm{Tors}(X,\mathcal{G}_{\omega})$ denotes the pointed set of $\mathcal{G}_{\omega}$-torsors over $X$, and $\delta^1_q$ and $\delta^2_q$ are cohomological invariants of degree $1$ and $2$ respectively. It is easily checked that the twist (in the sense of \cite{CCMT14}) of $q_{\omega}$ by the $\mathcal{G}_{\omega}$-torsor $T_{\omega_X,\eta}$ is isometric to $q_{\eta}$. Applying the method of (\cite{CCMT14} Section 6.3), we obtain the following result.
\begin{thm}\label{thm-Tannakian-i}
For any orthogonal object $(M,q)$ in $\mathcal{T}$, we have the following identities in the \'etale cohomology ring $H^*(X_{\mathrm{et}},\mathbb{Z}/2\mathbb{Z})$:
$$\delta_q^1(T_{\omega_X,\eta})=w_1(q_\omega)+ w_1(q_{\eta});$$
$$\delta_q^2(T_{\omega_X,\eta})=w_2(q_\omega)+ w_1(q_\omega)\cdot w_1(q_\omega)+ w_1(q_\omega)\cdot w_1(q_{\eta})+w_2(q_{\eta}).$$
\end{thm}

In Section 4, we apply Theorem \ref{thm-Tannakian-i} to the Tannakian category giving rise to Nori's fundamental group scheme. Let $X$ be an integral proper $k$-scheme with a rational point $x\in X(k)$. Then $\pi^N_1(X/k,x)$ is defined as the Tannakian dual of a certain category of essentially finite bundles. Recall that there is a 1-1 correspondence between the set of maps $\pi^N_1(X/k,x)\rightarrow G$ and the set of triples $(T,G,t)$, where $G$ is a finite flat $k$-group scheme, $T$ a $G$-torsor over $X$ and $t\in T(k)$ is a point lying over $x$. Let $\pi^N_1(X/k,x)\twoheadrightarrow G=\mathrm{Spec}(A)$ be a finite flat quotient. For any generator $\theta$ of the integrals of the Hopf algebra $A$, there is a "unit" quadratic form 
$\kappa_{\theta}:A^D\times A^D\rightarrow k$
and an orthogonal representation
$G\rightarrow \mathbf{O}(\kappa_{\theta})$, where $A^D:=\mathrm{Hom}_k(A, k)$ is the linear dual. The map $\pi^N_1(X/k,x)\twoheadrightarrow G \rightarrow \mathbf{O}(\kappa_{\theta})$ corresponds to an essentially finite symmetric bundle $(\mathcal{V},q_{\theta})$ over $X$.
Pulling back (\ref{ext-i}) along $G\rightarrow \mathbf{O}(\kappa_{\theta})$, we obtain a central extension
\begin{equation}\label{ext2i}
1\longrightarrow \mathbb{Z}/2\mathbb{Z}\longrightarrow \widetilde{G}_{\theta}\stackrel{s_{\theta}}{\longrightarrow} G\longrightarrow 1.
\end{equation}
\begin{thm}\label{Norifond-i} Assume that the extension (\ref{ext2i}) is non-trivial. Then the following assertions are equivalent.
\begin{enumerate}
\item We have $$w_2(\kappa_{\theta})+w_1(\kappa_{\theta})\cdot w_1(\kappa_{\theta})+w_1(\kappa_{\theta})\cdot w_1(q_{\theta})+ w_2(q_{\theta})=0.$$
\item There exist a rational point $\sigma\in G(k)$ and a commutative diagram
\[ \xymatrix{
   & & \widetilde{G}_{\theta} \ar[d]^{s_{\theta}}  \\
\pi^N_1(X/k,x)\ar[rru]^{r}\ar[r]& G\ar[r]^{c_{\sigma}}& G
} \]
where $r$ is faithfully flat and $c_{\sigma}$ denotes the conjugation by $\sigma$.
\end{enumerate}
\end{thm}
For $G$ \'etale, there is a canonical choice $\theta_0$ for $\theta$,  and the form $\kappa_{\theta_0}$ generalize the classical unit form considered by Serre in \cite{Serre84}. In this case, the existence of a commutative diagram as in Theorem \ref{Norifond-i}(2) is analogous to the existence of a solution to Serre's embedding problem \cite{Serre84} in the context of classical Galois theory. Theorem \ref{Norifond-i} expresses the obstruction to the existence of a solution in terms of Hasse-Witt invariants.

In Section 5, we apply Theorem \ref{thm-Tannakian-i} to Fontaine's Tannakian categories of $B$-admissible representations. Let $k$ be a topological field, $K/k$ an extension and $B/k$ a filtered topological algebra such that $C:=\mathrm{gr}^0(B)$ is a t-Henselian field in the sense of (\cite{Gabber14}  Section 3 and \cite{pp}). We assume that $B$ is $(k, G_K)$-regular and that $C^{\times}/C^{\times 2}$ is trivial. We also assume that there exists a morphism of $K$-algebras $\bar{K}\rightarrow \mathrm{Fil}^0(B)$ which  is $G_K$-equivariant. Let $(V,q)$ be a quadratic $k$-vector space and let $\rho: G_K\rightarrow {\bf O}(q)(k)$ be an orthogonal continuous representation which is $B$-admissible. Then the Stiefel-Whitney invariants $sw_1(\rho)$, $sw_2(\rho)$ and the spinor class $sp_2(\rho)$ are well defined. We consider the twisted form induced by $q$ on $D_B(V):=(V\otimes_k B)^{G_K}$ that we denote by $q_D$.
\begin{thm}\label{thm-Fontaine-i}
We have the following identities in the Galois cohomology ring $H^*(G_K,\mathbb{Z}/2\mathbb{Z})$:
\begin{enumerate}
\item $w_1(q_D)=w_1(q)+sw_1(\rho)$.
\item $w_2(q_D)=w_2(q)+w_1(q)\cdot sw_1(\rho)+sw_2(\rho)+sp_2(\rho)$.
\end{enumerate}
\end{thm}
For $k=K$ and $B=\overline{k}$ a separable closure endowed with the discrete topology and the trivial filtration, we obtain the classical Serre-Fr\"ohlich Theorem \cite{Fr\"ohlich85}. For $k$ a $p$-adic field, $K/k$ a finite extension and $B=B_{dR}$ with its usual topology and filtration, we obtain a result of T. Saito (see \cite{Saito95} Theorem 1).

In Section 6, we apply Theorem \ref{thm-Tannakian-i} to the Tannakian category of Nori's mixed motives over a number field $E$ with coefficients in $E$. Here the fiber functor $\omega$ (respectively $\eta$) is given by the Betti (respectively the de Rham) realization. For a motive $M$ over $E$, we denote by $\mathcal{G}_{M/E}$ its motivic Galois group and by $\mathfrak{P}_{M}$ its torsor of formal periods in the sense of Kontsevich \cite{Kontsevich-Zagier01}. Recall that there is a canonical evaluation map $\mathfrak{P}_{M}\rightarrow \mathbb{C}$ whose image is the $E$-algebra $\mathbb{P}_{M}$ of classical periods of $M$. An orthogonal motive $(M,q)$ yields quadratic forms $q_B$ and $q_{dR}$ on the Betti and the de Rham realizations, and a continuous orthogonal representation $\rho_{\lambda}:G_E\rightarrow \mathbf{O}(q_B)(E_{\lambda})$ of the Galois group $G_E$ on the $\lambda$-adic cohomology, for any finite place $\lambda$ of $E$. We also consider their restrictions $\rho_{
 \lambda\mid G_{E_{\lambda}}}$ to the local Galois group $G_{E_{\lambda}}$. Similarly, for a real prime $\lambda$ of $E$, complex conjugation on Betti cohomology gives a representation
$\rho_{\lambda\mid G_{E_{\lambda}}}:G_{\mathbb{R}}\rightarrow \mathbf{O}(q_B)(\mathbb{R})$.
We shall consider the Hasse-Witt invariants $w_i(q_B), w_i(q_{dR})\in H^i(G_E,\mathbb{Z}/2\mathbb{Z})$ and the local Stiefel-Whitney invariants $sw_i(\rho_{\lambda\mid G_{E_{\lambda}}})\in H^i(G_{E_{\lambda}},\mathbb{Z}/2\mathbb{Z})$, as well as the Global Stiefel-Whitney invariant $sw_1(\rho_{\lambda}):=\mathrm{det}(\rho_{\lambda})\in H^1(G_{E},\mathbb{Z}/2\mathbb{Z})$. Finally, an orthogonal motive gives an orthogonal representation $\mathcal{G}_{M/E}\rightarrow\mathbf{O}(q_B)$ hence a central extension
$$1\longrightarrow \mathbb{Z}/2\mathbb{Z} \longrightarrow \widetilde{\mathcal{G}}_{M/E}\longrightarrow \mathcal{G}_{M/E} \longrightarrow 1.$$  
Then $\delta_{q}^2(\mathfrak{P}_M)$ exactly computes the obstruction for the existence of a $\widetilde{\mathcal{G}}_{M/E}$-torsor lifting $\mathfrak{P}_M$. In particular, if $\delta^2_q(\mathfrak{P}_M)\neq 0$ then the canonical map $\mathcal{G}_E\rightarrow \mathcal{G}_{M/E}$ does not factor through $\widetilde{\mathcal{G}}_{M/E}$, where $\mathcal{G}_E$ is the full motivic Galois group. On the other hand, $\delta_{q}^1(\mathfrak{P}_M)\in E^{\times}/E^{\times 2}$ determines the periods of the determinant motive $\mathrm{det}_E(M)$, in the sense that $\mathbb{P}_{\mathrm{det}_E(M)}=E\left(\sqrt{\delta_{q}^1(\mathfrak{P}_M)}\right)$. 

\begin{thm}\label{cor-dR-B-i} Let $(M,q)$ be an orthogonal motive over $E$ with $E$-coefficients. The following identities take place in the ring $H^*(G_E,\mathbb{Z}/2\mathbb{Z})$.
\begin{enumerate}
\item We have 
\begin{eqnarray*}
\delta_{q}^2(\mathfrak{P}_M)&=&w_2(q_B)+ w_1(q_B)\cdot w_1(q_B)+ w_1(q_B)\cdot w_1(q_{dR})+w_2(q_{dR}).
\end{eqnarray*}
\item Assume that either $E=\mathbb{Q}$ or that $E$ is totally imaginary. Then one has
$$\delta_q^2(\mathfrak{P}_M)=\sum_{\lambda} sw_2(\rho_{\lambda\mid{G_{E_{\lambda}}}})+ sp_2(\rho_{\lambda\mid{G_{E_{\lambda}}}})$$
in $H^2(G_E,\mathbb{Z}/2\mathbb{Z})\subset \bigoplus_{\lambda} H^2(G_{E_{\lambda}},\mathbb{Z}/2\mathbb{Z})$, where $\lambda$ runs over the set of all places of $E$.
\item Let $V$ be a projective smooth variety over $E$ of even dimension $n$. Then there exists a unique orthogonal structure on $M=h^n(V)(n/2)$ inducing the usual quadratic form on Betti cohomology $H^n(V(\mathbb{C}),\mathbb{Q}(n/2))$. 
\item  For $(M,q)=(h^n(V)(n/2),q)$ and any finite place $\lambda$, one has
$$\delta_q^1(\mathfrak{P}_M)=sw_1(\rho_{\lambda}).$$

\end{enumerate}
\end{thm}
In order to show that Poincar\'e duality is motivic, so that $h^2(V)(n/2)$ inherits an orthogonal structure, we need to work out the K\"unneth formula for Nori motives in Section 8. Taking this for granted, the proof of Theorem \ref{cor-dR-B-i} is a simple combination of Theorem \ref{thm-Tannakian-i}, Theorem \ref{thm-Fontaine-i} and $p$-adic Hodge theory.  Note also that it is conjectured that the group of connected components of the full motivic Galois group $\mathcal{G}_E$ is the usual absolute Galois group $G_E$. This would force $\mathrm{det}_E(M)$ to be an orthogonal Artin motive, and the identity $\delta_q^1(\mathfrak{P}_M)=sw_1(\rho_{\lambda})$ would immediately follow. One may avoid this conditional argument using the fact, due to T. Saito, that $sw_1(\rho_{l})$ is independent on $l$, which in turn relies on the Weil conjectures.

Finally, we give in Section 7 examples and computations for certain orthogonal motives. We treat in details the case of certain explicit Artin motives, and we compute some of these invariants for the orthogonal motives of the form $h^n(V)(n/2)$ where $V$ is an hypersurface or a complete intersection. For example, we have the following

\begin{cor}  Let $V$ be a smooth hypersurface of even dimension $n\geq 2$ in $\mathbf{P}_E^{n+1}$,  defined by an homogeneous polynomial $f$ of degree $d$, and let 
$(M=h^n(V)(n/2), q)$ be the corresponding orthogonal motive. Then we have
$$\delta^1_q(\mathfrak{P}_M)=\begin{cases}
(-1)^{\frac{d-1}{2}}\cdot\mathrm{disc}_{\mathrm{d}}(f) \textrm{ if  d  is  odd } \\
(-1)^{\frac{d}{2}\cdot\frac {n+2}{2}}\cdot\mathrm{disc}_{\mathrm{d}}(f) \textrm{ if  d  is  even}
\end{cases}$$
and
$$\delta^2_q(\mathfrak{P}_M)=w_2(q_{dR})+
\begin{cases}
\frac{d-1}{2}(-1, -1) \textrm{ if  d  is  odd } \\
\frac{n}{4}(1+\frac{d}{2})(-1, -1)\textrm { if  d  is  even  and } n\equiv 0 \textrm{ mod } 4 \\
\left(-1,\mathrm{disc}_{\mathrm{d}}(f)\right)+(\frac{n+2}{4})(1+\frac{d}{2})(-1, -1) \textrm{ if  d  is  even  and }\  n \equiv 2 \textrm{ mod } 4
\end{cases}$$
where $\mathrm{disc}_{\mathrm{d}}(f)$ is the divided discriminant of $f$ in the sense of \cite{Saito11}.
\end{cor}
\vskip 0.2 truecm
\noindent {\it Acknowledgements:} The authors warmly thank A. Huber and T. Saito for their comments and A. Auel for sharing with them some of his ideas on  the subject.

\tableofcontents

\section{Preliminaries}

Throughout the paper $k$ denotes a field of characteristic $\neq 2$. In this section we review basic definitions on torsors and symmetric bundles over a scheme (see \cite{CCMT14}).  One should note that  a symmetric bundle over $ \mbox{Spec}(k)$ is given by a pair $(V, q)$ where $V$ is a $k$-vector space of finite dimension and $q$ is a non degenerate quadratic form on $V$. 

\subsection{The classifying topos of a pro-group-scheme}  
Let $\mathcal{S}$ be a topos and let $G$ be a  group-object in $\mathcal{S}$. The topos $B_G$ is the category of objects of $\mathcal{S}$ endowed with a left $G$-action. A right $G$-torsor in $\mathcal{S}$ is given by an object $T$ with a $G$-action $\mu:T\times G\rightarrow T$ such that the map from $T$ to the final object is epimorphic and the map $(\mathrm{pr},\mu):T\times G\rightarrow T\times T$ is an isomorphism, where $\mathrm{pr}$ is the first projection.  We denote by ${\bf Tors}(\mathcal{S},G)$ the category of right $G$-torsors in $\mathcal{S}$. Then ${\bf Tors}(\mathcal{S},G)$ is a groupoid, i.e. any morphism of this category is an isomorphism. There is a canonical morphism $\pi:B_G\rightarrow\mathcal{S}$ whose inverse image functor sends a a sheaf $\mathcal{F}$ on $\mathcal{S}$ to the object of $B_G$ given by $\mathcal{F}$ with trivial $G$-action. In particular $\pi^*G$ is the group-object of $B_G$ given by $G$ with trivial left $G$-action. The topos $B_G$ has a un
 iversal $\pi^*G$-torsor $E_G$ given by $G$ on which $G$ acts both on the left and on the right by multiplication. 
\begin{thm}\label{thm-classify}
The functor
$$\appl{{\bf Homtop}_{\mathcal{S}}(\mathcal{S},B_{G})^{op}}{{\bf Tors}(\mathcal{S},G)}{f}{f^*E_G}$$
is an equivalence of groupoids.
\end{thm}
Here ${\bf Homtop}_{\mathcal{S}}(\mathcal{S},B_{G})$ denotes the category of morphisms of $\mathcal{S}$-topoi from $\mathcal{S}$ to $B_{G}$ (with respect to $\mathrm{Id}:\mathcal{S}\rightarrow\mathcal{S}$ and $\pi:B_{G}\rightarrow\mathcal{S}$), and $(-)^{op}$ denotes the opposite category. Given a $G$-torsor $Y$, we denote the corresponding morphism of topoi by the same symbol:
$$Y:\mathcal{S}\longrightarrow B_{G}.$$
The morphism $Y$ can be described as follows: For any sheaf $\mathcal{F}$ on $B_G$, one has
$$Y^*(\mathcal{F})=Y\wedge^G\mathcal{F}:=(Y\times\mathcal{F})/G$$
where $G$ acts on the product $Y\times\mathcal{F}$ via the given right (respectively left) action on $Y$ (respectively on $\mathcal{F}$).
A morphism of group-objects $f:G\rightarrow G'$ in $\mathcal{S}$ (i.e. a morphism of sheaves of groups on $\mathcal{S}$) induces a morphism of topoi $B_f:B_G\rightarrow B_{G'}$, where the inverse image functor $B_f^*$ is simply given by restriction of the group of operators.

We shall use the following variant. Let $G:="\underleftarrow{\mathrm{lim}}"\, G_i$ be a pro-group in $\mathcal{S}$ indexed by a filtered small category $I$. Recall that, for another group $H$, we have 
$$\mathrm{Hom}_{\mathrm{Grp}(\mathcal{S})}(G,H)= \underrightarrow{\mathrm{lim}}\,\mathrm{Hom}_{\mathrm{Grp}(\mathcal{S})}(G_i,H)$$
where $\mathrm{Grp}(\mathcal{S})$ denotes the category of group-objects in $\mathcal{S}$. The classifying topos of the pro-group $G$ is defined as an inverse limit in $2$-category of topoi:
$$B_{G}:=\underleftarrow{\mathrm{lim}}\, B_{G_i}.$$
By definition of the inverse limit, the canonical functor
$${\bf Homtop}_{\mathcal{S}}(\mathcal{S},B_{G})\stackrel{\sim}{\longrightarrow}\underleftarrow{\mathrm{lim}} \,{\bf Homtop}_{\mathcal{S}}(\mathcal{S},B_{G_i})$$
is an equivalence. It follows that $B_G$ classifies \emph{$G$-protorsors}. A $G$-protorsor is given by a $G_i$-torsors $T_i$ for any $i\in I$ and an isomorphism of $G_j$-torsors $f_{\alpha}:T_i\wedge^{G_i}G_j\simeq T_j$ for any $\alpha:i\rightarrow j$ in $I$, such that the isomorphisms $f_{\alpha}$  satisfy the obvious compatibility constraint for composite maps in $I$. We denote by $\mathbf{Tors}(\mathcal{S},G)$ the groupoid of $G$-protorsors. Theorem \ref{thm-classify} gives an equivalence
\begin{equation}\label{protors-equiv}
{\bf Homtop}_{\mathcal{S}}(\mathcal{S},B_{G})^{op}\stackrel{\sim}{\longrightarrow}{\bf Tors}(\mathcal{S},G).
\end{equation}
Let $X$ be a scheme and let $G:="\underleftarrow{\mathrm{lim}}"\, G_i$ be a pro-object in the category of $X$-group schemes of finite type. The classifying topos $B_{G}$ is defined as above by taking $\mathcal{S}$ to be the (big) fppf-topos $X_{\mathrm{fppf}}$. Here $G_i$ is of course seen as a group-object in $X_{\mathrm{fppf}}$ by Yoneda.

\subsection{Hasse-Witt invariants}

\subsubsection{}  Let $k^s/k$ be a separable closure of $k$ . We denote by $G_k:=\mbox{Gal}(k^s/k)$ the absolute Galois group of $k$. Suppose that $q$ is a non degenerate quadratic form of rank $n$ over $k$. We choose a diagonal form $<a_1, a_2,\cdots,a_n> $ of $q$ with $a_i\in k^*$, and consider the cohomology classes $$(a_i)\in k^{\times}/(k^{\times })^2\simeq H^1(G_{k},\mathbb{Z}/2\mathbb{Z}).$$ 
For $1\leq m\leq n$, the \emph{$m^{th}$ Hasse-Witt invariant} of $q$ is defined to be  
\begin{equation}\label{inv1} w_m(q)=\sum_{1\leq i_1<\cdots<i_m\leq n} (a_{i_1})\cdots(a_{i_m})\in H^m(G_{k},\mathbb{Z}/2\mathbb{Z})\end{equation}
where $(a_{i_1})\cdots(a_{i_m})$ is the cup product. Moreover we set $w_0(q)=1$ and $w_m(q)=0$ for $m>n$. It can be shown that $w_m(q)$ does not depend on the choice of the diagonal form of $q$.

\subsubsection{} This definition can be generalized over more general base schemes as follows (see \cite{CCMT14} Section 4). Let $X\rightarrow \mbox{Spec}(\mathbb Z[1/2])$ be a $\mathbb{Z}[1/2]$-scheme. A \emph{symmetric bundle over $X$} is a pair $(V, q)$ where $V$ is a locally free $\mathcal{O}_X$-module of constant finite rank and $q:V\otimes_{\mathcal{O}_X}V\rightarrow \mathcal{O}_X$  is a symmetric map of $\mathcal{O}_X$-modules inducing an isomorphism $V\stackrel{\sim}{\rightarrow}V^{\vee}$, where $V^{\vee}:=\underline{\mathrm{Hom}}_{\mathcal{O}_X}(V,\mathcal{O}_X)$ is the dual of $V$. Given two symmetric bundles $(V, q)$ and $(W, r)$  of same rank over $X$, we consider the sheaf ${\bf Isom}(q,r )$ on $X_{\mathrm{fppf}}$ given by ${\bf Isom}(q,r )(Y):=\mathrm{Isom}(q_Y, r_Y)$ for any $X$-scheme $Y$, where 
$\mathrm{Isom}(q_Y, r_Y)$ is the set of isometries between the symmetric vector bundles $(V_Y, q_Y)$ and $(W_Y, r_Y)$ over $Y$,  obtained from $(V, q)$ and $(W, r)$ by base-change along $Y\rightarrow X$. We define the orthogonal group ${\bf O}(q)$ as the group ${\bf Isom}(q,q )$. Consider the standard form $(\mathcal{O}^n_X, t_n=\sum x_i^2)$ of rank $n$ over $X$. We set ${\bf O}(n):={\bf Isom}(t_n,t_n)$. 

Let $(V,q)$ be a symmetric bundle of rank $n$ over $X$. The sheaf ${\bf Isom}(t_n,q)$ on $X_{\mathrm{fppf}}$ is a ${\bf O}(n)$-torsor on $X_{\mathrm{fppf}}$, hence Theorem \ref{thm-classify} provides us with  a canonical map
\begin{equation}\label{mapisom}
{\bf Isom}(t_n,q):X_{\mathrm{fppf}}\longrightarrow B_{{\bf O}(n)}.
\end{equation}
There is a canonical isomorphism of graded $\mathbb{Z}/2\mathbb{Z}$-algebras
$$H^*(B_{{\bf O}(n)},\mathbb{Z}/2\mathbb{Z})\simeq A[HW_1,\cdots,HW_n]$$
where $HW_m$ has degree $m$ and $A$ is the \'etale cohomology algebra $H^*(X_{\mathrm{et}},\mathbb{Z}/2\mathbb{Z})$. The \emph{$m^{th}$ Hasse-Witt invariant} of $q$ is defined to be the pull-back 
\begin{equation}\label{inv2} 
w_m(q):={\bf Isom}(t_n,q)^*(HW_m)\in H^m_{\mathrm{et}}(X,\mathbb{Z}/2\mathbb{Z})
\end{equation} 
of $HW_m$ along (\ref{mapisom}).
When $X=\mbox{Spec}(k)$ the definitions (\ref{inv1}) and (\ref{inv2}) agree.

\subsection{Twists of quadratic forms}\label{twist}

Let $X\rightarrow \mathrm{Spec}(\mathbb{Z}[1/2])$ be a $\mathbb{Z}[1/2]$-scheme, let $G$ be a pro-group-scheme over $X$, let $(V,q)$ be a symmetric bundle over $X$ of rank $n$, let $\rho:G\rightarrow \mathbf{O}(q)$ be an orthogonal representation and let $T$ be a $G$-protorsor over $X$. The equivalence (\ref{protors-equiv}) yields the morphism of topoi
$$X_{\mathrm{fppf}}\stackrel{T}{\longrightarrow} B_{G}\stackrel{B_\rho}{\longrightarrow} B_{\mathbf{O}(q)}$$
where $B_{\rho}$ is induced by $\rho$. By Theorem \ref{thm-classify} (again), this morphism corresponds to an $\mathbf{O}(q)$-torsor.  Any such ${\bf O}(q)$-torsor is of the form
${\bf Isom}(q ,r)$ for an essentially unique symmetric bundle $(W,r)$ of rank $n$, which we denote by $T\wedge^{\mathcal G}q$. Then $T\wedge^{\mathcal G}q$ is called \emph{the twist of $q$ by the $G$-protorsor $T$}.

\section{Tannakian twists of quadratic forms}


Let $k$ be a field. A $k$-linear neutral Tannakian category $\mathcal{T}$ (or a neutral Tannakian category over $k$) is a rigid abelian tensor category such that $\mathrm{End}(\mathbf{1})=k$ for which there exists a $k$-linear, exact and faithful tensor functor $\mathcal{T}\rightarrow \mathrm{Vec}_k$, where $\mathrm{Vec}_k$ is the category of finite dimensional $k$-vector spaces and $\mathbf{1}\in\mathcal{T}$ is a unit for the tensor product. A Tannakian subcategory of $\mathcal{T}$ is a strictly full subcategory $\mathcal{T}'\subseteq\mathcal{T}$ which is stable by finite tensor products (in particular $\mathbf{1}\in\mathcal{T}'$), finite sums (in particular $\mathcal{T}'$ contains the zero object), duals and subquotients. Let $\omega:\mathcal{T}\rightarrow \mathrm{Vec}_k$ be a fiber functor (i.e. $\omega$ is an exact, $k$-linear and faithful tensor functor). We denote by $\mathcal{G}_{\omega}:={\bf Aut}^{\otimes}(\omega)$ the Tannaka dual of $(\mathcal{T},\omega)$. Recall t
 hat $\mathcal{G}_{\omega}$ is defined via its functor of points: For any $k$-scheme $Y$, the group
$$\mathcal{G}_{\omega}(Y)={\bf Aut}^{\otimes}(\omega)(Y):= \mathrm{Aut}^{\otimes}(\omega_Y)$$
is the group of automorphisms of the tensor functor
$$\omega_Y:\mathcal{T}\stackrel{\omega}{\longrightarrow} \mathrm{Vec}_k \longrightarrow \mathrm{VB}(Y)$$
where the second functor is induced by pull-back along the structure map $Y\rightarrow\mathrm{Spec}(k)$, and $\mathrm{VB}(Y)$ denotes the category of vector bundles over $Y$, i.e. the category of locally free finitely generated $\mathcal{O}_Y$-modules. Then $\mathcal{G}_{\omega}$ is an affine $k$-group scheme (\cite{Deligne-Milne} Theorem 2.11). In fact $\mathcal{G}_{\omega}$ is most naturally defined as a pro-algebraic group. Indeed, let $\{\mathcal{T}_i\subseteq \mathcal{T},i\in I\}$ be a partially ordered set of Tannakian subcategories, such that $\mathcal{T}=\bigcup_{i\in I} \mathcal{T}_i$ and $\mathcal{T}_i$ is generated as a Tannakian category by a single object $X_i\in \mathcal{T}_i$. We denote by $\omega_i:\mathcal{T}_i\hookrightarrow\mathcal{T}\rightarrow \mathrm{Vec}_k$ the fiber functor induced by $\omega$. Then $\mathcal{G}_{\omega_i}:={\bf Aut}^{\otimes}(\omega_i)$ is an affine algebraic group over $k$ (i.e. an affine $k$-group scheme of finite type), and one has
$\mathcal{G}_{\omega}=\underleftarrow{\mathrm{ lim}}\, \mathcal{G}_{\omega_i}$. We consider $\mathcal{G}_{\omega}$ as the pro-object
$$\mathcal{G}_{\omega}:="\underleftarrow{\mathrm{ lim}}"\, \mathcal{G}_{\omega_i}.$$
in the category of affine algebraic groups over $k$.

\subsection{The protorsor $T_{\omega_X,\eta}$} Let $(\mathcal{T},\omega)$ be a $k$-linear neutral Tannakian category as above and let $X$ be a $k$-scheme. We denote by $$\omega_{X}:\mathcal{T}\stackrel{\omega}{\longrightarrow} \mathrm{Vec}_k\stackrel{f^*}{\longrightarrow} \mathrm{VB}(X) $$ the fiber functor over $X$ induced by $\omega$, where $f:X\rightarrow \mathrm{Spec}(k)$ is the structure map.
Let $\eta:\mathcal{T}\rightarrow \mathrm{VB}(X)$ be another fiber functor over $X$. We also denote by $\omega_{i,X}:\mathcal{T}_i\rightarrow \mathrm{VB}(X)$ and $\eta_i:\mathcal{T}_i\rightarrow \mathrm{VB}(X)$ the fiber functors induced by $\omega_{X}$ and $\eta$ respectively. We consider the pro-$X$-scheme $${\bf Isom}^{\otimes}(\omega_X,\eta):="\underleftarrow{\mathrm{ lim}}"\, {\bf Isom}^{\otimes}(\omega_{i,X},\eta_i)$$ 
of isomorphisms of tensor functors from $\omega_X$ to $\eta$. Here ${\bf Isom}^{\otimes}(\omega_{i,X},\eta_i)$ is defined via its functor of points. More precisely,
for any $X$-scheme $Y$, $${\bf Isom}^{\otimes}(\omega_{i,X},\eta_i)(Y)=\mathrm{Isom}^{\otimes}(\omega_{i,Y},\eta_{i,Y})$$ is the set of morphisms of tensor functors $\omega_{i,Y}\rightarrow \eta_{i,Y}$ where
$\omega_{i,Y}$ and $\eta_{i,Y}$ are the fiber functors over $Y$ induced by $\omega_i$ and $\eta_i$ respectively. Note that any morphism $\omega_{i,Y}\rightarrow \eta_{i,Y}$ necessarily is an isomorphism (see \cite{Deligne-Milne}  Proposition 1.13 for the affine case). Then ${\bf Isom}^{\otimes}(\omega_X,\eta)$ is a projective system of faithfully flat affine $X$-schemes of finite type. Similarly, one has
$$\mathcal{G}_{\omega_i}(Y)=\mathrm{Aut}^{\otimes}(\omega_{i,Y})$$
for any $k$-scheme $Y$, hence there is a right action of $\mathcal{G}_{\omega}$ on ${\bf Isom}^{\otimes}(\omega_X,\eta)$ over $X$ which is defined by composition. This action turns ${\bf Isom}^{\otimes}(\omega_X,\eta)$ into a $\mathcal{G}_{\omega}$-protorsor over $X$ (see \cite{Deligne-Milne} Theorem 3.2). We denote by $T_{\omega_{i,X},\eta_i}:={\bf Isom}^{\otimes}(\omega_{i,X},\eta_i)$ the $\mathcal{G}_{\omega_i}$-torsor and by $T_{\omega_X,\eta}:={\bf Isom}^{\otimes}(\omega_X,\eta)$ the $\mathcal{G}_{\omega}$-protorsor just defined, and also by
$$T_{\omega_{i,X},\eta_i}:X_{\mathrm{fppf}}\longrightarrow B_{\mathcal{G}_{\omega_i}}\mbox{ and }T_{\omega_X,\eta}:X_{\mathrm{fppf}}\longrightarrow B_{\mathcal{G}_{\omega}}:= \underleftarrow{\mathrm{ lim}}\, B_{\mathcal{G}_{\omega_i}}$$
the corresponding morphisms of topoi.

\subsection{Orthogonal objects}
We denote by ${\bf 1}$ the unit object of $\mathcal{T}$. For an object $M\in \mathcal{T}$ we denote by $M^{\vee}=\underline{\mathrm{Hom}}(M,{\bf 1})$ its dual.

\begin{defn}
An \emph{orthogonal object} $(M,q)$ of $\mathcal{T}$ is an object $M\in\mathcal{T}$ endowed with a symmetric map
$q:M\otimes M\rightarrow {\bf 1}$ such that the induced map
$M\rightarrow M^{\vee}$ is an isomorphism. 
\end{defn}
We define the category of orthogonal objects of $\mathcal{T}$ in the obvious way: for two orthogonal objects $(M,q)$ and $(M',q')$, a morphism $f:(M,q)\rightarrow (M',q')$ is a map $f:M\rightarrow M'$ in $\mathcal{T}$ such that the diagram
\[ \xymatrix{
M\otimes M\ar[d]^{f\otimes f}\ar[r]^{q}
&{\bf 1}\ar[d]^{\mathrm{Id}}\\
M'\otimes M'\ar[r]^{q'}
&{\bf 1}}
\]
commutes.
 If $(M,q)$ is an orthogonal object of $\mathcal{T}$, then
$$q_\omega:\omega(M)\otimes\omega(M)\simeq \omega(M\otimes M)\stackrel{\omega(q)}{\longrightarrow}\omega({\bf 1})\simeq k$$
is a non-degenerate symmetric bilinear form. By Tannaka duality, an object $M\in\mathcal{T}$ can be seen as a finite dimensional representation of $\mathcal{G}_{\omega}$ of the form  $$\rho_M:\mathcal{G}_{\omega}\longrightarrow {\bf GL}(\omega(M))$$ 
where ${\bf GL}(\omega(M))$ is seen as a $k$-group scheme. If $M$ admits an orthogonal structure $(M,q)$, then we have a factorization
$$\mathcal{G}_{\omega}\longrightarrow {\bf O}(q_\omega)\longrightarrow {\bf GL}(\omega(M))$$
hence an orthogonal representation
$$\rho_{q}:\mathcal{G}_{\omega}\longrightarrow {\bf O}(q_{\omega}),$$
where ${\bf O}(q_{\omega})$ denotes the orthogonal $k$-group scheme associated with $q_{\omega}$.

The category of orthogonal $k$-linear representations of $\mathcal{G}_{\omega}$ is the category of orthogonal objects of the Tannakian category $\mathrm{Rep}_{k}(\mathcal{G}_{\omega})$ of $k$-linear representations of $\mathcal{G}_{\omega}$. The fact that the functor
$$\appl{\mathcal{T}}{\mathrm{Rep}_{k}(\mathcal{G}_{\omega})}{M}{\rho_M}$$
is an equivalence of rigid tensor $k$-linear categories (\cite{Deligne-Milne} Theorem 2.11) immediately gives the following

\begin{prop}
The functor
$(M,q) \longmapsto\rho_{q}$
is an equivalence from the category of orthogonal objects of $\mathcal{T}$ to the category of orthogonal $k$-linear representations of $\mathcal{G}_{\omega}$.
\end{prop}

\subsection{The maps $\delta_{q}^1$ and $\delta_q^2$}\label{sectDelta}

Let $(M,q)$ be an orthogonal object  of $\mathcal{T}$, and let $X$ be a $k$-scheme. The orthogonal representation
$$\rho_{q}:\mathcal{G}_{\omega}\longrightarrow {\bf O}(q_\omega)$$
yields canonical maps
$$\delta_q^i:\mathrm{Tors}(X,\mathcal{G}_{\omega})\longrightarrow H^i(X_{\mathrm{et}},\mathbb{Z}/2\mathbb{Z})$$
for  $i=1,2$. Here $\mathrm{Tors}(X,\mathcal{G}_{\omega})$ is the set of isomorphism classes of $\mathcal{G}_{\omega}$-pro-torsors in the topos $X_{\mathrm{fppf}}$. The maps $\delta_q^1$ and $\delta_q^2$ are defined as follows. 
The map
$$\delta_q^1:\mathrm{Tors}(X,\mathcal{G}_{\omega})\longrightarrow H^1(X_{\mathrm{et}},\mathbb{Z}/2\mathbb{Z}).$$
is induced by the composite map $$\mathcal{G}_\omega\longrightarrow 
{\bf O}(q_\omega)\stackrel{\mathrm{det}_{{\bf O}(q_\omega)}}{\longrightarrow}\mathbb{Z}/2\mathbb{Z}.$$
Indeed, a $\mathcal{G}_{\omega}$-pro-torsor $T$ in $X_{\mathrm{fppf}}$ gives a map $T:X_{\mathrm{fppf}}\rightarrow B_{\mathcal{G}_{\omega}}$, hence a map $X_{\mathrm{fppf}}\rightarrow B_{\mathcal{G}_{\omega}}\rightarrow B_{\mathbb{Z}/2\mathbb{Z}}$, hence a class $$\delta_q^1(T)\in H^1(X_{\mathrm{fppf}},\mathbb{Z}/2\mathbb{Z})\simeq H^1(X_{\mathrm{et}},\mathbb{Z}/2\mathbb{Z}).$$

The map $\delta_q^2$ is defined using the canonical central extension (see \cite{Fr\"ohlich85} Appendix I and \cite{Jardine92} Appendix)
\begin{equation}\label{C_qB}
1\rightarrow \mathbb{Z}/2\mathbb{Z}\rightarrow \widetilde{{\bf O}}(q_\omega) \rightarrow {\bf O}(q_\omega)\rightarrow 1.
\end{equation}
Indeed, (\ref{C_qB}) gives a class $[C_q]\in H^2(B_{{\bf O}(q_\omega)},\mathbb{Z}/2\mathbb{Z})$. For a $\mathcal{G}_{\omega}$-pro-torsor $T$, the class $\delta_q^2(T)$ is defined as the pull-back of $[C_q]$ along the morphism of topoi
$$X_{\mathrm{fppf}}\stackrel{T}{\rightarrow} B_{\mathcal{G}_{\omega}}\rightarrow B_{{\bf O}(q_\omega)}.$$ 

Another way to define $\delta_q^2(T)$ is the following. The representation $\rho_{q}$ factors through $\mathcal{G}_{\omega_i}\rightarrow {\bf O}(q_\omega)$ for any $\mathcal{T}_i$ containing $M$. Pulling back (\ref{C_qB}) along $\mathcal{G}_{\omega_i}\rightarrow {\bf O}(q_\omega)$ gives a central extension of affine group-schemes
\begin{equation}\label{C_qBnew}
1\rightarrow \mathbb{Z}/2\mathbb{Z}\rightarrow \widetilde{\mathcal{G}}_{\omega_i} \rightarrow \mathcal{G}_{\omega_i}\rightarrow 1.
\end{equation}
and an exact sequence of pointed sets (see \cite{Giraud}IV.4.2.10)
\begin{equation}\label{pointed-Exact-Sequ}
...\longrightarrow H^1(X_{\mathrm{fppf}},\widetilde{\mathcal{G}}_{\omega_i}) \longrightarrow H^1(X_{\mathrm{fppf}},\mathcal{G}_{\omega_i})\stackrel{\delta_q^2}\longrightarrow H^2(X_{\mathrm{fppf}},\mathbb{Z}/2\mathbb{Z})\longrightarrow...
\end{equation}
in which $\delta_q^2$ appears as a boundary map. 

\subsection{Comparison formulas} Let $(\mathcal{T},\omega)$ be a $k$-linear neutral Tannakian category as above, let $X$ be a $k$-scheme, let $\eta:\mathcal{T}\rightarrow \mathrm{VB}(X)$ be a fiber functor over $X$, and let $(M,q)$ be an  orthogonal object  of $\mathcal{T}$. 

We consider the quadratic form $q_{\omega}$ over $k$ and the following symmetric vector bundle over $X$:
$$q_\eta:\eta(M)\otimes_{\mathcal{O}_X}\eta(M)\simeq \eta(M\otimes M)\stackrel{\eta(q)}{\longrightarrow}\eta({\bf 1})\simeq \mathcal{O}_X.$$
Cup-product turns $H^*(X_\mathrm{et},\mathbb{Z}/2\mathbb{Z})$ into a commutative $\mathbb{Z}/2\mathbb{Z}$-algebra, and the structure map $f:X\rightarrow \mathrm{Spec}(k)$ induces a ring homorphism
$$H^*(\mathrm{Spec}(k)_\mathrm{et},\mathbb{Z}/2\mathbb{Z})\rightarrow H^*(X_\mathrm{et},\mathbb{Z}/2\mathbb{Z})$$
so that we may consider $w_i(q_{\omega})\in H^i(\mathrm{Spec}(k)_\mathrm{et},\mathbb{Z}/2\mathbb{Z})$ as an element of $H^i(X_\mathrm{et},\mathbb{Z}/2\mathbb{Z})$.

\begin{thm}\label{thm-Tannakian}
For any orthogonal object $(M,q)$ of $\mathcal{T}$, we have identities in $H^*(X_{\mathrm{et}},\mathbb{Z}/2\mathbb{Z})$:
$$\delta_q^1(T_{\omega_X,\eta})=w_1(q_\omega)+ w_1(q_{\eta});$$
$$\delta_q^2(T_{\omega_X,\eta})=w_2(q_\omega)+ w_1(q_\omega)\cdot w_1(q_\omega)+ w_1(q_\omega)\cdot w_1(q_{\eta})+w_2(q_{\eta}).$$
\end{thm}
\begin{proof}
We set $A:=H^*(\mathrm{Spec}(k)_{et},\mathbb{Z}/2\mathbb{Z})$. Recall the isomorphism of $A$-algebras $H^*(B_{{\bf O}(q_{\omega})},\mathbb{Z}/2\mathbb{Z})\simeq A[HW_1(q_{\omega}),\cdots,HW_n(q_{\omega})]$ where $n$ is the rank of $q_{\omega}$. By (\cite{CCMT14} Theorem 1.1), we have 
\begin{equation}\label{deg1}
\mathrm{det}[q_{\omega}]=w_1(q_{\omega})+HW_1(q_{\omega})
\end{equation}
and
\begin{equation}\label{deg2}
[C_{q_{\omega}}]=(w_1(q_{\omega})\cdot w_1(q_{\omega})+ w_2(q_{\omega}))+w_1(q_{\omega})\cdot HW_1(q_{\omega})+HW_2(q_{\omega})
\end{equation}
in the polynomial ring $ A[HW_1(q_{\omega}),\cdots,HW_n(q_{\omega})]$. Theorem \ref{thm-Tannakian} will be obtained by taking the pull-back of the formulas (\ref{deg1}) and (\ref{deg2}) along the morphism
$$X_{\mathrm{fppf}}\stackrel{T_{\omega_X,\eta}}{\longrightarrow} B_{\mathcal{G}_{\omega_X}}\stackrel{B_{\rho}}{\longrightarrow} B_{{\bf O}(q_{\omega_X})}.$$
Indeed, consider the twist $T_{\omega_X,\eta}\wedge^{\mathcal{G}_{\omega_X}}q_{\omega_X}$  of the form $q_{\omega_X}$ by the torsor $T_{\omega_X,\eta}$,  that we simply denote by $T_{\omega_X,\eta}\wedge^{\mathcal{G}_{\omega}}q_{\omega}$ (see Section \ref{twist}). As in the proof of (\cite{CCMT14} Corollary 6.5), we easily see that the pull-back of $HW_i(q_{\omega})$ is $w_i(T_{\omega_X,\eta}\wedge^{\mathcal{G}_{\omega}}q_{\omega})$ and the pull-back of $\mathrm{det}[q_{\omega}]$ (respectively $[C_{q_{\omega}}]$) is by definition $\delta_q^1(T_{\omega_X,\eta})$ (respectively $\delta_q^2(T_{\omega_X,\eta})$). It remains to show the following 
\begin{lem}\label{twisting-lemma}
The twisted form $T_{\omega_X,\eta}\wedge^{\mathcal{G}_{\omega}}q_{\omega}$ is isometric to $q_{\eta}$.
\end{lem}

\begin{proof} We simply denote by $\rho:\mathcal{G}_{\omega_X}\rightarrow {\bf O}(q_{\omega_X})$ the morphism of $X$-group schemes induced by $\rho_{q}$ (by base change). The form $T_{\omega_X,\eta}\wedge^{\mathcal{G}_{\omega}}q_{\omega}$ is determined by the map 
\begin{equation}\label{mapfortwist}
B_{\rho}\circ T_{\omega_X,\eta}:X_{\mathrm{fppf}}\stackrel{T_{\omega_X,\eta}}{\longrightarrow} B_{\mathcal{G}_{\omega_X}}\stackrel{B_{\rho}}{\longrightarrow} B_{{\bf O}(q_{\omega_X})}.
\end{equation}
Let $\mathcal{T}_i\subseteq \mathcal{T}$ be a Tannakian subcategory which is generated by a single object, and such that $M\in\mathcal{T}_i$. The following diagram commutes:
\[ \xymatrix{
& &B_{\mathcal{G}_{\omega_X}}\ar[d]^{}\ar[drr]^{B_{\rho}}
&\\
X_{\mathrm{fppf}}\ar[rru]^{T_{\omega_X,\eta}}\ar[rr]_{T_{\omega_{i,X},\eta_i}}& &B_{\mathcal{G}_{\omega_{i,X}}}\ar[rr]_{B_{\rho^i}}
&  & B_{{\bf O}(q_{\omega_X})}
}
\]
In particular we have an isometry
$$T_{\omega_{i,X},\eta_i}\wedge^{\mathcal{G}_{\omega_i}}q_{\omega}\simeq T_{\omega_X,\eta}\wedge^{\mathcal{G}_{\omega}}q_{\omega}$$
hence one may suppose that $\mathcal{T}=\mathcal{T}_i$ is generated by a single object. In other words, we may assume that $\mathcal{G}_{\omega}$ is an algebraic group and that $T_{\omega_X,\eta}$ is a $\mathcal{G}_{\omega}$-torsor in the usual sense.

In order to ease the notations, we set $T:=T_{\omega_X,\eta}$. So the form $T\wedge^{\mathcal{G}_{\omega}}q_{\omega}$ is determined by the morphism of topoi $B_{\rho}\circ T$ which is in turn determined by the ${\bf O}(q_{\omega_X})$-torsor over $X$ given by $(B_{\rho}\circ T)^*(E_{{\bf O}(q_{\omega_X})})$. Here $E_{{\bf O}(q_{\omega_X})}$ denotes the ${\bf O}(q_{\omega_X})$-torsor of  $B_{{\bf O}(q_{\omega_X})}$ given by the object ${\bf O}(q_{\omega_X})$ on which ${\bf O}(q_{\omega_X})$ acts both on the left and on the right by multiplication. Similarly, the form $q_{\eta}$ is determined by the ${\bf O}(q_{\omega_X})$-torsor ${\bf Isom}(q_{\omega_X},q_\eta)$. Hence $T\wedge^{\mathcal{G}_{\omega}}q_{\omega}$ is isometric to $q_{\eta}$ if and only if $(B_{\rho}\circ T)^*(E_{{\bf O}(q_{\omega_X})})$ and ${\bf Isom}(q_{\omega_X},q_{\eta})$ are isomorphic as ${\bf O}(q_{\omega_X})$-torsors. We have
$$(B_{\rho}\circ T)^*(E_{{\bf O}(q_{\omega_X})})=T^*\circ B_{\rho}^*(E_{{\bf O}(q_{\omega_X})})=
{\bf Isom}^{\otimes}(\omega_X,\eta)\wedge^{\mathcal{G}_{\omega_X}}{\bf O}(q_{\omega_X})$$
where $\mathcal{G}_{\omega_X}$ acts on the left on ${\bf O}(q_{\omega_X})$ via $\rho_q:\mathcal{G}_{\omega_X}\rightarrow{\bf O}(q_{\omega_X})$ and on the right on ${\bf Isom}^{\otimes}(\omega_X,\eta)$ by composition.

Let us define a map
$$\mu:{\bf Isom}^{\otimes}(\omega_X,\eta)\times {\bf O}(q_{\omega_X})\longrightarrow {\bf Isom}(q_{\omega_X},q_{\eta})$$
on points as follows. For any $X$-scheme $Y$, we define 
$$\fonc{\mu(Y)}{\mathrm{Isom}^{\otimes}(\omega_Y,\eta_Y)\times \mathrm{O}(q_{\omega_Y})}{\mathrm{Isom}(q_{\omega_Y},q_{\eta_Y})}{(\alpha,\sigma)}{\alpha(M)\circ \sigma}$$
where $\mathrm{O}(q_{\omega_Y}):=\mathbf{O}(q_{\omega_X})(Y)$ and
$$\alpha(M):\omega_Y(M)\stackrel{\sim}{\longrightarrow} \eta_Y(M)$$ is just given by evaluating the natural transformation $\alpha$ at the object $M$. In order to show that $\mu(Y)$ is well defined, we simply need to observe that $\alpha(M)$ is an isometry, i.e. respects the quadratic forms $q_{\omega_Y}$ and $q_{\eta_Y}$. But $q_{\omega_Y}$ and $q_{\eta_Y}$ are both induced by the pairing $q:M\otimes M\longrightarrow {\bf 1}$ in  $\mathcal{T}$ and $\alpha$ is a morphism of tensor functors, so that we have a commutative diagram
\[ \xymatrix{
\omega_Y(M)\otimes_{\mathcal{O}_Y}\omega_Y(M)\ar[d]^{\alpha(M)\otimes \alpha(M)}\ar[r]^{\hspace{0.5cm}\simeq}&\omega_Y(M\otimes M)\ar[d]^{\alpha(M\otimes M)}\ar[r]^{\hspace{0.6cm}\omega_Y(q)}
&\omega_Y({\bf 1})\ar[r]^{\simeq}\ar[d]^{\alpha({\bf 1})}&\mathcal{O}_Y\ar[d]^{\mathrm{Id}}\\
\eta_Y(M)\otimes_{\mathcal{O}_Y}\eta_Y(M)\ar[r]^{\hspace{0.5cm}\simeq}&\eta_Y(M\otimes M)\ar[r]^{\hspace{0.6cm}\eta_Y(q)}
&\eta_Y({\bf 1})\ar[r]^{\simeq}&\mathcal{O}_Y
}
\]
where the left and the right hand side squares commute because $\alpha$ is a morphism of tensor functors (see \cite{Deligne-Milne} Definition 1.12). The upper row (respectively the lower row) is the form $q_{\omega_Y}$ (respectively $q_{\eta_Y}$). Hence $\alpha(M)$ is an isometry, so that $\mu(Y)$ is well defined, hence so is $\mu$ (note that $\mu(Y)$ is functorial in $Y$).

Moreover, $\mu(Y)$ is invariant under the action of $\mathcal{G}_{\omega_X}(Y)=\mathrm{ Aut}^{\otimes}(\omega_Y)$ on the set $\mathrm{Isom}^{\otimes}(\omega_Y,\eta_Y)\times \mathrm{O}(q_{\omega_Y})$. Indeed, an element $g\in \mathcal{G}_{\omega_X}(Y)$ takes $(\alpha,\sigma)\in \mathrm{Isom}^{\otimes}(\omega_Y,\eta_Y)\times \mathrm{O}(q_{\omega_Y})$ to $(\alpha\circ g^{-1},\rho(g)\circ \sigma)$, and one has
$$\mu(Y)(\alpha\circ g^{-1},\rho(g)\circ \sigma)=\alpha(M)\circ \rho(g)^{-1}\circ \rho(g)\circ \sigma=\alpha\circ \sigma= \mu(Y)(\alpha, \sigma).$$
Hence $\mu$ is invariant under the action of $\mathcal{G}_{\omega_X}$.  Equivalently, $\mu$ induces a morphism
$$\tilde{\mu}:{\bf Isom}^{\otimes}(\omega_X,\eta)\wedge^{\mathcal{G}_{\omega_X}}{\bf O}(q_{\omega_X}):=
({\bf Isom}^{\otimes}(\omega_X,\eta)\times{\bf O}(q_{\omega_X}))/\mathcal{G}_{\omega_X}\longrightarrow {\bf Isom}^{\otimes}(q_{\omega},q_{\eta})$$
in the topos ${\mathrm{Spec}}(k)_{\mathrm{fppf}}$. The group ${\bf O}(q_{\omega_X})$ acts on ${\bf Isom}^{\otimes}(\omega_X,\eta)\wedge^{\mathcal{G}_{\omega_X}}{\bf O}(q_{\omega_X})$ by right multiplication on the second factor and trivially on the first factor. It acts on 
${\bf Isom}^{\otimes}(q_{\omega_X},q_{\eta})$ by composition. The map $\tilde{\mu}$ is obviously ${\bf O}(q_{\omega_X})$-equivariant, hence $\tilde{\mu}$ is a morphism of ${\bf O}(q_{\omega_X})$-torsors. The result follows, since any morphism of torsors is an isomorphism.

\end{proof}

\end{proof}

\section{An analogue of the Serre-Fr\"ohlich embedding problem for Nori's fundamental group scheme}

\subsection{The Serre-Fr\"ohlich embedding problem}
Let $F$ be a field of characteristic $\neq 2$, let $G$ be a finite group and let $p:G_F\rightarrow G$ be a surjective map, i.e. let $\bar{F}/F$ be a separable closure and let $\bar{F}/K/F$ be a Galois sub-extension of group $G$. We consider the {\it $G$-unit form} $(F[G], \kappa)$ where $\kappa$ is the quadratic form such that $\kappa(g, h)=\delta_{g, h}$ (Kronecker symbol) for $g, h \in G$. The action of $G$ on itself by left multiplication  extends to an orthogonal representation 
\begin{equation}\label{unit-serre}
\rho:G_F\rightarrow G\rightarrow \mathbf{O}(\kappa).
\end{equation} 
The Pin-extension (see \cite{Jardine92} Appendix) gives a central extension of discrete groups
\begin{equation}\label{Pinbar}
1\rightarrow \mathbb{Z}/2\mathbb{Z}\rightarrow \widetilde{\mathbf{O}}(\kappa)(\overline{F})\rightarrow \mathbf{O}(\kappa)(\overline{F})\rightarrow 1.
\end{equation}
Pulling back (\ref{Pinbar}) via $G\rightarrow \mathbf{O}(\kappa)(F) \rightarrow \mathbf{O}(\kappa)(\overline{F})$ we obtain a 
central extension of finite discrete groups
 $$1\rightarrow \mathbb{Z}/2\mathbb{Z}\rightarrow \widetilde{G}\rightarrow G\rightarrow 1 $$
which we suppose to be non-trivial. The embedding problem attached to $\tilde G\rightarrow G$ can be expressed as follows: Is there a commutative triangle
\[ \xymatrix{
   & & \widetilde{G} \ar[d]^s  \\
G_F\ar[rru]^{r}\ar[rr]^p& & G
} \]
where $r$ is surjective? In other words, is there an $F$-embedding $K\hookrightarrow \widetilde{K}$ where $\widetilde{K}/F$ is a Galois extension of group $\widetilde{G}$? This problem has a solution if and only if the Stiefel-Whitney class $sw_{2}(\rho)$ vanishes.  It follows from Serre's comparison formula (\cite{Serre84} Th\'eor\`eme 1)  that $sw_{2}(\rho)\in H^2(G_F,\mathbb{Z}/2\mathbb{Z})$ can be computed in terms of the Hasse-Witt invariant $w_2(\mathrm{Tr}_{K/F})$, where $\mathrm{Tr}_{K/F}$ is the quadratic from $x\mapsto \mathrm{Tr}_{K/F}(x^2)$. This result can be applied to decide whether or not the previous embedding problem has a solution in explicit examples.
 
\subsection{An analogue for Nori's fundamental group scheme}
The aim of this section is to formulate a similar embedding problem for Nori's fundamental group scheme and to deduce from the comparison formula of Theorem \ref{thm-Tannakian} a necessary and sufficient condition for the existence of a solution. More precisely, let $X$ be a  proper integral scheme over a field $k$ of characteristic $\neq 2$, with a $k$-rational point $x\in X(k)$. We consider Nori's fundamental group scheme $\pi^N_1(X/k,x)$ (see \cite{Nori76} and \cite{Szamuely09} 6.7). Let $G=\mathrm{Spec}(A)$ be a finite flat quotient of $\pi^N_1(X/k,x)$. We denote by $p:\pi^N_1(X/k,x)\rightarrow G$ the projection. 
\subsubsection{Generalization of the $G$-unit form}
The linear dual $A^D=\mathrm{Hom}_k(A, k)$ of $A$ has a natural structure of Hopf $k$-algebra. We denote by $S^D$ its antipode.
\begin{defn}\label{def-unit} For any generator $\theta$ of the $k$-vector space of the integrals of $A$ we define the form 
$$\fonc{\kappa_\theta}{A^D\times A^D}{k}{(u,v)}{(S^D(u)v)(\theta)}$$
\end{defn} Then $\kappa_{\theta}$ is a non degenerate quadratic form over $k$. The action of $A^D$ on itself by left multiplication corresponds to a linear representation $\rho$ of $G$ on $A^D$, dual of the regular representation of $G$. One can prove (\cite{CCMT15} Proposition 5.1)  that $\rho: G\rightarrow \mathbf{GL}_k(A^D)$ factors through $\mathbf{O}(\kappa_{\theta})$  and thus provides  an orthogonal representation of $G$. We refer to $(A^D, \kappa_{\theta})$ as a {\it  unit form}. If $G$ is \'etale then there is a canonical generator $\theta_0$ of the integrals of $A$, hence a canonical orthogonal representation $$\pi^N_1(X/k,x)\rightarrow G\rightarrow \mathbf{O}(\kappa_{\theta_0}),$$ 
which generalizes the unit form (\ref{unit-serre}).

\subsubsection{} For any generator $\theta$ of the integrals of $A$, we consider the representation $G\rightarrow \mathbf{O}(\kappa_{\theta})$ defined above. The Pin-extension (see \cite{Jardine92} Appendix again) induces a central extension
of finite flat $k$-group-schemes
$$1\rightarrow \mathbb{Z}/2\mathbb{Z}\rightarrow \widetilde{G}_{\theta}\stackrel{s_{\theta}}{\rightarrow}G\rightarrow 1$$
which we suppose to be non-trivial. Notice that $\widetilde{G}_{\theta}:= G\times_{\mathbf{O}(\kappa_{\theta})} \widetilde{\mathbf{O}}(\kappa_{\theta})$ depends on $\theta$. The embedding problem we are interested is the following:  
\begin{Problem}\label{prob} Let $\theta$ be a generator of the integrals of $A$. Is there a faithfully flat morphism $r:\pi^N_1(X/k,x)\rightarrow \widetilde{G}_{\theta}$ which renders
the following triangle
\[ \xymatrix{
   & & \widetilde{G}_{\theta} \ar[d]^{s_{\theta}}  \\
\pi^N_1(X/k,x)\ar[rru]^{r}\ar[rr]^p& & G
} \]
commutative? 
\end{Problem}
In the next section, we  give a necessary and sufficient condition for the existence of a solution to Problem \ref{prob} up to conjugation by an element of $G(k)$.

\subsection{Existence of a solution in terms of Hasse-Witt invariants}

Let $X$ be a proper integral scheme over a field $k$ of characteristic $\neq2$. Let $x\in X(k)$ be a rational point. In this situation, Nori defines the category $\mathrm{EF}(X)$ of essentially finite bundles on $X$. Then $\mathrm{EF}(X)$ is a full subcategory of the category $\mathrm{VB}(X)$ of vector bundles on $X$, which is closed under direct sums and tensor products. The pull-back functor along the map $x:\mathrm{Spec}(k)\rightarrow X$ gives a fiber functor $\omega:\mathrm{EF}(X)\rightarrow\mathrm{Vec}_k$, and $\mathrm{EF}(X)$ is a neutral Tannakian category. Its Tannaka dual 
$$\pi^N_1(X/k,x):=\mathbf{Aut}^{\otimes}(\omega)$$
is by definition the fundamental group-scheme of $X$ with respect to the rational point $x$. Then $\pi^N_1(X/k,x)$ is a projective limit of finite $k$-group schemes. It classifies pointed torsors under finite group-schemes: For any finite group-scheme $G$ over $k$, one has a bijection between $\mathrm{Hom}_k(\pi^N_1(X/k,x),G)$ and the set of triples $(G,T,t)$ where $T\rightarrow X$ is a $G$-torsor and $t\in T(k)$ is a rational point lying over $x\in X(k)$.
The Tannakian category $\mathrm{EF}(X)$ has a distinguished fiber-functor over $X$: the inclusion functor
$$\eta:\mathrm{EF}(X)\longrightarrow \mathrm{VB}(X).$$
As in the previous section, we also consider the composition
$$\omega_X:\mathrm{EF}(X)\stackrel{\omega}{\longrightarrow}\mathrm{Vec}_k\stackrel{f^*}{\longrightarrow} \mathrm{VB}(X)$$
where $f:X\rightarrow\mathrm{Spec}(k)$ is the structure map. Now let $G$ be a finite flat quotient of $\pi^N_1(X/k,x)$ given by the faithfully flat morphism $s:\pi^N_1(X/k,x)\twoheadrightarrow G$. The regular representation of $G$ gives a representation $\pi^N_1(X/k,x)\rightarrow G\rightarrow \mathbf{GL}_k(V)$. By Tannaka duality, we obtain an object $\mathcal{R}\in\mathrm{EF}(X)$. Let $\mathcal{V}:=\mathcal{R}^{\vee}$ be the dual of $\mathcal{R}$ and  let $<\mathcal{V}>=<\mathcal{R}>\subseteq\mathrm{EF}(X)$ be the Tannakian category generated by $\mathcal{V}$. Let $\omega_{X\mid<\mathcal{V}>}$ and $\eta_{\mid<\mathcal{V}>}$ the restrictions of $\omega_X$ and $\eta$. Then 
$$\mathbf{Aut}^{\otimes}(\omega_{X\mid<\mathcal{V}>})\simeq G\times_k X$$
and
$$T:=\mathbf{Isom}^{\otimes}(\omega_{X\mid<\mathcal{V}>},\eta_{\mid<\mathcal{V}>})$$
is a $G$-torsor over $X$. Note that $T$ has a canonical $k$-rational point $t:\mathrm{Spec}(k)\rightarrow T$ over $x$, which is given by the canonical isomorphism of tensor functors
$$x^*\circ\omega_{X\mid<\mathcal{V}>}\simeq x^*\circ \eta_{\mid<\mathcal{V}>}.$$
The triple $(G,T,t)$ corresponds to the morphism $s:\pi^N_1(X/k,x)\twoheadrightarrow G$.

Let $\theta$ be a generator of the integrals of $A$, let $(A^D,\kappa_{\theta})$ be the quadratic space of Definition \ref{def-unit} and let $\rho_{\theta}:G\rightarrow\mathbf{O}(\kappa_{\theta})$ be the corresponding orthogonal representation. By Tannaka duality again, $\rho_{\theta}$ endows the essentially finite vector bundle $\mathcal{V}$ with a canonical non-degenerate symmetric bilinear form
$$q_{\theta}:\mathcal{V}\otimes_{\mathcal{O}_X}\mathcal{V}\rightarrow \mathcal{O}_X.$$ 
Then we consider the canonical central Pin-extension
$$1\rightarrow \mathbb{Z}/2\mathbb{Z}\rightarrow \widetilde{\mathbf{O}}(\kappa_{\theta})\rightarrow \mathbf{O}(\kappa_{\theta})\rightarrow 1$$
and its base change
\begin{equation}\label{ext}
1\rightarrow \mathbb{Z}/2\mathbb{Z}\rightarrow \widetilde{G}_{\theta}\stackrel{s_{\theta}}{\rightarrow} G\rightarrow 1
\end{equation}
along the map $G\rightarrow \mathbf{O}(\kappa_{\theta})$. We consider the Hasse-Witt invariants $w_i(q_{\theta}),w_i(\kappa_{\theta})\in H^i(X_{\mathrm{et}},\mathbb{Z}/2\mathbb{Z})$ for $i=1,2$.  Finally, if $\sigma\in G(k)$ is a rational point we denote by $c_{\sigma}=\sigma^{-1}\cdot (-)\cdot\sigma$ the conjugation map.
\begin{thm}\label{Norifond} Assume that the extension (\ref{ext}) is non-trivial. Then the following assertions are equivalent.
\begin{enumerate}
\item We have $$w_2(\kappa_{\theta})+w_1(\kappa_{\theta})\cdot w_1(\kappa_{\theta})+w_1(\kappa_{\theta})\cdot w_1(q_{\theta})+ w_2(q_{\theta})=0.$$
\item There exist a rational point $\sigma\in G(k)$ and a commutative diagram
\[ \xymatrix{
   & & \widetilde{G}_{\theta} \ar[d]^{s_{\theta}}  \\
\pi^N_1(X/k,x)\ar[rru]^{r}\ar[r]& G\ar[r]^{c_{\sigma}}& G
} \]
where $r$ is faithfully flat.
\end{enumerate}
\end{thm}
\begin{proof} In order to ease the notations, we fix $\theta$ and we set $\widetilde{G}:=\widetilde{G}_{\theta}$,  $s:=s_{\theta}$, $\kappa:=\kappa_{\theta}$  and $q=q_{\theta}$. We have an exact sequence of pointed sets (see \cite{Giraud} IV.4.2.10) 
$$H^1(X,\widetilde{G})\longrightarrow H^1(X,{G})\stackrel{\delta^2}{\longrightarrow}H^2(X_{\mathrm{et}},\mathbb{Z}/2\mathbb{Z}).$$
The quadratic space $\omega(\mathcal{V},q)$ is the unit form $(A^D,\kappa)$, whereas $\eta(\mathcal{V},q)=(\mathcal{V},q)$. By Theorem \ref{thm-Tannakian}, we have
$$\delta^2(T)=w_2(\kappa)+ w_1(\kappa)\cdot w_1(\kappa)+ w_1(\kappa)\cdot w_1(q)+w_2(q).$$
By the exact sequence above, $\delta^2(T)=0$ iff there exists a $\widetilde{G}$-torsor $\widetilde{T}$ over $X$ and a $\widetilde{G}$-equivariant map $\widetilde{T}\rightarrow T$. 

We show $(2)\Longrightarrow (1)$. Suppose that there exists a commutative diagram as in (2). The map
$$\pi^N_1(X/k,x)\rightarrow G\stackrel{c_{\sigma}}{\rightarrow}G$$
corresponds to the triple $(G,T,t\cdot\sigma)$, where $(t\cdot\sigma)\in T(k)$ is the point
$$\mathrm{Spec}(k)\stackrel{(t,\sigma)}{\longrightarrow} T\times_k G=T\times_XG_X\stackrel{\mu}{\longrightarrow} T.$$
The commutative diagram (2) therefore gives a triple $(\widetilde{G},\widetilde{T},\widetilde{t})$ mapping to $(G,T,t\cdot\sigma)$ hence a $\widetilde{G}$-torsor $\widetilde{T}$ over $X$ and a $\widetilde{G}$-equivariant map $\widetilde{T}\rightarrow T$. We obtain $\delta^2(T)=0$, hence (1).

We show $(1)\Longrightarrow (2)$.  We have $\delta^2(T)=0$ hence there exists a $\widetilde{G}$-torsor $\widetilde{T}$ over $X$ together with a $\widetilde{G}$-equivariant map $\widetilde{T}\rightarrow T$. Taking the base change along $x\in X(k)$, we obtain a 
$\widetilde{G}$-torsor $\widetilde{T}_x:=\widetilde{T}\times_Xx$ over $k$ lifting $T_x$. Note that $T_x$ is the trivial $G$-torsor. The exact sequence of pointed sets
$$H^1(k,\mathbb{Z}/2\mathbb{Z})\longrightarrow H^1(k,\widetilde{G})\longrightarrow H^1(k,G)$$
shows that there exists $\alpha\in H^1(k,\mathbb{Z}/2\mathbb{Z})$ mapping to $\widetilde{T}_x$. Recall from (\cite{Giraud} III.3.4.5) the natural action of $H^1(k,\mathbb{Z}/2\mathbb{Z})$ on $H^1(k,\widetilde{G})$. We have $$[\widetilde{T}_x]=\alpha\cdot  [\widetilde{T}_{0}]$$
where $\widetilde{T}_{0}$ is the trivial $\widetilde{G}$-torsor over $k$.

Let $A$ be the image of $\alpha$ in $H^1(X,\mathbb{Z}/2\mathbb{Z})$. The action (see \cite{Giraud} III.3.4.5) of $H^1(X,\mathbb{Z}/2\mathbb{Z})$ on the set $H^1(X,\widetilde{G})$ is compatible with base-change. Let $\widetilde{T}'$ be a $\widetilde{G}$-torsor over $X$ such that we have 
$[\widetilde{T}']= A\cdot [\widetilde T]$ in $H^1(X,\widetilde{G})$. Taking the pull-back $\widetilde{T}'_x$ of $\widetilde{T}'$ along $x\in X(k)$, we have
$$[\widetilde{T}'_x]=x^*[\widetilde{T}']=x^*(A\cdot [\widetilde{T}])=(x^*A)\cdot (x^*[\widetilde{T}])=\alpha\cdot [\widetilde{T}_x]=\alpha\cdot \alpha\cdot[\widetilde{T}_{0}]=[\widetilde{T}_{0}].$$ 
Therefore, replacing $\widetilde{T}$ with $\widetilde{T}'$, we may suppose that
$\widetilde{T}_x$ is a trivial $\widetilde{G}$-torsor over $k$. Choose a rational point $\widetilde{t}\in \widetilde{T}_x(k)$, and let $t_0$ be its image in $T_x(k)$. Since $T_x$ is a $G$-torsor, there exists a unique $\sigma\in G(k)$ such that
$t_0=t\cdot\sigma$. Hence the triple $(\widetilde{G},\widetilde{T},\widetilde{t})$ maps to $(G,T,t_0)$, which corresponds to the map $\pi^N_1(X/k,x)\rightarrow G\stackrel{c_{\sigma}}{\rightarrow}G.$
We obtain the commutative diagram in (2).

It remains to show that the map $r:\pi^N_1(X/k,x)\rightarrow \widetilde{G}$ in the diagram is faithfully flat. The map $r$ factors as follows $\pi^N_1(X/k,x)\twoheadrightarrow  \pi \rightarrow \widetilde{G}$, where $\pi$ is a finite flat $k$-group scheme and $\pi^N_1(X/k,x)\twoheadrightarrow  \pi$ is faithfully flat. It is therefore enough to show that the induced map $r':\pi \rightarrow \widetilde{G}$ is faithfully flat. The map $r'$ gives a morphism of exact sequences
\[ \xymatrix{
1\ar[r]&N\ar[r]\ar[d]^{\beta} &\pi\ar[r]\ar[d]^{r'} & G\ar[r]\ar[d]^{c_{\sigma}}&1  \\
1\ar[r]&\mathbb{Z}/2\mathbb{Z}\ar[r]& \widetilde{G}\ar[r]^{s}& G\ar[r]&1
} \]
We have an exact sequence (see \cite{CCMT14} Proposition 3.8)
$$H^1(B_G,\mathbb{Z}/2\mathbb{Z})\rightarrow H^1(\mathrm{Spec}(k)_{\mathrm{fppf}},\mathbb{Z}/2\mathbb{Z})\rightarrow \mathrm{Ext}_k(G,\mathbb{Z}/2\mathbb{Z})\rightarrow H^2(B_G,\mathbb{Z}/2\mathbb{Z}).$$
But the map $H^1(B_G,\mathbb{Z}/2\mathbb{Z})\rightarrow H^1(\mathrm{Spec}(k)_{\mathrm{fppf}},\mathbb{Z}/2\mathbb{Z})$ is surjective, since it is induced by the canonical section $\mathrm{Spec}(k)_{\mathrm{fppf}}\rightarrow B_G$ of the structure map $B_G\rightarrow \mathrm{Spec}(k)_{\mathrm{fppf}}$. Hence the canonical map
$\mathrm{Ext}_k(G,\mathbb{Z}/2\mathbb{Z})\hookrightarrow H^2(B_G,\mathbb{Z}/2\mathbb{Z})$ is injective. Similarly, we have an injection
$\mathrm{Ext}_k(\pi,\mathbb{Z}/2\mathbb{Z})\hookrightarrow H^2(B_{\pi},\mathbb{Z}/2\mathbb{Z})$. The class $[E]\in \mathrm{Ext}_k(G,\mathbb{Z}/2\mathbb{Z})$ is non-trivial by assumption, but dies in $\mathrm{Ext}_k(\widetilde{G},\mathbb{Z}/2\mathbb{Z})$, and a fortiori in $\mathrm{Ext}_k(\pi,\mathbb{Z}/2\mathbb{Z})$ (here we denote by $[E]$ the class of the extension denoted by (\ref{ext}) in the statement of the theorem). We consider the exact sequence (see \cite{CCMT14} Corollary 3.13)
$$\rightarrow H^0(B_G,\underline{\mathrm{Hom}}(N,\mathbb{Z}/2\mathbb{Z}))\rightarrow H^2(B_G,\mathbb{Z}/2\mathbb{Z})\rightarrow H^2(B_\pi,\mathbb{Z}/2\mathbb{Z}).$$ Then $[E]$ lies in $\mathrm{Ker}\left(H^2(B_G,\mathbb{Z}/2\mathbb{Z})\rightarrow H^2(B_\pi,\mathbb{Z}/2\mathbb{Z})\right)$, and $$\beta\in \mathrm{Hom}(N,\mathbb{Z}/2\mathbb{Z})^{G(k)}= H^0(B_G,\underline{\mathrm{Hom}}(N,\mathbb{Z}/2\mathbb{Z}))$$
maps to $[E]\in H^2(B_G,\mathbb{Z}/2\mathbb{Z})$, which is non-trivial since $\mathrm{Ext}_k(G,\mathbb{Z}/2\mathbb{Z})\hookrightarrow H^2(B_G,\mathbb{Z}/2\mathbb{Z})$ is injective. Hence $\beta$ is non-trivial too, hence surjective. It follows that $r'$ is an epimorphism in $\mathrm{Spec}(k)_{\mathrm{fppf}}$, i.e. $r'$ is faithfully flat.

\end{proof}

\begin{rem}
We may replace  $\rho: G\rightarrow \mathbf{O}(\kappa)$ with an arbitrary orthogonal representation  $G\rightarrow \mathbf{O}(u)$, and consider Problem \ref{prob} in this more general situation.  This problem is the analogue of Fr\"ohlich's embedding problem (\cite{Fr\"ohlich85}, Section 7) for Nori's fundamental group scheme.  We can prove mutatis mutandis  Theorem \ref {Norifond} in this case.
\end{rem}

\section{Orthogonal $B$-admissible representations}

Various comparison formulas have been obtained,   in different set-ups,   between the Hasse-Witt invariants of a form and the Hasse-Witt invariants of its twists (see  \cite{Serre84},  \cite{Fr\"ohlich85},  \cite{Saito95}). In this section we show that they all can be seen as consequences of Theorem \ref{thm-Tannakian} when applied to a well choosen tannakian category. Fontaine's theory of $B$-admissible representations and regular $G$-rings provides us with a large range of such categories. 
\subsection{$B$-twist of a form}\label{section-Btwist}

Let $k$  be a topological field of characteristic $\neq 2$ and  let $G$ be a topological group.  We denote by $\mathrm{Rep}_k(G)$ the category of continuous finite dimensional $k$-linear representations of $G$, and by 
$\omega: \mathrm{Rep}_k(G)\rightarrow \mathrm{Vec}_k$ be the forgetful functor.  Then $(\mathrm{Rep}_k(G),\omega)$ is a neutralized  $k$-linear tannakian category. We let $B$ be a topological commutative $k$-algebra (recall that the homomorphism $k\rightarrow B$ is continuous), equipped with a continuous action of $G$ by $k$-automorphisms. We set $K=B^G$. For any $k$-representation $V$ of $G$ we set 
$$D_B(V)=(B\otimes_kV)^{G}$$ 
where $G$ acts  diagonally on $B\otimes_k V$.  By scalar extension we obtain a map 
\begin{equation}
\alpha_V: B\otimes_KD_B(V)\rightarrow B\otimes_k V.
\end{equation}  
The map $\alpha_V$ is $B$-linear and $G$-equivariant. Here $G$ acts trivially on $D_B(V)$ and diagonally on $B\otimes_KD_B(V)$. We now recall the definition of a ($k, G$)-regular algebra (\cite{Fontaine00}  Definition 2.8). 
\begin{defn}\label{def-regular} We say that $B$ is \emph{$(k,G)$-regular} if the following conditions hold:
\begin{enumerate}
\item  $B$ is a domain.
\item $B^{G}=Fr(B)^{G}$.
\item If $b\in B$, $b\neq 0$, such that $kb$ is stable by $G$, then $b$ is invertible in $B$. 
\end{enumerate}

\end{defn}
We assume from now on that $B$ is $(k,G)$-regular. This implies that $K$ is a field and that $\alpha_V$ is injective for any $k$-representation of $G$ (see \cite{Fontaine00} Theorem 2.13).

\begin{defn}  
A $k$-representation $V$ of $G$ is said to be $B$-admissible if $\alpha_V$ is an isomorphism. 
\end{defn}
We denote by  $\mbox{Rep}_k^B(G)$  the full subcategory of $\mbox{Rep}_k(G)$ consisting of representations of $G$ which are $B$-admissible. This category is a tannakian sub-category of $\mbox{Rep}_k(G)$ with  fiber functor given by the restriction of $\omega$. Moreover  the restriction of $D_B$ to $\mbox{Rep}_k^B(G)$ defines  a fiber functor  
$\eta: \mbox{Rep}_k^B(G)\rightarrow \mathrm{Vec}_K$  (\cite{Fontaine00} Theorem 2.13). We let  $\mathcal{G}_B$ be the Tannaka dual ${\bf Aut}^\otimes (\omega)$ and  $T_{B}$ be the 
$\mathcal{G}_{B, K}$-torsor ${\bf Isom}^\otimes (\omega_K, \eta)$. There is a continuous morphism $G \rightarrow  \mathcal G_B(k)$ and the functor sending $\rho:\mathcal{G}_B\rightarrow {\bf GL}(V)$  to 
$\bar{\rho}:G\rightarrow \mathcal G_B(k)\rightarrow {\bf GL}_V(k)$ induces an equivalence of tensor categories (\cite{Fontaine00}  Proposition 1.2.3)
$$\mbox{Rep}_k(\mathcal{G}_B)\stackrel{\sim}{\longrightarrow}\mbox{Rep}_k^B(G). $$ We identify these categories.  

An orthogonal object of $\mbox{Rep}_k^B(G)$ is given by a continuous  orthogonal representation $\rho: G\rightarrow {\bf O}(q)(k)$ where $(V,q)$ is a quadratic form over $k$ and ${\bf O}(q)(k)$ is endowed with topology induced by that of $k$. Following the method  of Section 3, we  associate with   $(V, q, \rho)$, via the fiber functors $\omega$ and $\eta$,   non degenerate quadratic forms on $k$ and $K$,  on the one hand the form $(V,q)$ obtained through the forgetful functor $\omega$ and the form $(D_B(V), q_D)$ obtained through  $\eta$, where 
$q_D$ is the restriction to $D_B(V)$ of the form $B\otimes q$.  The quadratic space $(D_B(V), q_{D})$ will be called the {\it  $B$-twist of $(V,q)$}. 

Since $q_D$ is the twist of $q$ by the $\mathcal G_{B, K}$ torsor $T_{B}$ (see \cite{Fontaine94}, definition 2.4), then Theorem 3.4 provides us with a comparison formula for the Hasse-Witt invariants of these forms. More precisely, let $(V, q, \rho)$ be an orthogonal $B$-admissible $k$-representation of $G$ and let $<V>$ be the subtannakian category generated by $V$. Then $\omega$ and $\eta$ induce fiber functors $\omega'$ and $\eta'$  on $<V>$. We  consider the subalgebra $B_V$ of $B$ consisting of all the periods of all the objects of $<V>$ ( \cite{Fontaine94}, Section 1.7)). The algebra $B_V/k$ is a $(k, G)$-regular algebra of finite type. We set $T_{B, V}=\mbox{Spec}(B_V)$. By \cite{Fontaine94}, Theorem 1.7.3 we know that $T_{B, V}$ represents the functor 
${\bf {Isom}}^\otimes (\omega'_K, \eta')$. Therefore Theorem 3.4 expresses the invariants  $\delta^1_q(T_{B})=\delta^1_q(T_{B, V})$ and $\delta^2_q(T_{B})=\delta^2_q(T_{B, V})$ of this algebra  in terms of the Hasse-Witt invariants of the forms $q$ and $q_D$. 

Our aim is now to express $\delta^1_q(T_B)$ and $\delta^2_q(T_B)$ in terms of invariants attached to $\rho$.  
We start by introducing  some notations. As before $B/k$ is a topological algebra and  $K/k$ is  a field extension. We   fix a separable closure $\bar K$ of $K$ and we   take for  group $G$ the profinite group  $G_K=\mbox{Gal}(\bar K/K)$. We  suppose that $B$ is a $(k, G_K)$ regular algebra, such that $B^{G_K}=K$. Moreover we assume that $B/K$ is  a filtered equivariant algebra. In other words   there exists a decreasing filtration by $K$-subvector spaces  $\{B^i, i \in \mathbb  Z\}$, stable under the action of $G_K$,   such that  $K\subset B^0$ and $B^iB^j\subset B^{i+j}$ for $i, j \in \mathbb Z$. If $V$ is a $k$-representation of $G_K$, the filtration on $B$ induces natural filtrations on $B\otimes_k V$, $D_B(V)$  and $B\otimes_K D_B(V)$. More precisely for any integer $i$, we set 
$$(B\otimes_k V)^i=B^i\otimes_k V,\ \ D_B(V)^i=(B^i\otimes_k V )^{G_K},  $$
and we endow  $B\otimes_K D_B(V)$ with the  tensor product filtration given by $(B\otimes_K D_B(V))^i=\sum _{k+l=i}B^k\otimes_K D_B(V)^l$. 

We recall that a Hausdorff topological field $F$ is called t-Henselian if for every etale morphism $X\rightarrow Y$ of $F$-varieties the induced map 
$X(F)\rightarrow Y(F)$ is a local homeomorphism,  where $X(F)$ and $Y(F)$ are endowed with the topology of $F$ ( \cite{Gabber14}  Section 3 and \cite{pp} ). 

\begin{defn} We say that the algebra $B/K$ is well filtered if the following conditions hold:  
\begin{enumerate}
\item There exists a $G_K$-equivariant $\bar K$-structure on $B^0$ extending the $K$-algebra structure. 
\item $C:=\mathrm{gr}^0(B)$ is a   t-Henselian field  when endowed with the quotient topology such that $C^{\times}=C^{\times 2}$. 
\end{enumerate}

\end{defn}
If $B/K$ is filtered  a $B$-admissible representation $V$ of $G_K$ is said to be {\it filtered}  if $\alpha_V$ and $\alpha_V^{-1}$ both respect the filtrations. 

Let $\rho: G_K\rightarrow {\bf O}(q)(k)$ be a continuous,  orthogonal and $B$-admissible representation.  We define  {\it the first Stiefel-Whitney class $sw_1(\rho)$} as $\mbox{det}(\rho) \in 
H^1(G_K, \mathbb Z/2\mathbb Z)$.  We now   consider the central extension of group schemes 
\begin{equation}\label{exact1} 1\rightarrow \mathbb Z/2\mathbb Z\rightarrow \tilde{\bf O}(q_C)\rightarrow {\bf O}(q_C)\rightarrow 1, \end{equation}
where $q_C=q\otimes_k C$. We associate to (\ref{exact1}) an exact sequence of groups 
\begin{equation}\label{exact2}
1\rightarrow \mathbb Z/2\mathbb Z\rightarrow \tilde{\bf O}(q)(C)\rightarrow {\bf O}(q)(C)\rightarrow H^1(\mbox{Spec}(C)_{\mathrm{et}}, \mathbb  Z/2\mathbb  Z).\end{equation}
 Since $B/K$ is well filtered and since the morphism $\tilde{\bf O}(q_C)\rightarrow {\bf O}(q_C)$ is etale we deduce from (\ref{exact2})  an exact sequence of topological groups 
 \begin{equation}\label{exact3}
 1\rightarrow \mathbb Z/2\mathbb Z\rightarrow \tilde{\bf O}(q)(C)\rightarrow {\bf O}(q)(C)\rightarrow 1, 
 \end{equation}
 with the trivial $G_K$-action.   Since the map $k \rightarrow C$ is a continuous and  $G_K$-equivariant  ring homomorphism,  the map $\rho$ defines a $1$-continuous cocycle of $G_K$ with values in $C$ and thus a class $[\rho] \in H^1(G_K, {\bf O}(q)(C))$. We define {\it the second Stiefel-Whitney  class 
$sw_2(\rho)$ of $\rho$}  as the image of $[\rho]$ by the boundary map of continuous cohomology 
$$\delta^2: H^1(G_K,  {\bf O}(q)(C))\rightarrow H^2(G_K,  \mathbb Z/2\mathbb Z).$$

Composing $\rho$ with the spinor norm homomorphism $$sp:{\bf O}(q)(k)\rightarrow k^{\times}/k^{\times 2}\rightarrow K^{\times}/K^{\times 2}=H^1(G_K, \mathbb Z/2\mathbb Z),$$ we obtain a homomorphism  
$sp\circ \rho \in \mbox{Hom}(G_K, H^1(G_K, \mathbb Z/2\mathbb Z))$. {\it The spinor class $sp_2(\rho)$} 
is defined as the image of $sp\circ \rho$ by the map
$$\mbox{Hom}(G_K, H^1(G_K, \mathbb Z/2\mathbb Z))
\simeq H^1(G_K, \mathbb Z/2\mathbb Z)\otimes H^1(G_K, \mathbb Z/2\mathbb Z)\longrightarrow H^2(G_K, \mathbb Z/2\mathbb Z)$$
defined by cup-product.

\begin{thm}\label{thm-Fontaine} Let $B$ be  a $(k, G_K)$-regular  algebra and let $\rho: G_K\rightarrow {\bf O}(q)(k)$ be a continuous,  orthogonal and $B$-admissible representation. 
 Then  we have:
 \begin{enumerate}
 \item $w_1(q_D)=w_1(q)+sw_1(\rho).$
 \item Moreover, if $B/K$ is well filtered and the representation  is filtered then   
 $$w_2(q_D)=w_2(q) +w_1(q)sw_1(\rho)+sw_2(\rho)+sp_2(\rho).$$
 \end{enumerate}
\end{thm}
\begin{proof} It follows from Theorem  \ref{thm-Tannakian} that in order to prove the first equality it suffices  to prove  the equality
$$\delta^1_q(T_B)=sw_1(\rho) $$ 
in $H^1(G_K, \mathbb Z/2\mathbb Z)$. The morphism of group schemes $\rho: \mathcal G_B \rightarrow  {\bf O}(q)$ induces a map $\rho_*: H^1(G_K, \mathcal G_B)\rightarrow H^1(G_K, 
{\bf O}(q))$. The image of $T_B$ by $\rho_*$  is the ${\bf O}(q)$-torsor $\mbox {\bf Isom}(q_K, q_D)$ over $K$ (  [CCMT],  Section 4).  It  defines a class in the continuous cohomology set $H^1(G_K, {\bf O}(q)(\bar K))$, when  ${\bf O}(q)(\bar K)$ is endowed with the discrete topology.  One defines a $1$-cocycle representative of $\rho_*(T_B)$ by choosing a $\bar K$-point of $\mbox {\bf Isom}(q_K, q_D)$, given by an isometry 
$$f: \bar K\otimes_k V\rightarrow \bar K\otimes_KD_B(V), $$
and considering the map $c: G_K\rightarrow {\bf O}(q)(\bar K)$, $g \rightarrow f^{-1}\circ (g\otimes 1)\circ f\circ  (g^{-1}\otimes 1)$.  We have the equality 
\begin {equation} \delta^1_q(T_B)=\mbox {det}([c]),  \end{equation} where $[c]\in H^1(G_K, {\bf O}(q)(\bar K))$ is the class of $c$. 

 Scalar extension by $B$ induces  a $G_K$-equivariant group homomorphism ${\bf O}(q)(\bar K) \rightarrow {\bf O}(q)(B)$ which is continuous, with respect to the discrete topology  on ${\bf O}(q)(\bar K)$ and the topology induced by $B$ on ${\bf O}(q)(B)$. Therefore this group homomorphism induces a map of pointed sets 
$$i_B: H^1(G_K, {\bf O}(q)(\bar K))\rightarrow H^1(G_K, {\bf O}(q)(B)), $$
and the  class $i_B([c])$ is represented by the cocycle $c_B$ associated  with the $B$-point $f_B$ of $\mbox {\bf Isom}(q_K, q_D)$ obtained by  scalar extension from $f$
$$f_B: B\otimes_{ k} V\rightarrow B\otimes_{ K} D_B(V).$$
Since $\rho$ is $B$-admissible we obtained an other $B$-point of $\mbox {\bf Isom}(q, q_D)$ by considering 
$$\alpha_V^{-1}: B\otimes _k V\rightarrow B\otimes _K D_B(V).$$
Let   $c'_B: G_K\rightarrow {\bf O}(q)(B)$ be the map  $(g \rightarrow \alpha_V\circ (g\otimes 1)\circ \alpha_V^{-1}\circ  (g^{-1}\otimes 1)$. 
  We  observe that $\alpha_V$  is    $G_K$-equivariant  when  $G_K$ acts diagonally on $B\otimes V$ and via $B$ on $B\otimes D_B(V)$. Therefore we have
\begin{equation}\label {equiv} \alpha_V\circ (g\otimes 1)\circ \alpha_V^{-1}\circ  (g^{-1}\otimes 1)(b\otimes v)=\alpha_V\circ (g\otimes 1)\circ \alpha_V^{-1}(g^{-1}(b\otimes \rho(g)(v))=b\otimes \rho(g)(v). 
\end{equation} 
Since $k \rightarrow B$ is continuous and preserves $G_K$-action, the group homomorphism $\rho: G_K\rightarrow {\bf O}(q)(k)$ defines a  $1$-cocycle of the continuous cohomology set $H^1(G_K, {\bf O}(q)(B))$ and (\ref {equiv})  shows   the equality $c'_B=\rho$. Moreover since the $1$-cocycles $c_B$ and $c'_B$ are both defined through $B$- points they are cohomologous. Thus  we have proved that 
\begin{equation}\label {fond} i_B([c])=[\rho] \end{equation}  in $H^1(G_K, {\bf O}(q)(B))$. 
Since the group homomorphism $\mbox{det}: {\bf O}(q)(k)\rightarrow \mathbb Z/2\mathbb Z$ factors through ${\bf O}(q)(B)$ we conclude that 
\begin{equation}
\delta^1_q(T_B)=\mbox {det}([c])=\mbox{det}([i_B([c])=\mbox {det}(\rho)=sw_1(\rho)
\end{equation}

We now assume that $B$ is well filtered and the representation $\rho$ is  filtered. It follows from Theorem  \ref{thm-Tannakian} and the previous description of  $\rho_*(T_B)$ that it is sufficient to show the equality 
$$\delta^2([c])=sw_2(\rho)+sp_2(\rho), $$ where $\delta^2$ is the boundary map associated with the exact sequence of groups 
$$1\rightarrow \mathbb Z/2\mathbb Z\rightarrow \tilde{\bf O}(q)(\bar K)\rightarrow {\bf O}(q)(\bar K)\rightarrow 1$$
with the natural $G_K$-action. Moreover it follows from a general result of \cite{Saito95}, Lemma 3,  that 
$sw_2(\rho)+sp_2(\rho)=\delta_C^2(\rho)$,   where $\delta^2_C$  is the boundary map associated with  the exact sequence of topological groups  (\ref{exact3})
$$1\rightarrow \mathbb Z/2\mathbb Z\rightarrow \tilde{\bf O}(q)(C)\rightarrow {\bf O}(q)(C)\rightarrow 1.$$
Therefore, in order to prove (ii) we have to show,  in $H^2(G_K, \mathbb Z/2\mathbb Z)$,  that 
$$\delta^2([c])=\delta^2_C(\rho).$$

We denote by ${\bf O}^{\bullet}(q)(B)$ the group of isometries $B\otimes_k V\rightarrow B\otimes_k V$ which respect the filtration of $B\otimes_k V$. Since $B$ is well filtered we note that $G_K$ acts on ${\bf O}^{\bullet}(q)(B)$. It is now easy to check that since $f_B$ and $\alpha_V$ both preserve the filtration thus  the $1$-cocycles $c_B$ and $c'_B$ have  values in ${\bf O}^{\bullet}(q)(B)$ and define the same class in $H^1(G_K,  {\bf O}^{\bullet}(q)(B))$.  We conclude to the equality  
\begin{equation}\label {fond1} i^{\bullet}_B([c])=[\rho] 
\end{equation}                                                            
in $H^1(G_K,  {\bf O}^{\bullet}(q)(B)))$. Any isometry $ B\otimes_k V \rightarrow B\otimes_k V$ which respects  the filtration induces an isometry 
$\mathrm{gr}^0(B)\otimes V\rightarrow \mathrm{gr}^0(B)\otimes V$. This yields to a  group homomorphism ${\bf O}^{\bullet}(q)(B)\rightarrow {\bf O}(q)(C)$ and thus a map 
$H^1(G_K, {\bf O}^{\bullet}(q)(B))\rightarrow H^1(G_K, {\bf O}(q)(C)$. Using this map we deduce from  (\ref {fond1}) that 
\begin{equation}  \label{fond2} i_C([c])=[\rho] \end {equation} 
in $H^1(G_K, {\bf O}(q)(C)$, where $i_C$ is induced on the cohomology sets  by the group homomorphism ${\bf O}(q)(\bar K)\rightarrow {\bf O}(q)(C)$.  
From (\ref {fond2}) and from the commutativity of the  diagram

\[\xymatrix{\{1\}\ar[r]^{}&
({\mathbb Z}/2{\mathbb Z})\ar[r]^{}\ar[d]_{\textrm{\emph{Id}}}&\widetilde {\bf O}(q)(\bar K)\ar[r]^{}\ar[d]_{\widetilde i_C}\ar[r]^{}& {\bf O}(q)(\bar K)\ar[d]_{i_C}\ar[r]^{} &\{1\}\\
\{1\}\ar[r]^{}&({\mathbb Z}/2{\mathbb Z })\ar[r]^{}&\widetilde {\bf O}(q)(C)\ar[r]^{}& {\bf O}(q)(C)\ar[r]^{}&\{1\} \\
}\]
we deduce that $\delta^2([c])=\delta^2_C(\rho)$ as required

\end{proof}

\subsection{Examples}

We consider several examples of $(k, G_K)$-algebras where we can apply the formulas of Theorem \ref{thm-Fontaine}. 

We introduce a small amount of notations. We set 
$$H^*(K, \mathbb Z/ 2\mathbb Z)=\bigoplus_{0\leq i \leq 2}H^i(G_K, \mathbb Z/2\mathbb Z) $$
endowed by cup product with a structure of abelian group. For a quadratic form $q$ over $K$ we define its Hasse-Witt invariant by  $w(q)=1+w_1(q)+w_2(q) \in H^*(K, \mathbb Z/ 2\mathbb Z)$. For  a virtual quadratic form $q_1-q_2$ over $K$ we write $w(q_1-q_2)=w(q_1)w(q_2)^{-1}$. For a $B$-admissible orthogonal representation 
$\rho$ of $G_K$ we define respectively the Stiefel-Whitney  and the spinor norm classes of $\rho$ by  $sw(\rho)=1+sw_1(\rho) +sw_2(\rho)$  and $sp(\rho)=1 +sp_2(\rho)$
\subsubsection{Fr\"ohlich-Serre formulas}
We let $k$ be a field of characteristic $\neq 2$, $K=k$ and $B=\bar k$ the separable closure of $k$   endowed with the discrete topology. Then $\bar k/k$ is a $(k, G_k)$-algebra, which satisfies the hypotheses of Section 4.1. 
The $\bar k$-admissible orthogonal representations of $G_k$ are the group homomorphisms $\rho: G_k\rightarrow {\bf O}(q)(k)$,  with open kernel, where $q$ is a non degenerate quadratic form on the $k$-vector space $V$  of finite dimension. It is clear that $\bar k$ is a well filtered  when endowed with the filtration given by  $\bar k^0=\bar k$ and $\bar k^1=\{0\}$ and that any orthogonal representation is filtered.  The $\bar k$-twist of $(V, q)$ is the $k$-vector space $(\bar k\otimes_k V)^{G_k}$, with the form obtained by restriction from $\bar k\otimes_k q$.  Theorems 2 and 3 of 
Fr\"ohlich \cite{Fr\"ohlich85} are the equalities of Theorem \ref{thm-Fontaine} in this special case. More precisely:
\begin{cor} \label{FS} Let $\rho: G_k\rightarrow {\bf O}(q)(k)$ be an orthogonal representation with open kernel and let $\underline q$ be the form $(\bar k\otimes q)^{G_k}$. Then 
$$sw(\rho)+sp(\rho)=w(\underline q-q).$$
\end{cor} 

Let $H$ be  an open subgroup  of $G_k$. We   let $V$ be the $k$-vector space with basis the left-cosets  $gH$ of $H$ in $G_k$ and  $q$ be the quadratic form on $V$ which has $\{gH\}$ as an orthonormal basis. The action of $G_k$ by permuting the cosets $gH$ extends to an orthogonal representation $\rho$ of $G_k$. The Hasse-Witt invariants of $q$ are trivial since $q$ is the form $x_1^2+...x_n^2$. The $\bar k$-twist of $q$ is the trace form $\mbox{Tr}_{L/k}$ where $L$ is the fixed field of $H$. Moreover one can compute $sp_2(\rho)$ in this case and prove that it is equal to the cup product $(2)(d_{L/k})$ where $d_{L/k}$ is the relative discriminant of $L/k$. Therefore Corollary \ref{FS} provides us with the equality
 $$sw_2(\rho)+(2)(d_{L/k})=w_2(\mbox{Tr}_{L/k}), $$ which is 
Serre's formula  \cite{Serre84}.
\vskip 0.1 truecm
The references for the next two sections are \cite{Fontaine94} and \cite{Fontaine00}.   Throughout these  sections  $k$ and $ K$   denote complete discrete valuation fields  of characteristic $0$, whose
residue fields  are  perfect of characteristic $p$. We assume that $k \subset K$. We let $C_p$ be the completion of an algebraic closure $\bar K$ of $K$. We note that since $C_p$ is Henselian it is t-Henselian (\cite{Conrad12} Lemma 5-3, or \cite{Gabber14} Proposition 3.1.4). There exist many examples of $(k, G_K)$ regular rings in this set-up. For instance $C_p$ itself is $(k, G_K)$ regular. The $C_p$-admisible representations are those for which the inertia acts through a finite quotient. We can filter $C_p$ by setting $C_p^0=C_p$ and $C_p^1=\{0\}$. Therefore for any $C_p$-admissible orthogonal representation $\rho: G_K
\rightarrow {\bf O}(q)(k)$ we will deduce from Theorem \ref {thm-Fontaine} that 
$$sw(\rho)+sp(\rho)=w(q'-q), $$
where $q'$ is the form $(C_p\otimes_k q)^G_K$. We treat in more details the case of Hodge-Tate and de Rham representations. 
\subsubsection {Orthogonal Hodge-Tate representations}
 We fix a basis $t$ of $\mathbb Z_p(1)$, so that $g(t)=\chi(g)t$  for any $g\in G_K$, where $\chi$ is the cyclotomic character. 
We consider  the ring of Hodge-Tate periods  
$$B_{HT}=C_p[t, t^{-1}]=\bigoplus_{k\in \mathbb Z}C_pt^k.$$
This is a graded $C_p$-algebra endowed with a semilinear action of $G_K$. Moreover $B_{HT}$ is a $(k, G_K)$ regular ring such that $B_{HT}^{G_K}=K$.  We endow $B_{HT}$  with a filtration by setting
$$B_{HT}^i=\bigoplus_{k\geq i}C_pt^k, $$ and thus $\mathrm{gr}^0(B)\simeq C_p$.  One  checks that $B_{HT}$ is a well filtered algebra.  We denote by $(D_{HT}(V), q_{HT})$
 the $B_{HT}$-twist  of  an orthogonal Hodge-Tate representation $(V, q, \rho)$ of $G_K$. 
 \begin{cor}\label{HT} Let $\rho: G_K\rightarrow {\bf O}(q)(k)$ be a filtered  Hodge-Tate $k$-representation of $G_K$. Let $q_K$ be the form $K\otimes q$ and $q_{HT}$ be  the $B_{HT}$-twist of $q$. Then
 $$sw(\rho)+sp(\rho)=w( q_{HT}-q_K).$$ 
 \end{cor}
 Let $MF_K$ be the category of filtered $K$-vector spaces. This is a $K$-linear tensor, addive category, with an identity object and an internal homomorphism. Therefore,  following definition 3.1, we may define an orthogonal object of this category. One easily checks that a quadratic object of $MF_K$ is a filtered quadratic $K$-vector space as defined  in \cite
{Saito95}, Section 1. The fiber functor $D_{HT}: \mbox{Rep}^{dR}_k(G_K)\rightarrow \mathrm{Vec}_k$ factorises through $\mbox{Rep}^{dR}_k(G_K)\rightarrow MF_K$. Hence if $(V,q, \rho)$ is an orthogonal Hodge-Tate representation of $G_K$, then $(D_{dR}(V), q_{dR})$ is a filtered quadratic $K$-vector space. Moreover,  if  the representation is filtered, the comparison isomorphism $\alpha_V$ respects the filtration and so  induces an isometry 
$$\alpha_V^0: \mathrm{gr}^0(B_{HT}\otimes_K D_{HT}(V))\simeq \mathrm{gr}^0(B_{HT}\otimes_k V)$$
which can be written  
$$\alpha_V^0: \bigoplus_i \mathrm{gr}^i(D_{HT}(V))\otimes_K C_p(-i) \simeq C_p \otimes_k V. $$ 
Therefore the orthogonal representations considered  by Saito in \cite{Saito95} are precisely the filtered orthogonal Hodge-Tate $k$-representations of $G_K$ and Corollary \ref{HT}  provides a new proof of \cite{Saito95}, Theorem 1.

\subsubsection{Orthogonal de Rham  representations}\label{section-deRhamTwists}
We consider the ring of periods $B= B_{dR}$  endowed with "its natural topology" (see \cite{Fontaine00}, 5.2.2). Then $B_{dR}/k$ is a $(k, G_K)$-regular ring  such that $B_{dR}^{G_K}=K$. The field  $B_{dR}$ is the fraction field of  the discrete valuation ring $B_{dR}^+$. One can check that $B_{dR}$,  filtered   by the $i$-th powers of the maximal ideal of $B_{dR}^+$,  is well filtered and that any de Rham representation of $G_K$ is filtered \cite{Brinon-Conrad} Proposition 6.3.7.  The ring   $\mathrm{gr}^0(B_{dR})$   is by definition the residue field of $B_{dR}^+$.  It  identifies with 
$C_p$ and the   topology induced on $C_p$ coincides with the  usual topology. For any continuous, orthogonal, de Rham representation $\rho: G_K\rightarrow {\bf O}(q)(k)$, we can define the  $B_{dR}$-twist of $q$, that we denote by $q_{dR}$.  
\begin{cor}\label{dR} Let $\rho: G_K\rightarrow {\bf O}(q)(k)$ be an orthogonal de Rham $k$-representation of $G_K$. Let $q_K=K\otimes_k q$ and let $q_{dR}$ be the $B_{dR}$- twist of $q$. Then 
$$sw(\rho)+sp(\rho)=w(q_{dR}-q_K).$$
\end{cor}

Let ${\bf \mathrm {Grad}}_K$ be the category of graded finite dimensional $K$-vector spaces. The fiber functor $D_{HT}$ factorizes through a fiber  functor $D_{HT}: Rep_k^{HT}(G_K)\rightarrow {\bf \mathrm {Grad}}_K$.  Let $(V, q, \rho)$ be an orthogonal de Rham representation.  Since $D_{dR}(V)$ is filtered we can consider $gr(D_{dR}(V))$ as an object of ${\bf \mathrm {Grad}}_K$ et prove that $gr(D_{dR}(V))$ and $D_{HT}(V)$ are isomorphic in this category. By applying the forgetful functor we obtain an isomorphism of quadratic space
$$ gr(D_{dR}(V))\simeq D_{HT}(V).$$
 
 Let  $X/K$ be  a proper and smooth scheme of even dimension $n$. The cup product defines  a non degenerate quadratic form $q$ on the $\mathbb Q_p$-vector space $V=H^n(X_{\bar {K}}, \mathbb Q_p)(n/2)$. The action of $G_{K}$ on $V$ provides us with  a continuous orthogonal de Rham representation
$\rho: G_{K}\rightarrow {\bf O}(q)(\mathbb Q_p)$ to which the equality of Corollary \ref{dR} applies. It follows from  a theorem of Faltings that  the $B_{dR}$-twist $(D_{dR}(V), q_{dR}) $ of $(V,q)$ coincides in this case, up to isometry,   with the $K$-vector space  $H^n_{dR}(X/K)$, endowed with the quadratic form defined by the cup product.  

\subsubsection{Orthogonal de Rham representation  in the case of  non  perfect residue field}
 We let $K$ be a  complete discrete valuation fields  of characteristic $0$, with residue field $k_K$  of characteristic $p$, such that $[k_K:k_K^p]=p^d$ with $d\geq 0$. We fix $\bar K/K$ an algebraic closure of $K$ and we set $G_K=\mathrm{Gal}(\bar K/K)$. Let $C$ be the completion of $\bar K$ for the $p$-adic topology. In \cite{Brinon1}  O. Brinon has  constructed a ring of periods $B_{dR}$ which is a generalization to the case of imperfect residue fields of the algebra  considered in Section 4.2.3. One should note that $B_{dR}$ is not necessarily a field when $d\geq 1$. This is a toplogical $(\mathbb Q_p, G_K)$-regular algebra such that $B_{dR}^{G_K}=K$.  Therefore for any $B_{dR}$-admissible othogonal representation $\rho: G_K \rightarrow {\bf O}(q)(\mathbb Q_p)$  we have 
$$ w_1(q_{dR})=w_1(q)+sw_1(\rho).$$
$B_{dR}$  has  a structure of $\bar K$-algebra and is endowed with a decreasing, exhaustive and separated filtration, stable under the action of $G_K$. It follows from \cite{Brinon1} Proposition 2.19  that $\mathrm{gr}^0(B_{dR})$ is a ring of polynomials over $C$ in $d$ variables and so that  $H^1(\mathrm{Spec}(C)_{\mathrm{et}}, {\mathbb Z}/2{\mathbb Z})=\{0\}$. Hence we obtain  from (\ref{exact2})  an exact sequence of  groups 
 \begin{equation}
 1\rightarrow \mathbb Z/2\mathbb Z\rightarrow \tilde{\bf O}(q)(C)\rightarrow {\bf O}(q)(C)\rightarrow 1. 
 \end{equation}
However for such a $C$ we don't know if  the map $\tilde{\bf O}(q)(C)\rightarrow {\bf O}(q)(C)$ is a local homeomorphism and thus if $B_{dR}$ is well filtered.  

\section{Orthogonal $E$-motives}
In this section we apply Theorem \ref{thm-Tannakian} to the tannakian category of Nori's mixed motives over a number field (see \cite{Nori} and \cite{Huber-Muller-Stach-12}). Here the fiber functors $\omega$ and $\eta$ are the Betti and the de Rham realization respectively, and the torsor $T_{\omega, \eta}$ is given by the spectrum of the algebra of Kontsevich's formal periods together with its natural action of the motivic Galois group.

\subsection{Nori motives}

\subsubsection{}Let $E/\mathbb{Q}$ be a number field with a given complex embedding $\sigma:E\rightarrow \mathbb{C}$. We denote by $\mathrm{NMM}_E$ the category of cohomological Nori mixed motives over $E$ with $\mathbb{Q}$-coefficients (see \cite{Nori}, \cite{Levine}, \cite{Huber-Muller-Stach-12}). It is a neutral Tannakian category with Tannaka dual $\mathcal{G}_{\mathrm{mot}}(E)$. More precisely, Betti cohomology induces a fiber functor
$$H^*_B:\mathrm{NMM}_E\longrightarrow \mathrm{Vec}_{\mathbb{Q}}$$
and the motivic Galois group $\mathcal{G}_{\mathrm{mot}}(E)$ is defined as the automorphism group-scheme ${\bf Aut}^{\otimes}(H^*_{B})$ of the tensor functor  $H^*_{B}$. By Tannaka duality, an object $M$ is determined by the $\mathbb{Q}$-linear representation
\begin{equation}\label{repres}
\mathcal{G}_{\mathrm{mot}}(E)\longrightarrow {\bf GL}(H^*_B(M))
\end{equation}
where ${\bf GL}(H^*_B(M))$ is seen as a group scheme over $\mathbb{Q}$. De Rham cohomology induces another fiber functor 
$$H^*_{dR}:\mathrm{NMM}_E\longrightarrow \mathrm{Vec}_E.$$
If $X$ is a variety over $E$, $Y\subset X$ a subvariety and $i$ an integer, then there is an object $h^i(X,Y)\in\mathrm{NMM}_E$ such that $H^*_B(h^i(X,Y))=H^i(X_{\sigma}(\mathbb{C}),Y_{\sigma}(\mathbb{C}),\mathbb{Q})$ and $H^*_{dR}(h^i(X,Y))=H^i_{dR}(X,Y)$. Here $X_{\sigma}(\mathbb{C})$ denotes the set of complex points of $X$ over $E$ with respect to the complex embedding $\sigma$. We write $h^i(X):=h^i(X,\emptyset)$ and $h^i(X,Y)(j)=h^i(X,Y)\otimes {\bf 1}(j)$. The category of Artin motives $\mathrm{NMM}^0_E$ is the Tannakian sub-category of $\mathrm{NMM}_E$ generated by objects of the form $h^0(\mathrm{Spec}(F))$ where $F/E$ is a finite extension. Then $\mathrm{NMM}_E^0$ is equivalent to the category of $\mathbb{Q}$-linear representations of the absolute Galois group $G_E$. For any prime number $p$, \'etale $p$-adic cohomology yields a tensor functor
$$H^*_p:\mathrm{NMM}_E\longrightarrow \mathrm{Rep}_{\mathbb{Q}_p}(G_{E})$$
where $\mathrm{Rep}_{\mathbb{Q}_p}(G_{E})$ denotes the Tannakian category of continuous $\mathbb{Q}_p$-linear representations of $G_{E}$. Artin's comparison isomorphism yields an isomorphism of tensor functors $\mathrm{Ou}\circ {H^*_p}\simeq H^*_B\otimes \mathbb{Q}_p$, where $\mathrm{Ou}$ is the forgetful functor $\mathrm{Rep}_{\mathbb{Q}_p}(G_{E})\rightarrow \mathrm{Vec}_{\mathbb{Q}_p}$.  Note also that $p$-adic Hodge theory gives an isomorphism of tensor functor
\begin{equation}\label{p-adiccompa}
\left(\mathrm{Ou}\circ {H^*_p}\right)\otimes_{\mathbb{Q}_p} B_{dR,\lambda}\simeq \left(H^*_{dR}\otimes_E E_{\lambda}\right)\otimes_{E_{\lambda}} B_{dR,\lambda}
\end{equation}
for any prime number $p$ and any finite place $\lambda$ of $E$ lying over $p$, where $B_{dR,\lambda}$ is the field of $p$-adic periods (see Section \ref{section-deRhamTwists}). Moreover, (\ref{p-adiccompa}) is compatible with the natural Galois actions and filtrations. For an archimedean place $\lambda$ of $E$ we have a tensor functor
$$\fonc{H^*_{\lambda}}{\mathrm{NMM}_E}{\mathrm{Rep}_{E_{\lambda}}(G_{E_{\lambda}})}{h^i(X,Y)(j)}{H^i\left(X_{\lambda}(\mathbb{C}),Y_{\lambda}(\mathbb{C}),(2i\pi)^j\mathbb{Q}\right)\otimes_{\mathbb{Q}}E_{\lambda}}$$
where $X_{\lambda}(\mathbb{C})$ is the set of $\mathbb{C}$-points of $X$ over $E$ with respect to the complex embedding $\tilde{\lambda}:E\rightarrow \mathbb{C}$, for any choice of a representative $\tilde{\lambda}$ of $\lambda$.
Finally, the tensor functor $H^*_p$ induces
a continuous morphism
$$G_{E}\longrightarrow \mathcal{G}_{\mathrm{mot}}(E)(\mathbb{Q}_p).$$
Here $\mathcal{G}_{\mathrm{mot}}(\mathbb{Q}_p)$ is endowed with its natural topology (i.e. the projective limit of the $p$-adic topology on the $\mathbb{Q}_p$-points of the algebraic quotients of $\mathcal{G}_{\mathrm{mot}}(\mathbb{Q}_p)$.

\subsubsection{} We refer to \cite{Stalder08} for generalities about scalar extensions of Tannakian categories, and we define the category $\mathrm{NMM}_E(E)$ of Nori mixed motives with $E$-coefficients as the scalar extension
$\mathrm{NMM}_E(E):=\mathrm{NMM}_E\otimes_{\mathbb{Q}}E$. Then $\mathrm{NMM}_E(E)$ is an $E$-linear Tannakian category endowed with a universal $\mathbb{Q}$-linear tensor functor
$$t:\mathrm{NMM}_E\longrightarrow \mathrm{NMM}_E(E).$$ In particular, a $\mathbb{Q}$-linear tensor functor $\mathrm{NMM}_E\rightarrow \mathcal{C}$ into some $E$-linear tensor category $\mathcal{C}$ induces an essentially unique $E$-linear tensor functor $\mathrm{NMM}_E(E)\rightarrow \mathcal{C}$. We still denote by $h^i(X,Y)(j)$ its image in $\mathrm{NMM}_E(E)$.

The Betti realization functor above induces a fiber functor
$$H^*_{B,E}:\mathrm{NMM}_E(E)\longrightarrow \mathrm{Vec}_E$$
and we have $$\mathcal{G}_{E}:={\bf Aut}^{\otimes}(H^*_{B,E})\simeq {\bf Aut}^{\otimes}(H^*_{B})\otimes_{\mathbb{Q}}E=:\mathcal{G}_{\mathrm{mot}}(E)\otimes_{\mathbb{Q}}E.$$
The realization functor $H^*_{dR}$ above induces a fiber functor
$$H^*_{dR,E}:\mathrm{NMM}_E(E)\longrightarrow  \mathrm{Vec}_E.$$
The $p$-adic realization functor induces
$$H^*_{\lambda}:\mathrm{NMM}_E(E)\longrightarrow \mathrm{Rep}_{E_\lambda}(G_{E})$$
for any place $\lambda$ of $E$ lying over $p$. Artin's comparison theorem gives an isomorphism of tensor functors
\begin{equation}\label{Artin}
\mathrm{Ou}\circ H^*_{\lambda}\simeq H_{B,E}^*\otimes_{E}E_{\lambda},
\end{equation}
where $\mathrm{Ou}$ is the forgetful functor $\mathrm{Rep}_{E_\lambda}(G_{E})\rightarrow \mathrm{Vec}_{E_\lambda}$.
For any finite place $\lambda$ of $E$, (\ref{p-adiccompa}) induces by extension of scalars an isomorphism of tensor functors
\begin{equation}\label{p-adic-comparison}
\left(\mathrm{Ou}\circ H^*_{\lambda}\right)\otimes_{E_{\lambda}}B_{dR,\lambda}\simeq H^{*}_{dR,E_{\lambda}}\otimes_{E_{\lambda}}B_{dR,\lambda}
\end{equation}
which is moreover compatible with the natural $G_{E_{\lambda}}$-actions and filtrations, where $H^{*}_{dR,E_{\lambda}}:=H^{*}_{dR,E}\otimes_E E_{\lambda}$. Finally, for any archimedean place $\lambda\mid\infty$ the tensor functor $H^*_{\lambda}$ above induces
$$H^*_{\lambda}:\mathrm{NMM}_E(E)\longrightarrow\mathrm{Rep}_{E_{\lambda}}(G_{E_{\lambda}}).$$

\subsection{The torsor of formal periods}

\subsubsection{} The space of effective formal periods $\mathcal{P}_+$ is the $\mathbb{Q}$-vector space
generated by the symbols $(X,D,\omega,\gamma)$, where $X$ is a variety of dimension $d$ over $\mathbb{Q}$, $D\subset X$ a subvariety, $\omega\in H^d_{dR}(X,D)$ a de Rham cohomology class and $\gamma\in H_d(X(\mathbb{C}),D(\mathbb{C}),\mathbb{Q})$ a singular homology class modulo the following relations: 
\begin{enumerate}
\item (linearity) $(X,D,\omega,\gamma)$ is linear in both $\omega$ and $\gamma$.

\item (change of variables) If $f:(X,D)\rightarrow(X',D')$ is a morphism of pairs over $\mathbb{Q}$, $\gamma\in H_d(X(\mathbb{C}),D(\mathbb{C}),\mathbb{Q})$ and $\omega'\in H^d_{dR}(X',D')$ then
$$(X,D,f^*\omega',\gamma)= (X',D',\omega',f_*\gamma)$$.

\item (Stokes formula) For every triple $Z\subset Y\subset X$ we have
$$(Y,Z,\omega,\delta\gamma)= (X,Y,d\omega,\gamma)$$
where $\delta:H_d(X(\mathbb{C}),Y (\mathbb{C}),\mathbb{Q})\rightarrow H_{d-1}(Y(\mathbb{C}),Z(\mathbb{C}),\mathbb{Q})$ is the boundary operator and $d:H_{dR}^{d-1}(Y,Z)\rightarrow H_{dR}^{d}(X,Y)$ is the differential.

\end{enumerate}
Then $\mathcal{P}_+$ is turned into a $\mathbb{Q}$-algebra via
$$(X,D,\omega,\gamma)\cdot (X',D',\omega',\gamma'):=(X\times X', X\times D'\cup D\times X',\omega\wedge\omega',\gamma\times\gamma').$$
The space of formal periods $\mathcal{P}:=\mathcal{P}_+[(\mathbb{G}_m,\{1\},dT/T,{\bf S}^1)^{-1}]$ is defined as a localization of $\mathcal{P}_+$, where ${\bf S}^1\subset \mathbb{G}_m(\mathbb{C})$ is the unit circle.
Then
$$
\appl{ H^d_{dR}(X,D)\times H_d(X(\mathbb{C}),D(\mathbb{C}),\mathbb{Q})}{\mathbb{C}}{(\omega,\gamma)}{\int_{\gamma}\omega}
$$
induces a map of $\mathbb{Q}$-algebras
$$\mathrm{ev}:\mathcal{P}\rightarrow \mathbb{C}.$$
\begin{defn}
The image of $\mathrm{ev}$ is the space $\mathbb{P}\subset\mathbb{C}$ of Kontsevich-Zagier periods. 
\end{defn}
Below is Grothendieck's period conjecture in the formulation of  Kontsevich-Zagier, see (\cite{Kontsevich-Zagier01} Conjecture 1) and (\cite{Kontsevich-Zagier01} Conjecture 4.1). 
\begin{conj}(Kontsevich-Zagier) 
The map $\mathrm{ev}:\mathcal{P}\rightarrow \mathbb{C}$ is injective.
\end{conj}

We denote by $\mathcal{P}_{\mathrm{alg}}$ the sub-algebra of $\mathcal{P}$ generated by symbols $(X,D,\omega,\gamma)$ where $\mathrm{dim}(X)=0$. Then the map $\mathrm{ev}$ induces an isomorphism $\mathcal{P}_{\mathrm{alg}}\stackrel{\sim}{\rightarrow}\overline{\mathbb{Q}}\subset\mathbb{C}$, where $\overline{\mathbb{Q}}\subset\mathbb{C}$ denotes the set of complex numbers which are algebraic over $\mathbb{Q}$. Let $z$ be such a complex number. Let $m_z(T)$ be its minimal polynomial over $\mathbb{Q}$ and let $F=\mathbb{Q}[T]/(m_z(T))$ be the corresponding number field. We set $X=\mathrm{Spec}(F)$. Then $\bar{T}\in F$ is a class in $H_{dR}^0(X)=F$. The complex number $z\in\mathbb{C}$ gives a point in $X(\mathbb{C})=\mathrm{Hom}_{\mathbb{Q}}(F,\mathbb{C})$, i.e. a $0$-cycle $\gamma_z$ in $X(\mathbb{C})$, such that
$\mathrm{ev}(\bar{T},\gamma_z)=z$.

\subsubsection{} We set $\mathfrak{P}:=\mathrm{Spec}(\mathcal{P})$. Then $\mathfrak{P}$ carries a natural action of $\mathcal{G}_{\mathrm{mot}}(\mathbb{Q})$. By (\cite{Huber-Muller-Stach-12} Theorem 2.10), there is an isomorphism of $\mathcal{G}_{\mathrm{mot}}(\mathbb{Q})$-torsors
\begin{equation}\label{isotors}
\mathfrak{P}:=\mathrm{Spec}(\mathcal{P})\simeq {\bf Isom}^{\otimes}(H^*_B,H^*_{dR}).
\end{equation}
Here ${\bf Isom}^{\otimes}(H^*_B,H^*_{dR})$ is defined via its functor of points: For any commutative $\mathbb{Q}$-algebra $R$ we have
$${\bf Isom}^{\otimes}(H^*_B,H^*_{dR})(R)=\mathrm{Isom}^{\otimes}(H^*_B\otimes_{\mathbb{Q}}R,H^*_{dR}\otimes_{\mathbb{Q}}R)$$
where the tensor functor $H^*_B\otimes_{\mathbb{Q}}R$ (resp. $H^*_{dR}\otimes_{\mathbb{Q}}R$) is the composition of $H^*_B$ (resp. $H^*_{dR}$) with $(-)\otimes_{\mathbb{Q}}R$, and $\mathrm{Isom}^{\otimes}(-,-)$ denotes the set of isomorphisms of tensor functors.
Let $E/\mathbb{Q}$ be a number field with a complex embedding $\sigma$. There is a unique map $\mathfrak{P}\rightarrow \mathrm{Spec}(E)$ such that $E\rightarrow \mathcal{P}\stackrel{\mathrm{ev}}{\rightarrow}\mathbb{C}$ coincides with $\sigma$. Considering $\mathcal{G}_{\mathrm{mot}}(E)$ as a subgroup of $\mathcal{G}_{\mathrm{mot}}(\mathbb{Q})$ leads to an exact sequence of pointed sets
$$1\rightarrow\mathcal{G}_{\mathrm{mot}}(E)\rightarrow \mathcal{G}_{\mathrm{mot}}(\mathbb{Q})\rightarrow\mathrm{Hom}_{\mathbb{Q}}(E,\mathbb{C})\rightarrow *$$
where $\mathrm{Hom}_{\mathbb{Q}}(E,\mathbb{C})$ is pointed by $\sigma$. Restricting the group of operators, $\mathfrak{P}$ becomes a $\left(\mathcal{G}_{\mathrm{mot}}(E)\otimes_\mathbb{Q}E\right)$-torsor over  $\mathrm{Spec}(E)$. Moreover, composition with the functor $t:\mathrm{NMM}_E\rightarrow \mathrm{NMM}_E(E)$ yields an equivalence from the category of $R$-valued fiber functors on $\mathrm{NMM}_E(E)$ to the category of $R$-valued fiber functors on $\mathrm{NMM}_E$. In particular, we have $\mathcal{G}_E\simeq\mathcal{G}_{\mathrm{mot}}(E)\otimes_\mathbb{Q}E$ and the identification (\ref{isotors}) induces an isomorphism of $\mathcal{G}_E$-torsors
$$\mathfrak{P}\simeq {\bf Isom}^{\otimes}(H^*_{B,E},H^*_{dR,E})$$
over $\mathrm{Spec}(E)$.

If $M$ is an object of $\mathrm{NMM}_E(E)$, we denote by  $\langle M\rangle$ 
the Tannakian subcategory of $\mathrm{NMM}_E(E)$ generated by $M$, i.e. the smallest strictly full $E$-linear  abelian subcategory of $\mathrm{NMM}_E(E)$ containing $M$, which is stable by tensor products, sums, duals and sub-quotients. 
We denote by $H^*_{B,E\mid <M>}$ and $H^*_{dR,E\mid <M>}$ the restriction of the fiber functors $H^*_{B,E}$ and $H^*_{dR,E}$ to $\langle M\rangle$. We set $\mathcal{G}_{M/E}:={\bf Aut}^{\otimes}(H^*_{B,E\mid <M>})$ and we consider the $\mathcal{G}_{M/E}$-torsor ${\bf Isom}^{\otimes}(H^*_{B,E\mid <M>},H^*_{dR,E\mid <M>})$. Note that the sheaf ${\bf Isom}^{\otimes}(H^*_{B,E\mid <M>},H^*_{dR,E\mid <M>})$ on $\mathrm{Spec}(E)_{\mathrm{fppf}}$ is represented by an affine scheme. The inclusion $\langle M\rangle\subset\mathrm{NMM}_E(E)$ induces a faithfully flat map
$\mathcal{G}_{E}\rightarrow \mathcal{G}_{M/E}$
and there is a canonical isomorphism of $\mathcal{G}_{M/E}$-torsors
$$ {\bf Isom}^{\otimes}(H^*_{B,E},H^*_{dR,E})\wedge^{\mathcal{G}_{E}} \mathcal{G}_{M/E}\simeq {\bf Isom}^{\otimes}(H^*_{B,E\mid <M>},H^*_{dR,E\mid <M>}).$$
We consider the $E$-sub-algebra $\mathcal{P}_{M}\subset \mathcal{P}$ such that (\ref{isotors}) induces an isomorphism of $\mathcal{G}_{M/E}$-torsors
\begin{equation}\label{isotorsM}
\mathfrak{P}_{M}:=\mathrm{Spec}(\mathcal{P}_{M})\simeq {\bf Isom}^{\otimes}(H^*_{B,E\mid <M>},H^*_{dR,E\mid <M>}).
\end{equation}
Finally, we consider the restriction of the evaluation map
$$\mathrm{ev}_{\mid\mathcal{P}_{M}}:\mathcal{P}_{M}\hookrightarrow\mathcal{P}\rightarrow \mathbb{C},$$
and we the denote by
$$\mathbb{P}_{M}:=\mathrm{Im}(\mathrm{ev}_{\mid\mathcal{P}_{M}})$$
the $E$-algebra of periods of $M$.

\subsection{Invariants of orthogonal $E$-motives} An \emph{orthogonal Nori motive $(M,q)$ over $E$ with $E$-coefficients}, or simply an orthogonal $E$-motive, is an orthogonal object of $\mathrm{NMM}_E(E)$, i.e. an object $M$ endowed with a symmetric map
$q:M\otimes M\rightarrow {\bf 1}$
inducing an isomorphism $M\stackrel{\sim}{\rightarrow} M^{\vee}$. Let $(M,q)$ be such a motive.  

\subsubsection{The Hasse-Witt invariants $w_i(q_B)$ and $w_i(q_{dR})$}

Recall that $M_B:=H^*_{B,E}(M)$ and $M_{dR}:=H^*_{dR,E}(M)$ denote the Betti and de Rham realizations of $M$. We denote by $$q_B:M_B\otimes M_B\rightarrow E\hspace{0.5cm} \mbox{and} \hspace{0.5cm} q_{dR}: M_{dR}\otimes M_{dR}\rightarrow E$$ 
the non-degenerate symmetric $E$-bilinear forms obtained by applying the tensor functors $H_{B,E}^*$ and $H^*_{dR,E}$ to the map $q:M\otimes M\rightarrow {\bf 1}$. We shall consider the Hasse-Witt invariants $w_i(q_B)\in H^i(G_{E},\mathbb{Z}/2\mathbb{Z})$ and  $w_i(q_{dR})\in H^i(G_{E},\mathbb{Z}/2\mathbb{Z})$ of these quadratic forms.

\subsubsection{The global Stiefel-Whitney invariants  $sw_i(\rho_{\lambda\mid G_{E_{\lambda}}})$} Let $\lambda\nmid\infty$ be a finite prime of $E$. Recall that we denote $M_{\lambda}:=H^*_{\lambda}(M)$. Applying the tensor functor $H^*_{\lambda}$ to the orthogonal motive $(M,q)$ we obtain a quadratic form
$$q_{\lambda}:M_{\lambda}\otimes M_{\lambda}\longrightarrow E_{\lambda}$$
and a continuous $E_{\lambda}$-linear orthogonal representation of the Galois group $G_{E}$ on $(M_{\lambda},q_{\lambda})$, which we denote by
$$\rho_{\lambda}:G_{E}\longrightarrow {\bf O}(q_{\lambda})(E_{\lambda})$$
where ${\bf O}(q_{\lambda})(E_{\lambda})\subset GL_{E_{\lambda}}(H^*_{\lambda}(M))$ is endowed with the $\lambda$-adic topology. If one forgets the Galois action, we have an isometry $(M_{\lambda},q_{\lambda})\simeq (M_B,q_B)\otimes _EE_{\lambda}$. Composing $\rho_{\lambda}$ with the determinant map, we obtain a continuous morphism
$$\mathrm{det}(\rho_{\lambda}):G_{E}\longrightarrow {\bf O}(q_{\lambda})(E_{\lambda})\longrightarrow \mathbb{Z}/2\mathbb{Z}.$$
The first Stiefel-Whitney invariant $sw_1(\rho_{\lambda})\in H^1(G_{E},\mathbb{Z}/2\mathbb{Z})$ is the cohomology class of  $\mathrm{det}(\rho_{\lambda})$. 

Let $\mathbb{C}_{\lambda}$ be the completion of $\overline{E}_{\lambda}$, given with its natural topology. Consider the continuous representation
$$\overline{\rho}_{\lambda}:G_{E}\stackrel{\rho_{\lambda}}{\longrightarrow} {\bf O}(q_{\lambda})(E_{\lambda})\longrightarrow {\bf O}(q_{\lambda})(\mathbb{C}_{\lambda}).$$
We have a central extension of topological groups
\begin{equation}\label{extCOC}1\longrightarrow \mathbb{Z}/2\mathbb{Z}\longrightarrow \widetilde{\bf O}(q_{\lambda})(\mathbb{C}_{\lambda}) \longrightarrow {\bf O}(q_{\lambda})(\mathbb{C}_{\lambda})\longrightarrow 1
\end{equation}
inducing a central extension in the topos $\widetilde{\mathrm{Top}}$ of sheaves on the site of all topological spaces endowed with the open cover topology (see \cite{Giraud}VIII.8.1). Pulling-back (\ref{extCOC}) via $\overline{\rho}_{\lambda}$, we obtain
central extension of topological groups 
\begin{equation}\label{extCOG}1\longrightarrow \mathbb{Z}/2\mathbb{Z}\longrightarrow \widetilde{G_{E}} \longrightarrow G_{E}\longrightarrow 1
\end{equation}
inducing a central extension in $\widetilde{\mathrm{Top}}$. The second Stiefel-Whitney invariant $sw_2(\rho_{\lambda})\in H^2(G_{E},\mathbb{Z}/2\mathbb{Z})$ is the cohomology class (see \cite{Giraud}VIII.8.2) of the central extension (\ref{extCOG}).

\subsubsection{The local Stiefel-Whitney invariants $sw_i(\rho_{\lambda\mid G_{E_{\lambda}}})$} For a finite prime $\lambda\nmid\infty$, we choose an $E$-embedding $\overline{E}\rightarrow \overline{E_{\lambda}}$ and we consider the restriction of $\rho_{\lambda}$ to $G_{E_{\lambda}}$, i.e. the representation
$$\rho_{\lambda\mid G_{E_{\lambda}}}:G_{E_{\lambda}}\longrightarrow G_{E}\longrightarrow {\bf O}(q_{\lambda})(E_{\lambda})$$
and we define $sw_i(\rho_{\lambda\mid G_{E_{\lambda}}})\in H^i(G_{E_{\lambda}},\mathbb{Z}/2\mathbb{Z})$ as above. Then $sw_i(\rho_{\lambda\mid G_{E_{\lambda}}})$ is the image of $sw_i(\rho_{\lambda})$ under the restriction map $H^i(G_{E},\mathbb{Z}/2\mathbb{Z})\rightarrow H^i(G_{E_{\lambda}},\mathbb{Z}/2\mathbb{Z})$. For an archimedean prime $\lambda\mid\infty$, we denote by $M_{\lambda}:=H^*_{\lambda}(M)$ and by
$$q_{\lambda}: M_{\lambda}\otimes M_{\lambda}\longrightarrow E_{\lambda}$$ 
the quadratic form over $E_{\lambda}$ obtained by applying the tensor functors $H_{\lambda}^*$ to the map $q$. If $\lambda$ is a real place, then we obtain an orthogonal representation $$\rho_{\lambda\mid G_{E_{\lambda}}}: G_{E_{\lambda}}\longrightarrow \mathbf{O}(q_{\lambda})(E_{\lambda})$$ 
and $sw_i(\rho_{\lambda\mid G_{E_{\lambda}}})\in H^i(G_{E_{\lambda}},\mathbb{Z}/2\mathbb{Z})$ is its classical Stiefel-Whitney invariant.

\subsubsection{The local spinor class $sp_2(\rho_{\lambda\mid G_{E_{\lambda}}})$} Let $\lambda$ be a place of $E$.
The spinor norm is the map 
$$sp:\mathbf{O}(q_{\lambda})(E_{\lambda})=H^0(G_{E_{\lambda}},\mathbf{O}(q_{\lambda})(\overline{E_{\lambda}}))\longrightarrow H^1(G_{E_{\lambda}},\mathbb{Z}/2\mathbb{Z})=E_{\lambda}^{\times}/E_{\lambda}^{\times 2}$$
induced by the exact sequence of group-schemes
$$1\rightarrow\mathbb{Z}/2\mathbb{Z}\rightarrow \widetilde{\mathbf{O}}(q_{\lambda})\rightarrow  \mathbf{O}(q_{\lambda})\rightarrow 1.$$
Composing $\rho_{\lambda\mid G_{E_{\lambda}}}$ with the spinor norm we obtain an element $sp\circ\rho_{\lambda\mid G_{E_{\lambda}}}$ in the group
$$\mbox{Hom}(G_{E_{\lambda}}, E_{\lambda}^{\times}/E_{\lambda}^{\times 2})=\mbox{Hom}(G_{E_{\lambda}}, H^1(G_{E_{\lambda}}, \mathbb{Z}/2\mathbb{Z}))\simeq H^1(G_{E_{\lambda}}, \mathbb{Z}/2\mathbb{Z})\otimes H^1(G_{E_{\lambda}}, \mathbb{Z}/2\mathbb{Z}).$$
The spinor class $sp_2(\rho_{\lambda\mid G_{E_{\lambda}}})$ is defined as the image of $sp\circ\rho_{\lambda\mid G_{E_{\lambda}}}$ by the cup product 
$$H^1(G_{E_{\lambda}}, \mathbb{Z}/2\mathbb{Z})\otimes H^1(G_{E_{\lambda}}, \mathbb{Z}/2\mathbb{Z})\stackrel{\cup}{\longrightarrow} H^2(G_{E_{\lambda}}, \mathbb{Z}/2\mathbb{Z}).$$

\subsubsection{The classes $\delta_{q}^1(\mathfrak{P}_M)$ and  $\delta_{q}^2(\mathfrak{P}_M)$.}

As in section \ref{sectDelta}, an orthogonal Nori motive $(M,q)$ over $E$ with $E$-coefficients yields an orthogonal representation
$$\rho_{q}:\mathcal{G}_{E}\longrightarrow {\bf O}(q_B)$$
inducing maps
$$\delta_{q}^i:H^1(\mathrm{Spec}(E)_{\mathrm{fppf}},\mathcal{G}_{E})\longrightarrow H^i(G_{E},\mathbb{Z}/2\mathbb{Z})$$
for $i=1,2$. The protorsor $\mathfrak{P}$ gives a canonical class in $H^1(\mathrm{Spec}(E)_{\mathrm{fppf}},\mathcal{G}_{E})$ and we shall consider its image $\delta_{q}^i(\mathfrak{P})$. The orthogonal representation $\rho_q$ factors through 
$$\mathcal{G}_{M/E}\longrightarrow {\bf O}(q_B)$$
hence gives 
$$\delta_{q}^i:H^1(\mathrm{Spec}(E)_{\mathrm{fppf}},\mathcal{G}_{M/E})\longrightarrow H^i(G_{E},\mathbb{Z}/2\mathbb{Z}).$$
By definition, we have 
$$\delta_{q}^i(\mathfrak{P})=\delta_{q}^i(\mathfrak{P}_M) \in H^i(G_{E},\mathbb{Z}/2\mathbb{Z}).$$

\subsection{Interpretation of $\delta_{q}^1(\mathfrak{P}_M)$ and $\delta_{q}^2(\mathfrak{P}_M)$}

\subsubsection{The class $\delta_{q}^1(\mathfrak{P}_M)$ and the periods of $\mathrm{det}_E(M)$}

Let $M\in\mathrm{NMM}_E(E)$ be a motive given by the representation
$\mathcal{G}_{E}\rightarrow \mathbf{GL}(M_B)$. The motive $\mathrm{det}_E(M)$ is the motive associated with the determinant representation $$\mathcal{G}_{E}\longrightarrow \mathbf{GL}(\Lambda^{\mathrm{dim}_E(M_B)}_{E}M_B)$$ 
Using the canonical isomorphism $\mathbf{GL}(\Lambda^{\mathrm{dim}_E(M_B)}_{E}M_B)\simeq \mathbf{G}_m$, where $\mathbf{G}_m$ is the multiplicative group over $E$, we see that 
$\mathrm{det}_E(M)$ is the motive corresponding to
\begin{equation}\label{here}
\mathcal{G}_{E}\longrightarrow \mathbf{GL}(M_B)\stackrel{\mathrm{det}}{\longrightarrow}\mathbf{G}_m.
\end{equation}
An orthogonal structure $q$ on $M$ induces an orthogonal structure $\mathrm{det}(q)$ on $\mathrm{det}_E(M)$. Recall that the class $\delta_q^1(\mathfrak{P}_M)$ belongs to the  group
$H^1(\mathrm{Spec}(E),\mathbb{Z}/2\mathbb{Z})\simeq E^{\times}/E^{\times 2}$. 
\begin{prop}\label{propdelta1}
Let $(M,q)$ be an orthogonal motive. 
\begin{enumerate}
\item We have
$$\mathbb{P}_{\mathrm{det}_E(M)}=E\left(\sqrt{\delta_q^1(\mathfrak{P}_M)}\right).$$ 
\item If $\delta_q^1(\mathfrak{P}_M)$ is non-trivial, then we have $\mathcal{G}_{\mathrm{det}_E(M)/E}\simeq \mathbb{Z}/2\mathbb{Z}$ and 
an isomorphism
$$\mathfrak{P}_{\mathrm{det}_E(M)}\simeq \mathrm{Spec}\left(E\left(\sqrt{\delta_q^1(\mathfrak{P}_M)}\right)\right)$$
of $\mathbb{Z}/2\mathbb{Z}$-torsor over $\mathrm{Spec}(E)$.
\end{enumerate}
\end{prop}
\begin{proof}
Consider the composite map
\begin{equation}\label{herela}
\mathcal{G}_{E}\longrightarrow \mathbf{O}(q_B)\stackrel{\mathrm{det}}{\longrightarrow}\mathbb{Z}/2\mathbb{Z}\simeq \mathbf{O}(\mathrm{det}(q)_B).
\end{equation}
Here $\mathbb{Z}/2\mathbb{Z}\simeq \mathbf{O}(\mathrm{det}(q)_B)$ is the orthogonal group of  the Betti realization of $(\mathrm{det}_E(M),\mathrm{det}(q))$. The motivic Galois group $\mathcal{G}_{\mathrm{det}_E(M)/E}$ comes with a monomorphism
$$\mathcal{G}_{\mathrm{det}_E(M)/E}\hookrightarrow \mathbb{Z}/2\mathbb{Z}$$
so that $\mathcal{G}_{\mathrm{det}_E(M)/E}$ is either trivial or isomorphic to $\mathbb{Z}/2\mathbb{Z}$. 
The canonical morphism $\mathcal{G}_{M/E}\rightarrow\mathcal{G}_{\mathrm{det}_E(M)/E}$ induces  $H^1(E,\mathcal{G}_{M/E})\rightarrow H^1(E,\mathcal{G}_{\mathrm{det}_E(M)/E})$, which maps $\mathfrak{P}_M$ to $\mathfrak{P}_{\mathrm{det}_E(M)}$.
Moreover, $\delta_q^1$ is the map 
$H^1(E,\mathcal{G}_{M/E})\rightarrow H^1(E, \mathbb{Z}/2\mathbb{Z})$ induced by
$$\mathcal{G}_{M/E}\rightarrow \mathcal{G}_{\mathrm{det}_E(M)/E}\hookrightarrow \mathbb{Z}/2\mathbb{Z}.$$
In particular, if $\mathcal{G}_{\mathrm{det}_E(M)/E}$ is trivial then $\delta_q^1(\mathfrak{P}_M)$ is trivial, $\mathfrak{P}_{\mathrm{det}_E(M)}$ is trivial and $\mathbb{P}_{\mathrm{det}_E(M)}=E$. So one may assume that $\mathcal{G}_{\mathrm{det}_E(M)/E}\simeq \mathbb{Z}/2\mathbb{Z}$. In this case we have $\delta_q^1(\mathfrak{P}_M)= \mathfrak{P}_{\mathrm{det}_E(M)}$. If  
$\mathfrak{P}_{\mathrm{det}_E(M)}$ is the trivial $\mathbb{Z}/2\mathbb{Z}$-torsor, then $\mathbb{P}_{\mathrm{det}_E(M)}=E$
 and $\delta_q^1(\mathfrak{P}_M)$ is trivial. So one may assume that  the $\mathbb{Z}/2\mathbb{Z}$-torsor $\mathfrak{P}_{\mathrm{det}_E(M)}$ is non-trivial. In this case we have $\mathfrak{P}_{\mathrm{det}_E(M)}=\mathrm{Spec}\left(E\left(\sqrt{\delta_q^1(\mathfrak{P}_M)}\right)\right)$ and the evaluation map $\mathcal{P}_{\mathrm{det}_E(M)}\rightarrow \mathbb{P}_{\mathrm{det}_E(M)}$ is an isomorphism.

\end{proof}

\subsubsection{The class $\delta_{q}^2(\mathfrak{P}_M)$ and the lifting problem $\mathrm{Lift}(M,q)$.}\label{sectDelta2}

\begin{defn}
Let $\widetilde{\mathcal{G}}\rightarrow \mathcal{G}$ be a map of group schemes and let $T$ be a $\mathcal{G}$-torsor. We say that \emph{$T$ admits a $\widetilde{\mathcal{G}}$-lifting} if there exist a $\widetilde{\mathcal{G}}$-torsor $\widetilde{T}$ and an isomorphism of $\mathcal{G}$-torsors $\widetilde{T}\wedge^{\widetilde{\mathcal{G}}}\mathcal{G}\simeq T$.
\end{defn}
Let $(M,q)$ be an orthogonal motive  given by the representation 
$$\mathcal{G}_{E}\rightarrow \mathcal{G}_{M/E}\rightarrow {\bf O}(q_B).$$
Pulling back the canonical central extension 
$$1\rightarrow \mathbb{Z}/2\mathbb{Z}\rightarrow \widetilde{{\bf O}}(q_B) \rightarrow {\bf O}(q_B)\rightarrow 1$$
we obtain a central extension 
\begin{equation}\label{centrale}
1\rightarrow \mathbb{Z}/2\mathbb{Z}\rightarrow \widetilde{\mathcal{G}}_{M/E} \rightarrow \mathcal{G}_{M/E}\rightarrow 1.
\end{equation} 
We shall also consider its base change 
\begin{equation}\label{centralep}
1\rightarrow \mathbb{Z}/2\mathbb{Z}\rightarrow \widetilde{\mathcal{G}}_{M/E}\otimes_{E}E_{\lambda} \rightarrow \mathcal{G}_{M/E}\otimes_{E}E_{\lambda}\rightarrow 1.
\end{equation} 
We would like to address the following lifting problem for torsors.
\begin{notation}
We say that $\mathrm{Lift}(M,q)$ has a solution if the $\mathcal{G}_{M/E}$-torsor $\mathfrak{P}_M$ admits a  $\widetilde{\mathcal{G}}_{M/E}$-lifting.
We say that $\mathrm{Lift}(M,q)$ has a solution over $E_{\lambda}$ if the $\left(\mathcal{G}_{M/E}\otimes_{E}E_{\lambda}\right)$-torsor $\mathfrak{P}_M \otimes_{E}E_{\lambda}$ admits a  $\left(\widetilde{\mathcal{G}}_{M/E}\otimes_{E}E_{\lambda}\right)$-lifting.
\end{notation}
We identify $H^*(\mathrm{Spec}(E)_{\mathrm{fppf}},\mathbb{Z}/2\mathbb{Z})$ with Galois cohomology $H^*(G_E,\mathbb{Z}/2\mathbb{Z})$. Recall that the class $\delta_q^2(\mathfrak{P}_M)$ belongs to the  groups $H^2(G_E,\mathbb{Z}/2\mathbb{Z})$ which is described  by class field theory: One has the exact sequence
\begin{equation}\label{ESCFT}
0\longrightarrow H^2(G_E,\mathbb{Z}/2\mathbb{Z})\longrightarrow \bigoplus_{\lambda}H^2(G_{E_{\lambda}},\mathbb{Z}/2\mathbb{Z})\stackrel{\sum}\longrightarrow \mathbb{Z}/2\mathbb{Z}\longrightarrow0
\end{equation}
and $H^2(G_{E_{\lambda}},\mathbb{Z}/2\mathbb{Z})\simeq\mathbb{Z}/2\mathbb{Z}$ for $\lambda$ finite and for $\lambda$ real. Finally, we denote by
$$\appl{H^*(G_E,\mathbb{Z}/2\mathbb{Z})}{H^*(G_{E_{\lambda}},\mathbb{Z}/2\mathbb{Z})}{x}{x_{\lambda}}$$
the canonical restriction map, for any place $\lambda$ of $E$.

\begin{prop}\label{propdelta2}
Let $(M,q)$ be an orthogonal motive.
\begin{enumerate}
\item  The following assertions are equivalent.
\begin{enumerate}
\item $\delta^2_q(\mathfrak{P}_M)=0$.
\item The problem $\mathrm{Lift}(M,q)$ has a solution.
\item The problem $\mathrm{Lift}(M,q)$ has a solution over $E_{\lambda}$ for any place $\lambda$ of $E$.
\end{enumerate}
\item Let $\lambda$ be a finite place. The problem $\mathrm{Lift}(M,q)$ has a solution over $E_{\lambda}$ if and only if $\delta^2_q(\mathfrak{P}_M)_{\lambda}=0$. 
\end{enumerate}
\end{prop}
\begin{proof} The exact sequence (\ref{centrale}) induces an exact sequence of pointed sets
$$H^1(\mathrm{Spec}(E)_{\mathrm{fppf}},\widetilde{\mathcal{G}}_{M/E}) \longrightarrow H^1(\mathrm{Spec}(E)_{\mathrm{fppf}},\mathcal{G}_{M/E})\stackrel{\delta^2_q}{\longrightarrow} H^2(\mathrm{Spec}(E)_{\mathrm{fppf}},\mathbb{Z}/2\mathbb{Z})$$
and similarly for (\ref{centralep}). Assertions (1) and (2) follow from these exact sequences of pointed sets, their compatibility with base change and from the  exact sequence (\ref{ESCFT}).
\end{proof}

\subsubsection{The motivic lifting problem $\mathrm{MotLift}(M,q)$.}

Let $(M,q)$ be an orthogonal $E$-motive given by the representation
$\rho_q:\mathcal{G}_E\rightarrow \mathbf{O}(q_B)$. We say that $\mathrm{MotLift}(M,q)$ has a solution if there exists a representation $\widetilde{\rho}_q$ such that the following diagram commutes:
\[ \xymatrix{
& & \widetilde{\mathbf{O}}(q_B)\ar[d]^{}\\
\mathcal{G}_E\ar[rr]^{\rho_q}\ar[rru]^{\widetilde{\rho}_q}
& &{\mathbf{O}}(q_B)
}\]

\begin{prop}
If $\mathrm{MotLift}(M,q)$ has a solution, then $\delta^2_q(\mathfrak{P}_M)=0$, $sw_2(\rho_{\lambda}) =0$ for any finite prime $\lambda$, and $sw_2(\rho_{\lambda\mid{G_{E_{\lambda}}}})=sp_2(\rho_{\lambda\mid{G_{E_{\lambda}}}})=0$ for any prime $\lambda$ of $E$.
\end{prop}
\begin{proof}
Assume that $\mathrm{MotLift}(M,q)$ has a solution. Then $\mathcal{G}_E\rightarrow \mathcal{G}_{M/E}$ factors through $\widetilde{\mathcal{G}}_{M/E}$, and 
$\mathfrak{P}\wedge^{\mathcal{G}_E}\widetilde{\mathcal{G}}_{M/E}$ lifts $\mathfrak{P}_M$ so that $\delta_q^2(\mathfrak{P}_M)=0$. Moreover, for $\lambda$ finite the map 
$\rho_{\lambda}:G_{E}\rightarrow \mathbf{O}(q_{\lambda})(E_{\lambda})$ factors through $\widetilde{\mathbf{O}}(q_{\lambda})(E_{\lambda})$, so that $sw_2(\rho_{\lambda})=0$. Similarly, for any prime $\lambda$ of $E$, the map $\rho_{\lambda\mid G_{E_{\lambda}}}:G_{E_{\lambda}}\rightarrow \mathbf{O}(q_{\lambda})(E_{\lambda})$ factors through $\widetilde{\mathbf{O}}(q_{\lambda})(E_{\lambda})$, which yields $sw_2(\rho_{\lambda\mid{G_{E_{\lambda}}}})=sp_2(\rho_{\lambda\mid{G_{E_{\lambda}}}})=0$. 

\end{proof}

\subsubsection{The class $\delta_q(\mathfrak{P}_M)$ as an obstruction to the existence of isometry $q_B\simeq q_{dR}$.}

Let $E$ be a totally imaginary number field and let $(M,q)$ be an arbitrary orthogonal $E$-motive. Recall that $E$ has strict cohomological dimension $2$. We consider the multiplicative group
$$H^*(G_E,\mathbb{Z}/2\mathbb{Z})^{\times}:=\lbrace 1+a_1+a_2\in \sum_{0\leq i\leq 2} H^i(G_E,\mathbb{Z}/2\mathbb{Z}),\,\,a_i\in H^i(G_E,\mathbb{Z}/2\mathbb{Z})\rbrace$$
of the cohomology ring $H^*(G_E,\mathbb{Z}/2\mathbb{Z})$. 
We consider the elements 
\begin{eqnarray*}
\delta_{q}(\mathfrak{P}_M)&=&1+\delta_{q}(\mathfrak{P}_M)+ \delta^2_{q}(\mathfrak{P}_M)\\
w(q_B)&=&1+w_1(q_B)+ w_2(q_B)\\
w(q_{dR})&=&1+w_1(q_{dR})+w_2(q_{dR})
\end{eqnarray*}
of $H^*(G_E,\mathbb{Z}/2\mathbb{Z})^{\times}$. We already record the following corollary of Theorem \ref{cor-dR-B}.

\begin{cor} Assume that $E$ is a totally imaginary. The forms $q_{B}$ and $q_{dR}$ are isometric if and only if $\delta_q(\mathfrak{P}_M)=1$.
\end{cor}
\begin{proof}
Using the notation introduced above, Theorem \ref{cor-dR-B}(1) reads as follows:
\begin{equation}\label{short}
\delta_q(\mathfrak{P}_M)= w(q_{B})^{-1}\cdot w(q_{dR}).
\end{equation}
By (\ref{short}) we have $\delta_q(\mathfrak{P}_M)=1$ if and only if $w(q_B)=w(q_{dR})$. This is equivalent to $w(q_B)_{\lambda}=w(q_{dR})_{\lambda}$ for any place $\lambda$, hence equivalent to $q_B\otimes_E E_{\lambda}\simeq q_{dR}\otimes_E E_{\lambda}$ for any $\lambda$, which is in turn equivalent to $q_{B}\simeq q_{dR}$.
\end{proof}

\subsection{Examples of orthogonal Nori motives}\label{exSaito}

Let $V/E$ be a smooth projective algebraic variety of even dimension $n$. By Corollary \ref{corqSaito}, the motive $h^n(V)(n/2)\in\mathrm{NMM}_E(E)$ has a canonical orthogonal structure $q$ such that the forms $q_B$ and $q_{dR}$ 
are the perfect pairings
$$q_B:H^n_{B}(V(\mathbb{C}),(2i\pi)^{n/2}\mathbb{Q})_E\otimes H^n_{B}(V(\mathbb{C}),(2i\pi)^{n/2}\mathbb{Q})_E\stackrel{\cup}{\longrightarrow} H^{2n}_{B}(V(\mathbb{C}),(2i\pi)^{n}\mathbb{Q})_E\stackrel{tr}{\longrightarrow} E$$
and
$$q_{dR}:H^n_{dR}(V/E)\otimes H^n_{dR}(V/E)\stackrel{\wedge}{\longrightarrow}H^{2n}_{dR}(V/E)\stackrel{tr}{\longrightarrow} E$$
given by cup-product and trace maps. Similarly, for  $\lambda\mid p$ a finite prime of $E$, the $\lambda$-adic realization yields   
the orthogonal Galois representation $\rho_{\lambda}$ of $G_E$ given by $p$-adic \'etale cohomology
$$q_{\lambda}:H^n(V_{\overline{E}},\mathbb{Q}_p(n/2))_{E_{\lambda}}\otimes H^n(V_{\overline{E}},\mathbb{Q}_p(n/2))_{E_{\lambda}}\stackrel{\cup}{\longrightarrow} H^{2n}(V_{\overline{E}},\mathbb{Q}_p(n/2))_{E_{\lambda}}\stackrel{tr}{\longrightarrow} E_{\lambda}$$
with its natural Galois action. If  $\lambda$ is a real prime of $E$,  then $\rho_{\lambda\mid G_{E_{\lambda}}}$ is the $G_{E_{\lambda}}$-equivariant quadratic form
$$H^n(V_{\lambda}(\mathbb{C}),(2i\pi)^{n/2}\mathbb{Q})_{E_{\lambda}}\otimes H^n(V_{\lambda}(\mathbb{C}),(2i\pi)^{n/2}\mathbb{Q})_{E_{\lambda}}\stackrel{\cup}{\longrightarrow} H^{2n}_{B}(V_{\lambda}(\mathbb{C}),(2i\pi)^{n}\mathbb{Q})_{E_{\lambda}}\stackrel{tr}{\longrightarrow} E_{\lambda}$$
where $V_{\lambda}(\mathbb{C})$ is the set of complex points of $V$ over $E$ with respect with the complex embedding $E\stackrel{\lambda}{\rightarrow}\mathbb{R}\rightarrow\mathbb{C}$.

\subsection{The main result}
Recall that we denote by
$$\appl{H^*(G_E,\mathbb{Z}/2\mathbb{Z})}{H^*(G_{E_{\lambda}},\mathbb{Z}/2\mathbb{Z})}{x}{x_{\lambda}}$$
the canonical restriction map, for any prime $\lambda$. In particular, if $x\in H^2(G_E,\mathbb{Z}/2\mathbb{Z})$ then we have
$x=\sum_{\lambda}x_{\lambda}$ in $H^2(G_E,\mathbb{Z}/2\mathbb{Z})$. 

\begin{thm}\label{cor-dR-B}  Let $(M,q)$ be an orthogonal Nori $E$-motive.
\begin{enumerate}
\item We have 
\begin{eqnarray*}
\delta_{q}^1(\mathfrak{P}_M)&=&w_1(q_B)+ w_1(q_{dR})\\
\delta_{q}^2(\mathfrak{P}_M)&=&w_2(q_B)+ w_1(q_B)\cdot w_1(q_B)+ w_1(q_B)\cdot w_1(q_{dR})+w_2(q_{dR}).
\end{eqnarray*}
\item Assume that $(M,q)$ is of the form $(h^n(V)(n/2),q)$ as in Section \ref{exSaito}. Then 
we have $$\delta_q^1(\mathfrak{P}_M)=sw_1(\rho_{\lambda}).$$
\item Let $(M,q)$ be an arbitrary orthogonal motive. Assume that either $E=\mathbb{Q}$ or $E$ is totally imaginary. Then we have $$\delta_q^2(\mathfrak{P}_M)=\sum_{\lambda} sw_2(\rho_{\lambda\mid{G_{E_{\lambda}}}})+ sp_2(\rho_{\lambda\mid{G_{E_{\lambda}}}}).$$
\end{enumerate}
\end{thm}

\begin{proof}  
In view of the isomorphism of $\mathcal{G}_{M/E}$-torsors
$\mathfrak{P}_M\simeq {\bf Isom}^{\otimes}(H^*_{B,E\mid <M>},H^*_{dR,E\mid <M>})$,
Assertion (1) follows from Theorem \ref{thm-Tannakian}. 

Let $\lambda$ be a finite place of $E$. By Artin's comparison theorem (\ref{Artin}), we have an isometry $(M_B,q_B)\otimes_{E}E_{\lambda}\simeq (M_{\lambda},q_{\lambda})$. We denote by $(M_{dR,E_{\lambda}},q_{dR,E_{\lambda}})$ the $E_{\lambda}$-quadratic space 
$(M_{dR}, q_{dR})\otimes_E E_{\lambda}$.  It follows from the first assertion that we have identities
\begin{equation}\label{Xp}
\delta_q^1(\mathfrak{P})_{\lambda}=w_1(q_{\lambda})+ w_1(q_{dR,E_{\lambda}})
\end{equation}
\begin{equation}\label{XXp}
\delta_q^2(\mathfrak{P})_{\lambda}=w_2(q_{\lambda})+ w_1(q_{\lambda})\cdot w_1(q_{\lambda})+ w_1(q_{\lambda})\cdot w_1(q_{dR,E_{\lambda}})+w_2(q_{dR,E_{\lambda}})
\end{equation}
since the Hasse-Witt invariants are compatible with base-change. 

By (\ref{p-adic-comparison}) the representation $\rho_{\lambda\mid G_{E_{\lambda}}}$ is de Rham and $(M_{dR,E_{\lambda}},q_{dR,E_{\lambda}})$ is a $\lambda$-adic twist of $(M_{\lambda}, q_{\lambda})$ by $\rho_{\lambda\mid G_{E_{\lambda}}}$ in the sense of Section \ref{section-Btwist} (see for example \cite{Brinon-Conrad}, Theorem 2.2.3). By (\ref{Xp}), (\ref{XXp}) and Corollary \ref{dR}, we obtain 
\begin{equation}\label{hehe}
\delta_q^1(\mathfrak{P})_{\lambda}=sw_1(\rho_{\lambda\mid G_{E_{\lambda}}})
\end{equation}
\begin{equation}
\delta_q^2(\mathfrak{P})_{\lambda}=sw_2(\rho_{\lambda\mid G_{E_{\lambda}}})+ sp_2(\rho_{\lambda\mid G_{E_{\lambda}}}) 
\end{equation}
for any $\lambda\nmid\infty$. In order to prove the assertion (3), it remains to consider the case $E=\mathbb{Q}$ and $\lambda=\infty$, and to prove that we have
\begin{equation}\label{requequ}
w_2(q_{\infty})+ w_1(q_{\infty})\cdot w_1(q_{\infty})+ w_1(q_{\infty})\cdot w_1(q_{dR,\mathbb{R}})+w_2(q_{dR,\mathbb{R}})=sw_2(\rho_{\infty\mid G_{\mathbb{R}}})+ sp_2(\rho_{\infty\mid G_{\mathbb{R}}})
\end{equation}
where $(M_{\infty},q_{\infty})=(M_B,q_{B})\otimes_{\mathbb{Q}}\mathbb{R}$ and $\rho_{\infty\mid G_{\mathbb{R}}}: G_{\mathbb{R}}\rightarrow \mathbf{O}(q_{\infty})(\mathbb{R})$ is induced by the action of the complex conjugation on $M_{B}$.
But in this case we have an isometry 
$$\left(M_{\infty}\otimes_{\mathbb{R}}\mathbb{C}\right)^{G_{\mathbb{R}}}\simeq M_{dR,\mathbb{R}}$$
and thus $q_{dR,\mathbb{R}}$ is the Fr\"ohlich twist of $q_{\infty}$ by $\rho_{\infty\mid G_{\mathbb{R}}}$, so that (\ref{requequ}) is given by Corollary \ref{FS}.

We show Assertion (2). By (\cite{Saito11} Corollary 3.3), the global Stiefel-Whitney invariant $sw_1(\rho_{\lambda})$ does not dependent on the finite place $\lambda$. Let $\mu$ be a fixed finite place of $E$. For any $\lambda$ finite, we have 
$$sw_1(\rho_{\mu})_{\lambda}=sw_1(\rho_{\lambda})_{\lambda}=sw_1(\rho_{\lambda\mid G_{E_{\lambda}}})=\delta_q^1(\mathfrak{P})_{\lambda}.$$ 
But the map
$$H^1(G_{E},\mathbb{Z}/2\mathbb{Z})\longrightarrow \Prod_{\lambda\nmid\infty} H^1(G_{E_{\lambda}},\mathbb{Z}/2\mathbb{Z})$$
is injective, by the Chebotarev density theorem. In view of (\ref{hehe}), we obtain $\delta_q^1(\mathfrak{P}_M)=sw_1(\rho_{\mu})$, as required.  
\end{proof}

\subsection{The case $M=h^n(V)(n/2)$}

Let $V$ be a smooth projective variety over $E$ of even dimension $n$. We endow $M=h^n(V)(n/2)$ with its canonical orthogonal structure $q$. For any prime $p$, we consider $sw_1(\rho_p)\in H^1(G_E,\mathbb{Z}/2\mathbb{Z})\simeq E^{\times}/E^{\times 2}$. We denote by 
$b_i= \mathrm{dim}_{\mathbb{Q}}H_B^i(V_{\sigma}(\mathbb C), \mathbb Q)$ the $i^{th}$ Betti number
and, following \cite{Saito94}, we set  $$r(V)=\sum_{i<n}(-1)^i\cdot b_i \hspace{0.5cm}\mbox{ (respectively } r(V)=\sum_{i\leq n}(-1)^i\cdot b_i\mbox{)}$$ 
if $n \equiv 0 \ \mathrm{mod}\ 4$ (respectively if $n \equiv  2 \ \mathrm{mod}\ 4$). 

\begin{cor}
If $V$ is a smooth projective variety over $E$ of even dimension $n$, then we have
$$\mathbb{P}_{\mathrm{det}_E\left(h^n(V)(n/2)\right)}=E\left(\sqrt{sw_1(\rho_p)}\right)$$
for any prime $p<\infty$, and
$$\delta_{q}^2(\mathfrak{P}_M)=w_2(q_B)+ (-1)^{r(V)}\cdot (-1)^{r(V)}+ (-1)^{r(V)}\cdot w_1(q_{dR})+w_2(q_{dR}).$$
\end{cor}
\begin{proof}
The first assertion immediately follows from Theorem \ref{cor-dR-B} and Proposition \ref{propdelta1}. 
As for the second assertion, it is enough to showing that $$w_1(q_B)=\mathrm{disc}(q_B)=(-1)^{r(V)}$$ in $E^{\times}/E^{\times2}$ by Theorem \ref{cor-dR-B}(1). This is equivalent to the equality (written additively)
$$w_1(q_{dR})=sw_1(\rho_p)+{r(V)}\cdot\lbrace-1\rbrace$$ by Theorem \ref{cor-dR-B}(2), which is (\cite{Saito94} Theorem 2) (see also \cite{Saito12}, Remark 2.4).
\end{proof}
\begin{rem}
It is conjectured that there is a short exact sequence
$$1\rightarrow \mathcal{G}^0_E\rightarrow \mathcal{G}_E\rightarrow G_E\rightarrow 1$$
where $\mathcal{G}^0_E$ is the connected component of the identity in $\mathcal{G}_E$, $G_E$ is the absolute Galois group and  the map $\mathcal{G}_E\rightarrow G_E$ is induced by the inclusion of Artin motives inside the category of all motives. Assume that there is such an exact sequence and let $(M,q)$ be an orthogonal motive, given by the orthogonal representation
$$\rho_M:\mathcal{G}_E\rightarrow \mathbf{O}(q_B).$$
Then the determinant motive $\mathrm{det}_E(M)$ is given by the representation
$$\rho_{\mathrm{det}_E(M)}:\mathcal{G}_E\rightarrow \mathbf{O}(\mathrm{det}(q)_B)=\mathbb{Z}/2\mathbb{Z}$$
which must factor as follows:
$$\rho_{\mathrm{det}_E(M)}:\mathcal{G}_E\rightarrow G_E\stackrel{\alpha}{\rightarrow}\mathbb{Z}/2\mathbb{Z}.$$
Hence $\mathrm{det}_E(M)$ would be an orthogonal Artin  motive, and
we would obtain $\mathrm{det}(\rho_{\lambda})=\alpha$ and 
$\mathbb{P}_{\mathrm{det}_E(M)}=E\left(\sqrt{sw_1(\rho_{\lambda})}\right)$
for any finite prime $\lambda$ of $E$.
\end{rem}

\section{Examples}

\subsection{Artin motives}

In the case of orthogonal Artin motives, Theorem \ref{cor-dR-B} reduces to classical results of Serre and Fr\"ohlich, see (\cite{Serre84} Theorem 1) and (\cite{Fr\"ohlich85} Theorems 2 and 3) over number fields. The aim of this section is to compute these invariants for some explicit Artin motives (see Proposition \ref{prop-Artin} for a description of Artin motives).

Let $(M,q)$ be an orthogonal Artin motive given by a representation of the form
$$\rho_q:\mathcal{G}_{E}\stackrel{}{\longrightarrow} G_{E}\stackrel{\rho_0}{\longrightarrow} {\bf O}(q_B)$$
We have $\mathfrak{P}_M=\mathrm{Spec}(K)$ for $K/E$ a finite Galois extension. In this situation Fr\"ohlich defines \cite{Fr\"ohlich85} the twisted form $(q_B\otimes_{E} K)^{G_{K/E}}$.
\begin{cor}\label{cor-Fro}
We have an isometry $q_{dR}\simeq (q_B\otimes_{E} K)^{G_{K/E}}$ and the following identities in $H^*(G_{E},\mathbb{Z}/2\mathbb{Z})$:
\begin{eqnarray*}
\delta_{q}^1(\mathrm{Spec}(K))&=&w_1(q_B)+ w_1(q_{dR})=sw_1(\rho_0);\\
\delta_{q}^2(\mathrm{Spec}(K))&=&w_2(q_B)+ w_1(q_B)\cdot w_1(q_B)+ w_1(q_B)\cdot w_1(q_{dR})+w_2(q_{dR})\\
&=&sw_2(\rho_0)+ sp_2(\rho_0).
\end{eqnarray*}

\end{cor}
\begin{proof}
The local representation $\rho_{\lambda\mid G_{E_\lambda}}$ identifies with the composition
$$G_{E_\lambda}\longrightarrow G_{E}\stackrel{\rho_0}{\longrightarrow}\mathbf{O}(q_B)(E)\longrightarrow \mathbf{O}(q_B)(E_\lambda)$$ for any place $\lambda$  of $E$. We obtain
$$sw_2(\rho_0)=\sum_{\lambda}sw_2(\rho_{\lambda\mid G_{E_\lambda}})\mbox{ and } sp_2(\rho_0)=\sum_{\lambda}sp_2(\rho_{\lambda\mid G_{E_\lambda}})$$
and $sw_1(\rho_0)=sw_1(\rho_{\lambda})$ for any $\lambda$ finite. The result then follows immediately from Theorem \ref{cor-dR-B} and its proof.
\end{proof}

Let $F/E$ be finite extension  let $M=h^0(\textrm{Spec}(F))(0)\in\mathrm{NMM}_E(E)$. Consider the canonical orthogonal structure $q$ on $M$. In this case one has $M_B=E^{\mathrm{Hom}_{E}(F,\mathbb{C})}$ (w.r.t. $\sigma:E\rightarrow\mathbb{C}$)
and 
$$\fonc{q_B}{E^{\mathrm{Hom}_{E}(F,\mathbb{C})}\otimes E^{\mathrm{Hom}_{E}(F,\mathbb{C})}}{E}{((v_{\tau})_{\tau},(w_{\tau})_{\tau})}{\sum_{\tau} v_{\tau}\cdot w_{\tau}}$$
is the standard form. Moreover, one has $M_{dR}=F$ and
$$\fonc{q_{dR}}{F\otimes F}{E}{(x,y)}{\mathrm{Trace}_{F/E}(x\cdot y)}$$
is the usual trace form. The representation $\rho_q:\mathcal{G}_{E}\rightarrow {\bf O}(q_B)$ factors as follows
$$\rho_q:\mathcal{G}_{E}\stackrel{}{\longrightarrow} G_{E}\stackrel{\rho_0}{\longrightarrow} {\bf O}(q_B)$$
where $G_{E}$ acts on $M_B=E^{\mathrm{Hom}_{E}(F,\mathbb{C})}=E^{\mathrm{Hom}_{E}(F,\overline{E})}$ by permuting the factors. Let $K$ be the Galois closure of (some embedding of) $F/E$ in $\overline{E}\subset\mathcal{P}$. We have $\mathfrak{P}_M=\mathrm{Spec}(K)$. We denote by $d_{F/E}$ the discriminant.
\begin{cor}
For $M=h^0(\mathrm{Spec}(F))$, one has  
$$\delta_{q}^1(\mathrm{Spec}(K))=w_1(q_{dR})=d_{F/E}=sw_1(\rho_0);$$
$$\delta_{q}^2(\mathrm{Spec}(K))=w_2(q_{dR})=sw_2(\rho_0)+ sp_2(\rho_0).$$
\end{cor}
\begin{proof}
The from $q_B$ is the standard form, hence $w_1(q_B)=w_2(q_B)=0$. 
\end{proof}

\subsubsection{The lifting problem $\mathrm{Lift}(h^0(\mathrm{Spec}(F)),q)$}\label{sectartexplicit}
As above, we consider  an extension of number fields $F/E$ with discriminant $d_{F/E}$ and let $M=h^0(\mathrm{Spec}(F))$. Let  $K/E$ be the  Galois closure of $F$ in $\overline{E}$. Recall that $\mathcal{G}_{M/E}=G_{K/E}$. The terms $\delta_q^2(\mathrm{Spec}(K))$ and $sw_2(\rho_0)$ have a Galois theoretic interpretation. We consider the exact sequence of groups 
\begin{equation}
1\longrightarrow\mathbb{Z}/2\mathbb{Z}\longrightarrow \widetilde{\mathbf{O}}(q_B)(\overline{E})\longrightarrow \mathbf{O}(q_B)(\overline{E})\longrightarrow 1. 
\end{equation}
Then the group-scheme $\widetilde{\mathcal{G}}_{M/E}$ can be seen as the group of $\overline{E}$-points
$$\widetilde{\mathcal{G}}_{M/E}(\overline{E})=G_{K/E}\times_{\mathbf{O}(q_B)(\overline{E})}\widetilde{\mathbf{O}}(q_B)(\overline{E})$$
endowed with its natural  $G_{E}$-action. We also consider
$$\widetilde{G}_{K/E}:=G_{K/E}\times_{\mathbf{O}(q_B)(\overline{E})}\widetilde{\mathbf{O}}(q_B)(\overline{E})$$
endowed with  the trivial action of $G_{E}$, as a constant group scheme over $E$. Note that the group-schemes $\widetilde{\mathcal{G}}_{M/E}$ and $\widetilde{G}_{K/E}$ are  isomorphic if and only if $\widetilde{\mathcal{G}}_{M/E}$ is constant, i.e. iff the natural $G_{E}$-action is trivial.
We say that the lifting problem $\mathrm{Lift}^\delta(h^0(\mathrm{Spec}(F)),q)$ has a solution if there exists a $\widetilde{G}_{K/E}$-torsor lifting $\mathfrak{P}_M=\mathrm{Spec}(K)$, i.e. if it is possible to embed the Galois extension $K/E$ into a Galois extension $\widetilde K/E$ with Galois group $\widetilde{G}_{K/E}$. Then $\mathrm{Lift}^\delta(h^0(\mathrm{Spec}(F)),q)$ has a solution if and only if $sw_2(\rho_0)=0$ while $\mathrm{Lift}(h^0(\mathrm{Spec}(F)),q)$ has a solution if and only if $\delta_q^2(\mathrm{Spec}(K))=0$. It follows that if $sp_2(\rho_0)\neq0$ then $\widetilde{\mathcal{G}}_{M/E}$ is not constant. Moreover, it follows from \cite{Fr\"ohlich85} and from \cite{EKV93} in a more general set up, that $sp_2(\rho)=(2)(d_{F/E})$.

\subsubsection{Explicit examples}
We assume in this section that $F/\mathbb Q$ is of degree $4$ and that $G_{K/\mathbb{Q}}=\mathfrak{S}_4$. The group $\widetilde{\mathfrak{S}}_4$ is characterized by the fact that a transposition in $\mathfrak{S}_4$ lifts to an element of order $2$ of  $\widetilde{\mathfrak{S}}_4$ whereas a product of $2$ disjoint transpositions lifts to an element of order $4$.  
Since $2$ is a square of $\mathbb R=\mathbb Q_{\infty}$ we note that $(2, d_F)_{\infty}=1$. We write  $w_2$ for the Hasse Witt invariant $w_2(\mbox{Tr}_{F/\mathbb Q})$ and 
we denote by $(r,s)$ the signature of the field $F$.  Proposition 6.5 gives $w_{2, \infty}=-1$ if $(r, s)=(0, 2)$ and $w_{2, \infty}=1$ otherwise. The following result is taken from \cite{Jehanne}. 
\begin{prop}\label{jehthm} Let  $p\geq 3$ be  a prime number.  Then  $\omega_{2,p}=(2,d_F)_p=1$ if $p$ is  non ramified  in  $F$.
Moreover we have the following equalities:
\begin{enumerate}
\item $w_{2,p}=(2,d_F)_p=(-1)^{\frac{p^2-1}8}$  if  $p\mathcal O_F=\mathfrak p_1^2\mathfrak p'_1\mathfrak   p''_1$, 
\item $w_{2,p}=(2,d_F)_p=1$ if  $p\mathcal O_F=\mathfrak p_1^3\mathfrak p'_1$,
\item $w_{2,p}=-(2,d_F)_p=-(-1)^{\frac{p^2-1}8}$ if  $p\mathcal O_F=\mathfrak p_1^2\mathfrak p'_2$, 
\item $w_{2,p}=(-1)^{\frac{p-1}2}$ and  $(2,d_F)_p=(-1)^{\frac{p^2-1}8}$ if  $p\mathcal O_F=\mathfrak p_1^4$, 
\item $w_{2,p}=(-1)^{\frac{p+1}2}$ and  $(2,d_F)_p=1$ if  $p\mathcal O_F=\mathfrak p_2^2$, 
\item $w_{2,p}=(-1)^{\frac{p-1}2}(d_F, p)_p$ and  $(2,d_F)_p=1$ if  $p\mathcal O_F=\mathfrak p_1^2\mathfrak p_1^{'2}$.
\end{enumerate}
\end{prop}

In the following examples $P$ is an irreducible polynomial of $\mathbb Q[X]$, $\alpha$ is a root of $P$ in $\mathbb C$, $F=\mathbb Q(\alpha)$ and  $G_{K/\mathbb Q} =\mathfrak{S}_4$.

\begin{itemize} 
\item $P=X^4+X-1$,   $d_F=-283$ and the signature of $F$ is $(2, 1)$. Since the ideal $283\mathcal O_F$ decomposes into a product $\mathfrak p_1^2\mathfrak p'_1\mathfrak   p''_1$ it follows from  Proposition \ref{jehthm}(1)   that $w_{283, p}=( 2, d_F)_{283}=-1$. Since  $283$ is the unique finite prime ramified in $F$, we have 
 $w_{2, p}=(2, d_F)_p$ for every $p\neq 2$. We  deduce from the product formula that $w_2$ and $(2, d_F)$ coincide at every prime $p$ of $\mathbb Q$ and so are equal. We conclude that
 $$w_2(\mathrm{Tr}_{F/\mathbb Q})=sp_2(\rho)\ne 0, \ sw_2(\rho)=0$$
hence $\mathrm{Lift}(h^0(\mathrm{Spec}(F)),q)$ has no solution and
$\mathrm{Lift}^\delta(h^0(\mathrm{Spec}(F)),q)$ has a solution.\\
 
  \item $P=X^4+X^3-2X-1$ and  $d_F=-5^211$. One checks that 
 $11\mathcal O_K=\mathfrak p_2\mathfrak p_1'^2$and  $5\mathcal O_K=\mathfrak p_2^2$. By elementary  computations on Hilbert symbols, it follows from Proposition \ref{jehthm}(3) and \ref{jehthm}(5) that 
 $$(2, d_F)_{11}=-w_{2, 11}=-1 , \  (2, d_F)_5=-w_{2, 5}=1.$$ We conclude that
 $$w_2(\mathrm{Tr}_{F/\mathbb Q})\ne 0,\ sp_2(\rho)\ne 0, \  sw_2(\rho)\ne 0$$
hence $\mathrm{Lift}(h^0(\mathrm{Spec}(F)),q)$ has no solution and
$\mathrm{Lift}^\delta(h^0(\mathrm{Spec}(F)),q)$ has no solution.\\

\item $P=X^4-2X^2-4X-1$,  $d_F=-2^811$ and the signature of $F$ is $(2, 1)$. One checks that $11\mathcal O_K=\mathfrak p_1^2\mathfrak p_2'$. Hence, by Proposition  \ref{jehthm}(3), we know that 
 $w_{2, 11}=-(2, d_F)_{11}=1$. We conclude that 
 $$w_2(\mathrm{Tr}_{F/\mathbb Q})=0,\ sp_2(\rho)\ne 0,\ sw_2(\rho)\ne 0$$
hence $\mathrm{Lift}(h^0(\mathrm{Spec}(F)),q)$ has a solution and
$\mathrm{Lift}^\delta(h^0(\mathrm{Spec}(F)),q)$ has no solution.\\

\item $P=X^4-X^3-X+1$, $d_F=2777$ and the signature of $F$ is $(4, 0)$. Since $2777\mathcal O_K=\mathfrak p_1^2\mathfrak p_1'\mathfrak p_1''$, by Proposition  \ref{jehthm} we obtain that $(2,2777)_{2777}=\omega_{2,2777}=1$. We conclude that
 $$w_2(\mathrm{Tr}_{F/\mathbb Q})=0,\ sp_2(\rho)= 0,\ sw_2(\rho)= 0$$
hence both lifting problems $\mathrm{Lift}(h^0(\mathrm{Spec}(F)),q)$ and
$\mathrm{Lift}^\delta(h^0(\mathrm{Spec}(F)),q)$ have a solution.

\end{itemize}

\subsection{Complete intersections.}

We consider a smooth and proper variety $V$ over the number field $E$  of even dimension $n$.  

\subsubsection{The Betti form $q_B$} For any integer $k\geq 0$ and for $R=\mathbb Z, \mathbb Q$ or $ \mathbb R$, we let $H^k(V;  R)$ be the singular cohomology group $H^k(V_\sigma (\mathbb C); R)$ with coefficients in $R$. We set $H^k(V)=H^k(V; \mathbb Z)$. We consider the quadratic form on $\mathbb Q$
$$q_B: H^n(V; \mathbb Q)\times H^n(V; \mathbb Q) \rightarrow  \mathbb Q$$ 
obtained by composing the cup product with  evaluation on the fundamental class. 
\begin{prop}\label{wbet} Let $(r, s)$ be the signature of the extended form $\mathbb R\otimes_{\mathbb Q}q_B$. Then $q_B$ is isometric to the quadratic form 
$$\sum_{1\leq i\leq r}x_i^2-\sum_{1\leq j\leq s}x_{r+j}^2.$$
In particular, we have $w_1(q_B)=s(-1)$ and $w_2(q_B)=\binom{s}{2}(-1,-1)$. 
\end{prop}
\begin{proof} We set $q=q_B$ and $q'=\sum_{1\leq i\leq r}x_i^2-\sum_{1\leq j\leq s}x_{r+j}^2$. For a quadratic form $t$ on $\mathbb Q$ and  any place $v$ of $\mathbb Q$ we denote by $t_v$ the extended form $\mathbb Q_v\otimes_{\mathbb Q} t$ on $\mathbb Q_v$.  In order to prove the proposition it suffices to show that $q$ and $q'$ are locally isometric for every place of $\mathbb Q$. Since $q$ and $q'$ have the same signature then $q_\infty$ and $q'_\infty$ are isometric over $\mathbb R$. The form $q$  is induced by  extension from the integral form 
$$q_0: H^n(V)_{\mathrm{cotor}}\times H^n(V)_{\mathrm{cotor}}\rightarrow \mathbb Z$$ 
where $H^n(V)_{\mathrm{cotor}}$ is the maximal quotient of $H^n(V)$ which is free as a $\mathbb{Z}$-module.

It follows from Poincar\'e duality that $q_0$ is unimodular.
 Therefore the determinant of  $q$ is congruent to $\pm 1$ modulo $\mathbb Q^{* 2}$. The group homomorphism $\mathbb Q^*\rightarrow \mathbb R^*$ induces an injective map  $ \{ \pm 1\}\mathbb Q^{*2}/\mathbb Q^{* 2}\rightarrow \mathbb R^*/\mathbb R^{* 2}$ 
such that $w_1(q)\rightarrow w_1(q_{ \mathbb R})$. This implies that
 
$$ w_1(q)=s(-1) $$ 
and so $w_1(q)=w_1(q')$.  
We now examine $w_2(q)\in H^2(\mathbb Q, \mathbb Z/2\mathbb Z)$. For any prime $p\leq \infty$ we let $w_2(q)_p$ be the $p$-component of 
$w_2(q)$ in $H^2(\mathbb Q_p, \mathbb Z/2\mathbb Z)$. Indeed $w_2(q)_p=w_2(q_{ p})$. For $p=\infty$ we note that  $w_2(q)_{\infty}=w_2(q')_{\infty}=\binom{s}{2}(-1,-1)$. Let $p\neq 2$ be a prime number.  The form $q_{ p}$ is  induced from the $\mathbb Z_p$-unimodular form $q_{0, p}$. Since $\mathbb Z_p$ is a local ring, this form has an orthogonal basis and  we may write $q_{0,p}=<a_1,a_2,...,a_n>$ for $a_i \in \mathbb Z_p^*$. Therefore 
$$w_2(q_{ p})=\sum _{i<j}(a_i, a_j). $$ Since the $a_i$'s are units we conclude that  $w_2(q_p)=0$. It now follows from the definition of $q'$ that $w_2(q'_p)=0$. Hence we have proved that $w_2(q)_v=w_2(q')_v$ for any plave $v\neq 2$. By  the product formula we deduce that $w_2(q)_2=w_2(q')_2$. Therefore the forms  $q$ and $q'$  have the same rank, the same signature and for any place $v$ of $\mathbb Q$ then $w_i(q_v)=w_i(q'_v)$,  for $i \in \{1, 2\}$. It follows that $q$ and $q'$ are isometric. In particular they have the same Hasse-Witt invariants.
\end{proof}
As before we denote by $q_{dR}$ the quadratic form on $H_{dR}^n(V/E)$ induced by the cup product and we set 
$$d_V= \mathrm{disc}(q_{dR})=w_1(q_{dR}).$$ As a consequence of  Proposition \ref {wbet} we observe that the comparison formula of Theorem \ref{cor-dR-B},  applied to the quadratic motive $(M=h^n(V)(n/2), q)$, should be written 
\begin{equation}\label{comp0}
\delta^1_q(\mathfrak P_M)=sw_1(\rho_p)=d_V+s(-1) \end{equation}
and 
\begin{equation}\label{compo1}  
\delta^2_q(\mathfrak{P}_M)=w_2(q_{dR})+s(d_V, -1)+\binom{s+1}{2}(-1, -1), 
\end{equation}
where in these formulas $(-1)$ (resp. $(-1, -1)$) has to be understood as the image of $-1$ (resp. $(-1, -1)$)  in $E^*/E^{*2}$ (resp. by the restriction map $H^2(\mathbb Q, \mathbb Z/2\mathbb Z)\rightarrow H^2(E, \mathbb Z/2\mathbb Z)$). 

\subsubsection{The Hasse-Witt invariants of $q_B$}
We assume that $V$ is a smooth  complete intersection  in the projective space $\mathbf{P}_E^{n+c}$. More precisely $V$ is a variety of even dimension $n$ defined  by $c$ homogeneous polynomials $f_1,..., f_c$ with coefficients in $E$. For $1\leq i\leq c$ we let $d_i$ be the degree of $f_i$. We note that $V_{\sigma}$ is a complete intersection in $\mathbf{P}_{\mathbb{C}}^{n+c}$ defined by the polynomials $f_i$ seen via $\sigma$ as polynomials with coefficients in $\mathbb C$. In this situation the rank of $q_B$ and the Hasse Witt invariants  $w_1(q_B)$ and $w_2(q_B)$ can be explicitely expressed  in terms of  the integers $n, d_1,...,d_c$.   We will treat  in some details the case where $V$ is an hypersurface where the formulas are very simple.  However for any given complete intersection  explicit   formulas can be obtained  as well but they may be rather complicated. As an example we will consider  a case  where $c=2$. When  $V$ is an hypersurface  ($c=1$)  we write $f$ for
  $f_1$ and $d$ for $d_1$. For any integer $k$ we denote by $b_k(V)$ the dimension of the $\mathbb Q$-vector space  $H^k(V, \mathbb Q)$. For reason of simplicity  we  will  often  denoted by $V$ the complex variety $V_{\sigma}$ and by $\mathbf{P}^n$ the projective  space  $\mathbf{P}_{\mathbb C}^n$. 
 
 We need some preliminary results. 
\begin{lem}\label{bn}  For any integer $k$ , $0\leq k\leq 2n$ then $H^k(V)$ is torsion free. Moreover  one has:
\begin{enumerate}
\item $b_k(V)=0$ if $k$ is odd and $b_k(V)=1$ if $k$ is even and $k \neq n$. 
\item $b_n(V)= \chi(V)-n$ where $\chi(V)=\sum_{0\leq k\leq 2n} (-1)^kb_k(V)$. 
\item $b_n(V)=2+ \frac{1}{d}[(1-d)^{n+2}-1]$ when $V$ is an hypersurface.
\end{enumerate}

\end{lem}  
\begin{proof} It follows from Lefschetz Theorem (see \cite{Dimca}, Chapter 5, (2.6) and (2.11)) that 
 $$H^k(V)=H^k(\mathbf{P}^{n})\ \ \mbox{for}\   k\neq n. $$
 This implies that $H^k(V)$ is torsion free for any integer $k$ and the equalities of (1). In order to show (2) it suffices to observe that 
 $$\chi(V)= \chi(\mathbf{P}^{n})+b_n(V)-b_n(\mathbf{P}^n).$$
 Since  the Euler characteristic of an hypersurface is given by $\chi(V)=n+2+\frac{1}{d}[(1-d)^{n+2}-1]$, then (3) follows from (2).
 \end{proof}
     There exist  several effective procedures to compute $\chi(V)$ when $V$ is a complete intersection. For instance this can be done   by evaluating the $n$-th Chern class of $V$ on the fundamental class.  This leads to prove that  $\chi(V)$ is the coefficient of $h^{n+c}$ in 
\begin{equation}\label{euler} \frac{(1+h)^{n+c+1}d_1...d_c h^c}{(1+d_1h)(1+d_2h)...(1+d_ch)}.\end{equation}

 Our aim is now to  compute the signature of $q_{B, \mathbb R}$. We consider   the index  $\tau_n(V)=r-s$  of the quadratic lattice $H^n(V)$.     This index   has been extensively studied  (see \cite{libgo} and \cite{libgob} for instance).   If we put $d=d_1d_2...d_c$, we have 
 \begin{equation}\label{libgo1}
\tau_n(V)\equiv \begin{cases}
0 \ \mbox{mod}\ 8 \ \mbox{if  $\binom{\frac{n}{2}+t}{t}$ is even}, \\
d \ \mbox{mod}\ 8 \ \mbox{if  $\binom{\frac{n}{2}+t}{t}$ is odd}.
\end{cases}
\end{equation} 
where $d_1,...d_t$ are even  and $d_{t+1},...d_c$ are odd integers. 
We define the integer $m$ by 
 \begin{equation}\label{libgo1}
m= \begin{cases}
\chi(V)-n  \ \mbox{if  $\binom{\frac{n}{2}+t}{t}$ is even}, \\
\chi(V)-n-d \  \ \mbox{if  $\binom{\frac{n}{2}+t}{t}$ is odd}.
\end{cases}
\end{equation} 
The following proposition is an immediate consequence of (\ref{libgo1}), Proposition \ref{wbet}  and Lemma  \ref{bn}.

\begin{prop}\label{bet1} 
Let $V$ be a smooth complete intersection in $\mathbf{P}_E^{n+c}$ of even dimension $n$. Then $m$ is even and we have 
$$w_1(q_B)=m'(-1)\ \ \mathrm{and}\ \ w_2(q_B)=\binom{m'}{2}(-1, -1), \ \mathrm{with}\ m=2m'.$$ 
\end{prop}

If $V$ is an hypersurface defined by a polynomial of degree $d$ we deduce from (\ref{libgo1}) that 

 \begin{equation}\label{libgo2}
\tau_n(V)\equiv \begin{cases}
0 \ \mbox{mod}\ 8 \ \mbox{if  $d$ is even and $n\equiv 2\ \mbox{mod}\ 4$}, \\
d \ \mbox{mod}\ 8 \ \mbox{otherwise}.
\end{cases}
\end{equation}
 We denote by $\{u\}$ the integral part of $u\in \mathbb Q$.
\begin{cor}\label{Hasse}
 Let $V$ be a smooth hypersurface  of degree $d$ and  even dimension $n\geq 2$ in the projective space $\mathbf{P}_{E}^{n+1}$. Then 
\begin{enumerate}
\item $w_1(q_B)=\frac{n}{2}(d-1)(-1)$
\item $w_2(q_B)= \begin{cases}
\frac{d-1}{2}(-1, -1) \ \mbox{if } d \mbox{ is odd}, \\
\{\frac{n+2}{4}\}(1+\frac{d}{2})(-1,-1)\ \mbox{if}\  d \mbox{ is even}.
\end{cases}$
\end{enumerate}
\end{cor}
\begin{proof}

We let $r_n$ be the integer defined by  the equality $[(1-d)^{n+2}-1]=dr_n$.  If  $d$ is odd, it follows from Lemma \ref{bn} and (\ref{libgo2}) that 
\begin{equation}\label{bn-}
s \equiv 1+\frac{r_n}{2}-\frac{d}{2}\ \mbox{mod}\ 4.\end{equation}
From the equality $(d-1)^{n+2}=1+dr_n$ we deduce that $d$ and $r_n$ are both odd and that one and only one of them is congruent to $1$ mod $4$. Suppose that $d=1+4v$. 
Since $1+dr_n\equiv 0\ \mbox{mod}\ 8$, then $r_n\equiv -1+4v\ \mbox{mod}\ 8$ and therefore $s\equiv 0\ \mbox{mod}\ 4$. If now $d\equiv 3\ \mbox{mod}\ 4$ and so 
$r_n=1+4t$, we obtain  that $d\equiv -1+4t\ \mbox{mod}\ 8$ and so that $s\equiv 2\ \mbox{mod}\ 4.$

We assume now that $d$ is even. First we check that 
\begin {equation}\label{rn}
1+\frac{r_n}{2}\equiv -\frac{n}{2}+(n+1)\frac{(n+2)}{2}\frac{d}{2}\ \mbox{mod}\ 4.\end{equation}
Suppose that $n\equiv  2\ \mbox{mod}\ 4$ and write $n=2+4u$.  It follows once again from  Lemma \ref{bn} and (\ref{libgo2}) that 
$s\equiv 1+2(1+u)(1+\frac{d}{2})\ \mbox{mod}\ 4$. Therefore we conclude that $s\equiv 3\ \mbox{mod}\  4$ if $n\equiv 2\ \mbox{mod}\ 8$ and 
$d\equiv 0\ \mbox{mod}\ 4$ and $s\equiv 1\ \mbox{mod}\ 4$ otherwise. Finally we suppose $n\equiv 0\ \mbox{mod}\ 4$ and we put $n=4u$. In this case  we know that 
$s\equiv 1+\frac{r_n}{2}-\frac{d}{2}\ \mbox{mod}\ 4$ which implies that $s\equiv 2u(1+\frac{d}{2})\ \mbox{mod}\ 4$. Hence 
$s\equiv 2\ \mbox{mod}\ 4$ if $d\equiv 0\ \mbox{mod}\ 4$ and $n$ not divisible by $8$ and $s\equiv 0\ \mbox{mod}\ {4}$ otherwise. The proposition is proved by a case by case checking. 

\end{proof}
In various situations  more precise computations of the index  leads to the determination  of the form $q_B$ itself.  This is the case when  $V$ is a  quadric  of $\mathbf{P}^{n+1}$.   By Lemma \ref{bn} the unimodular lattice $H^n(V)$ is of rank $2$. By Proposition \ref {Hasse} we have $w_1(q_B)=\frac{n}{2}(-1)$ and  $w_2(q_B)=0$. If $n\equiv2\ \mbox{mod}\ 4$, then 
$\tau_n(V)=0$ and  $q_B$ is isometric to the form $x^2-y^2$. If $n\equiv0\ \mbox{mod}\ 4$, $\tau_n(V)=2$ and $q_B$ is isometric to $x^2+y^2$.

 We now consider the case of a  cubic  surface $V$ ($d=3$ and $n=2$ ).  The rank of the lattice $H^2(V)$ given by  Lemma \ref{bn} is  equal to $7$. Moreover by (\ref{libgo2})  we know that $\tau_n(V)\equiv3\ \mbox{mod}\ 8$ and so   $\tau_n(V)$ is either equal to $3$ or $-5$.  We now observe that,  under our hypotheses,   (\ref{libgo2}) can be improved. More precisely it follows from  (2) in \cite{libgo} that $\tau_n(V)\equiv d+8 \ \mbox{mod}\ 16$. Hence 
$\tau_n(V)=-5$. We conclude that  $q_B$ is  isometric to 
$$x_1^2-x_2^2-x_3^2-x_4^2-x_5^2-x_6^2.$$

We treat now as an example the case of a smooth surface  $V$ defined  in $\mathbf{P}^{4}$ by the polynomials $f_1$ and $f_2$ of degree $d_1$ and $d_2$ ($n=2$ and $c=2$).  It follows from (\ref{euler}) that: 
\begin{equation}\label{euler1}\chi(V)=d_1d_2[d_1^2+d_2^2+d_1d_2-5(d_1+d_2)+10]. 
\end{equation}

\begin{cor} Let $V$ be a smooth surface of $\mathbf{P}_E^4$ defined by the homogeneous polynomials $f_1$ and $f_2$ of degree $d_1$ and $d_2$. Then 
\begin{enumerate}
\item $w_1(q_B)=(d_1d_2-1)(-1)$
\item $w_2(q_B)= \begin{cases}
(\{\frac{d_1}{2}\}+\{\frac{d_2}{2}\}+\{\frac{d_1}{2}\}\{\frac{d_2}{2}\})(-1, -1) \ \mbox{if } d_1 \mbox{and}\   d_2 \mbox{ are   both odd}, \\
(1+\{\frac{d_1}{2}\}\{\frac{d_2}{2}\})(-1,-1)\ \mbox{if}\  d_1 \mbox{or}\ d_2  \mbox{ otherwise}.
\end{cases}$
\end{enumerate}

\end{cor}
\begin{proof} The equalities are obtained from (\ref{euler1})  and  Lemma \ref{bet1} by an easy computation. 
\end{proof}
We may use the results of Corollary \ref{Hasse}  to improve the comparison formulas of Theorem \ref{cor-dR-B} when $(M, q)$ is the quadratic motive attached to an hypersurface.

\begin{cor}\label{dV} Let  $V$ be a smooth hypersurface of even dimension $n\geq 2$ in $\mathbf{P}_E^{n+1}$,  defined by an homogeneous polynomial $f$ of degree $d$ and let $(M=h^n(V)(n/2), q)$ be the associated motive. Then  
$$d_V=w_1(q_{dR})=\varepsilon'(n, d)\mathrm{disc}_{\mathrm{d}}(f) \ \mathrm{in}\ \  E^*/ E^{*2}, $$
where $\mathrm{disc}_{\mathrm{d}}(f)$ is the divided discriminant of $f$ and where $\varepsilon'(n, d)=(-1)^{\frac{(d-1)}{2}}$ if d is odd and is $(-1)^{(1+\frac{n}{2})(1+\frac{d}{2}) +1}$ if $d$ is even. 
\end{cor}
\begin{proof} This is an immediate  consequence of Corollary \ref{Hasse} and \cite{Saito11}.
\end{proof}

\begin{rem} Suppose for simplicity that $E=\mathbb Q$. When $n=0$, we deduce from (\ref {bn-}) that $s \equiv 0\ \mbox{mod}\ 4$ and so that $w_1(q_B)=w_2(q_B)=0$.  This result and Corollary \ref{dV} can be easily understood in this case. Suppose that $V$ is defined by the polynomial $f=X_0^d+a_{d-1}X_0^{d-1}X_1+...+a_0X_1^d$. We set $g(X)=f(X=\frac{X_0}{X_1}, 1)$. The motive $h^0(V)$ is isomorphic to the Artin motive $\mathrm{Spec}(F)$ where $F\simeq \mathbb Q[X]/(g(X))$. As we have seen in Section 6.1, the form $q_B$ is the form $x_1^2+...+x_d^2$ and thus its Hasse-Witt invariants are trivial. Moreover the form $q_{dR}$ identifies with the trace form $\mathrm{Tr}_{F/\mathbb Q}$. Therefore 
\begin{equation}\label{Tr}
d_V=\mathrm{disc}(F/\mathbb Q)=\mathrm{disc}(\mathrm{Tr}_{F/\mathbb Q})=\prod_{i<j}(x_i-x_j)^2, 
\end{equation}
where $x_1, ..., x_d$ are the roots of $g$ in $\mathbb C$. This equality can be expressed in terms of resultants. More precisely, by \cite{GKZ}, Chapter 12, (1.29) and (1.31), 
$$\prod_{i<j}(x_i-x_j)^2=(-1)^{\frac{d(d-1)}{2}}R_{d, d-1}(g, g')=\frac{\varepsilon'(0, d)}{d^{d-2}}R_{d-1,d-1}(D_0(f), D_1(f))=\varepsilon'(0, d)\mathrm{disc}_{\mathrm{d}}(f).$$
We conclude that Corollary \ref{dV} generalizes in higher dimension the classical equality (\ref{Tr}). 
\end{rem}

\begin{cor}  Let  $V$ be a smooth hypersurface of even dimension $n\geq 2$ in $\mathbf{P}_E^{n+1}$,  defined by an homogeneous polynomial  of degree $d$ and let 
$(M=h^n(V)(n/2), q)$ be the associated motive. Then we have
$$\delta^1_q(\mathfrak{P}_M)=\begin{cases}
(-1)^{\frac{d-1}{2}}\cdot\mathrm{disc}_{\mathrm{d}}(f) \textrm{ if  d  is  odd} \\
(-1)^{\frac{d}{2}\cdot\frac {n+2}{2}}\cdot\mathrm{disc}_{\mathrm{d}}(f) \textrm{ if  d  is  even}
\end{cases}$$
and

$$\delta^2_q(\mathfrak{P}_M)=w_2(q_{dR})+\begin{cases}
\frac{d-1}{2}(-1, -1) \textrm{ if  d  is  odd} \\
\frac{n}{4}(1+\frac{d}{2})(-1, -1)\textrm{ if  d  is  even  and $n\equiv 0$   $\mathrm{ mod }$ $4$}\\
\left(-1,\mathrm{disc}_{\mathrm{d}}(f)\right)+(\frac{n+2}{4})(1+\frac{d}{2})(-1, -1)\ \textrm{ if  d  is  even  and $n\equiv 2$   $\mathrm{ mod }$ $4$.}
\end{cases}$$
\end{cor}
\begin{proof} The first assertion follows from \cite{Saito11} and  Theorem \ref{cor-dR-B}(2).  The second assertion follows from Theorem \ref{cor-dR-B}(1) and Lemma \ref{Hasse}.
\end{proof}

\begin{rem} Let us assume for simplicity that $n \equiv\ 0\ \mathrm{mod}\ 4$. Following a conjecture of Saito \cite{Saito12}, proved for  $l>n+1$,   there should exist integers $\alpha(n, d)$ and $\beta (n, d)$, which can be explicitely expressed in terms of  $n$ and $d$,  such  that for any prime number $l$, one has 
$$ \delta^2_q(\mathfrak{P}_M)=sw_2(\rho_l)+\alpha(n, d)c_l+\beta(n,d)(-1, -1)+(2, d_V), $$
where $c_l$ is the unique element of $H^2(\mathbb Q, \mathbb Z/2\mathbb Z)$ which ramifies exactly in $l$ and $\infty$. 

\end{rem}

\subsection{The lifting problem $\mathrm{Lift}(M,q)$ over $\mathbb{R}$}
Let $(M, q)$ be an orthogonal $\mathbb{Q}$-motive, let $(M_B, q_B)$ be its Betti realization and let $\infty$ be the real place of $\mathbb{Q}$. Recall that 
$(M_{\infty}, q_{\infty}):= (M_B,q_B)\otimes_{\mathbb{Q}}\mathbb{R}$ and that we have an orthogonal representation 
$\rho_{\infty\mid G_{\mathbb{R}}}: G_{\mathbb{R}}\rightarrow {\bf O}(M_{\infty})$. Under this action $M_{\infty}$ decomposes into a direct sum 
$M_{\infty}=M_{\infty}^+\oplus M_{\infty}^{-}$ where the complex conjugation (i.e. the  non-trivial element in $G_{\mathbb{R}}$) acts as the identity $\mathrm{Id}$ (resp. as $- \mathrm{Id}$) on $M_{\infty}^+$ (resp. 
on $M_{\infty}^{-}$). Our aim is to compute the local invariants $sw_2(\rho_{\infty\mid G_{\mathbb{R}}})$ and $sp_2(\rho_{\infty\mid G_{\mathbb{R}}})$, which belong to the group $H^2(G_{\mathbb{R}}, {\mathbb{Z}}/2{\mathbb{Z}})\simeq {\mathbb{Z}}/2{\mathbb{Z}}$.

\begin{prop}\label{propR} Let $b^{-}$ be the dimension of $M_{\infty}^{-}$. Then we have 
$$ sw_2(\rho_{\infty\mid G_{\mathbb{R}}})={\frac{b^{-}(b^{-}-1)}{2}}\ \ \mbox{and}\ \ sp_2(\rho_{\infty\mid G_{\mathbb{R}}})=0. $$
\end{prop}
\begin{proof} In order to ease the notation, we write $\rho$ for $\rho_{\infty\mid G_{\mathbb{R}}}$. By choosing  orthogonal bases of $M_{\infty}^+$  and  $M_{\infty}^{-}$,  we decompose    $(M_{\infty}, q_{\infty})$ into an orthogonal  sum of dimension $1$ quadratic spaces,   stable under the action of $G_{\mathbb{R}}$. Therefore the orthogonal character $\chi_{\rho}$ of $\rho$ can be written as a sum of degree one orthogonal characters, 
\begin{equation}\label{orth}
\chi_{\rho}=b^+\chi_0+b^{-}\chi_1, 
\end{equation} 
where $\chi_0$ (resp. $\chi_1$) is the trivial (resp. non trivial) character of $G_{\mathbb{R}}$ and $b^+ $   is the dimension of $M_{\infty}^+$. The Stiefel-Whitney class of an orthogonal representation only depends on  the character of the representation. Moreover we obtain from  \cite{Fr\"ohlich85} the following addition formula.
\begin{lem} 
Let $\rho$ and $\rho'$ be orthogonal representations of a finite group $\Gamma$, of characters $\chi_{\rho}$ and $\chi_{\rho'}$. Then we have
$$sw_2(\chi_{\rho}+ \chi_{\rho'})=sw_2(\chi_{\rho})+sw_2(\chi_{\rho'}) + \mathrm{det}(\chi_{\rho})\cdot \mathrm{det}\  (\chi_{\rho'}).$$
Moreover $sw_2(\chi)=0$ for   $\chi \in \mathrm{Hom}(\Gamma, {\mathbb{Z}}/2{\mathbb{Z}})$.
\end{lem}
It follows from the lemma that $sw_2(\rho)=sw_2(\chi_{\rho})=sw_2(b^{-}\chi_1)$ and that 
$$sw_2(a\chi_1)= sw_2((a-1)\chi_1)+ (\chi_1)( \chi_1^{a-1})=sw_2((a-1)\chi_1)+(a-1)(\chi_1)(\chi_1),\ \forall a. $$
Therefore $sw_2(\rho)=\frac{b^{-}(b^{-}-1)}{2}(\chi_1)(\chi_1)$.  Since $(\chi_1)(\chi_1)$ identifies with the class of the Hamilton quaternion algebra over $\mathbb R$, under the group monomorphism 
$H^2(G_{\mathbb{R}}, {\mathbb{Z}}/2{\mathbb{Z}})\rightarrow Br(\mathbb R)$, we conclude that   $(\chi_1)( \chi_1)$  is the non trivial element  of $H^2(G_{\mathbb{R}},{\mathbb{Z}}/2{\mathbb{Z}})$. 

We now recall the definition of $sp_2(t)$ for an orthogonal representation of $t:G_{\mathbb{R}}\rightarrow {\bf O}(q_t)$ as given in \cite{Fr\"ohlich85}.  We let $sp: {\bf O}(q_t)\rightarrow 
\mbox{Hom} (G_{\mathbb{R}}, {\mathbb{Z}}/2{\mathbb{Z}})$ 
be the group homomorphism obtained by composing the spinor norm ${\bf O}(q_t) \rightarrow \Bbb R^{\times} /\Bbb R^{\times2}$  with the canonical isomorphim $\Bbb R^{\times} /\Bbb R^{\times2}\simeq \mbox{Hom} (G_{\mathbb{R}}, {\mathbb{Z}}/2{\mathbb{Z}})$. Let  $sp[t]: G_{\mathbb{R}}\rightarrow \mbox{Hom}(G_{\mathbb{R}}, {\mathbb{Z}}/2{\mathbb{Z}})$ be the group homomorphism $g\rightarrow sp(t(g))$. This map satisfies the relations 
\begin{equation}\label{sp1} 
sp[t+t']=sp[t]+sp[t'] \ \ (\mbox{addition  in}\ \  \mbox {Hom}(G_{\mathbb{R}}, \mbox{Hom} (G_{\mathbb{R}}, {\mathbb{Z}}/2{\mathbb{Z}}))
\end{equation}
and 
\begin{equation}\label{sp2} 
sp[\phi]=0\ \ \forall\  \phi\ \in  \mathrm{Hom}(G_{\mathbb{R}}, \pm 1)
\end{equation}
The map $sp[t]$ gives rise to a bilinear form 
$$c_t: G_{\mathbb{R}}\times G_{\mathbb{R}}\rightarrow  {\mathbb{Z}}/2{\mathbb{Z}}. $$
It   is a $2$-cocycle  which represents  $sp_2(t)$  in $H^2(G_{\mathbb{R}},{\mathbb{Z}}/2{\mathbb{Z}})$. 
The triviality of $sp_2(\rho)$ now follows from  (\ref{orth}), (\ref{sp1}) and (\ref{sp2}). 
\end{proof}

\begin{rem}
The lifting problem $\mathrm{Lift}(M,q)$ has a solution over $\mathbb{R}$ if and only if
$\frac{b^{-}(b^{-}-1)}{2}$ is even.
\end{rem}

\section{Appendix}

\subsection{The K\"unneth formula for effective Nori motives}\label{sect-appendix}

The aim of this section is to prove a K\"unneth decomposition  for effective Nori motives. For generalities on Nori motives, we refer to \cite{Nori}, \cite{Levine},  \cite{Huber-Muller-Stach-12} and \cite{Huber-Muller-Stach-15}.

Following Nori's construction  we attach to any diagram $D$ and any representation  $T: D\rightarrow \mathrm{Vec}_\mathbb Q$  a $\mathbb Q$-linear abelian category $C(D, T)$, a representation $\tilde T:D\rightarrow C(D, T)$ and a forgetful functor $f_T$ such that $f_T\circ \tilde T=T$ and $\tilde T$ is universal for this property. For any such pair $(D, T)$, we may consider the pair $(D^{op}, T^\vee)$ where $D^{op}$ is the opposite diagram and $T^\vee: D^{op} \rightarrow \mathrm{Vec}_\mathbb Q$  assigns to any vertex $p$ of $D$ the dual $\mathbb{Q}$-vector space of $T(p)$ and to any edge $m$ from $p$ to $q$ the transpose of $T(m)$. Let $k$ be a subfield of the field of complex numbers $\mathbb{C}$. The category of effective homological Nori motives (respectively of effective cohomological Nori motives) is defined as 
$EHM:= C(D, H_*)$ (resp. $ECM:= C(D^{op}, H_*^{\vee})$), where $D$ is the diagram of effective good pairs consisting of triples $(X, Y, i)$ where $X$ is a $k$-variety, $Y\hookrightarrow X$ a closed subvariety and $i$ an integer such that the singular homology $H_j(X(\mathbb C), Y(\mathbb C), \mathbb Q)=0$ if $j\neq i$, and  $H_*$ is the representation which maps $(X, Y, i)$ to the relative singular homology $H_i(X(\mathbb C), Y(\mathbb C), \mathbb Q)$.  For a good pair  $(X, Y, i)$ we write $h_i(X, Y)$ (resp. $h^i(X, Y))$ for the corresponding object in $EHM$ (resp. $ECM$).

We denote by $\mathrm{Ind}(EHM)$ the category of ind-objects of $EHM$,  i.e. the strictly full subcategory of the category of presheaves of sets on $EHM$ which are isomorphic to filtered colimits of representables. Then $\mathrm{Ind}(EHM)$ is an abelian category with enough injectives, and the Yoneda functor $y:EHM\rightarrow \mathrm{Ind}(EHM)$ is an  exact fully faithful functor. An object of $\mathrm{Ind}(EHM)$ lying in the essential image of $y$ is called essentially constant. Then we consider the abelian category $\mathbf{Ch}(\mathrm{Ind}(EHM))$ (resp. $\mathbf{Ch}_+(\mathrm{Ind}(EHM))$) of (resp. bounded below) homological chain complexes in $\mathrm{Ind}(EHM)$.

Let $X$ be an affine variety over $k$ (i.e. an affine scheme of finite type over $k$). A good filtration of $X$ is a sequence of closed subvarieties
$$\emptyset=X_{-1}\subset X_0 \subset ...\subset X_{n-1}\subset X_n=X$$
such that $(X_j, X_{j-1}, j)$ is a good pair for $0\leq j\leq n$. Nori's basic Lemma ensures that the set of good filtrations is non-empty and filtered. It is moreover functorial in the sense that given a morphism $f:X\rightarrow Y$ of affine varieties, there exist good filtrations $X_{*}$ on $X$ and $Y_*$ on $Y$ such that $f(X_i)\subset Y_i$ for any $i$. We set
$$C_*(X_*):=[h_n(X_n,X_{n-1})\rightarrow h_{n-1}(X_{n-1},X_{n-2})\rightarrow h_0(X_0,\emptyset)]\in \mathbf{Ch}_+(EHM)$$
where $h_0(X_0,\emptyset)$ sits in degree $0$. We obtain a functor
$$\fonc{C_*}{\mathbf{Aff}_k} {\mathbf{Ch}_+(\mathrm{Ind}(EHM))}{X}{\underrightarrow{lim}\,\, C_*(X_*)}$$
where the colimit is taken over the system of good filtrations $X_*$ of $X$, and  $\mathbf{Aff}_k$ denotes the category of affine $k$-schemes of finite type.

Let $X$ be any variety (i.e. a separated scheme of finite type over $k$). A \emph{rigidified affine cover}  $(\mathcal{U},r)$ is given by a finite open affine cover $\mathcal{U}=\{U_i\subset X,i\in I\}$ of $X$, and a \emph{surjective} map of sets
$r:X(\mathbb{C})\twoheadrightarrow I$ such that $x\in U_{r(x)}(\mathbb{C})$ for any $x\in X(\mathbb{C})$,
where $X(\mathbb{C})$ denotes the (discrete) set of complex points of $X$. A morphism $(X,\mathcal{U},r)\rightarrow (X',\mathcal{U}',r')$ of varieties with rigidified affine covers is a pair $(f,\phi)$ given by a morphism of schemes $f:X\rightarrow X'$ inducing a map of sets $\phi:I\rightarrow I'$ (i.e. such that the obvious square commutes) such that $f(U_i)\subseteq U'_{\phi(i)}$ for every $i\in I$. Note that, if $(\mathcal{U},r)$ and $(\mathcal{U'},r')$ are two rigidified affine covers of $X$, 
there exist at most one morphism $(X,\mathcal{U'},r')\rightarrow (X,\mathcal{U},r)$ whose underlying map $X\rightarrow X$ is the identity. If such a morphism exists, we say that $(\mathcal{U}',r')$ refines $(\mathcal{U},r)$.

Given a variety $X$ endowed with a rigidified affine cover $(\mathcal{U},r)$, we consider the simplicial Cech complex $\check{C}_{\bullet}(\mathcal{U})$, a simplicial affine scheme.  Applying the functor $C_*$ above, we get a simplicial object $C_*\check{C}_{\bullet}(\mathcal{U})$ of $\mathbf{Ch}_+(\mathrm{Ind}(EHM))$, and taking the associated complex we get a double complex in $\mathrm{Ind}(EHM)$. We denote by $\mathrm{Toc}(C_*\check{C}_{\bullet}(\mathcal{U}))\in \mathbf{Ch}_+(\mathrm{Ind}(EHM))$ the total complex. In other words, we consider the following functor
\begin{equation}\label{functorToc}
\mathrm{Toc}:\mathrm{Fun}(\Delta^{op},\mathbf{Ch}_+(\mathrm{Ind}(EHM)))\longrightarrow \mathbf{Ch}_+(\mathbf{Ch}_+(\mathrm{Ind}(EHM)))\stackrel{\mathrm{Tot}}{\longrightarrow}\mathbf{Ch}_+(\mathrm{Ind}(EHM))
\end{equation}
where the first functor sends a simplicial object to the corresponding unnormalized homological complex and the second functor sends a double complex to its total complex. Then $\mathrm{Toc}(C_*\check{C}_{\bullet}(\mathcal{U}))$ is functorial in $(\mathcal{U},r)$. Moreover, if $(\mathcal{V},s)$ refines $(\mathcal{U},r)$
then 
$$\mathrm{Toc}(C_*\check{C}_{\bullet}(\mathcal{V}))\rightarrow \mathrm{Toc}(C_*\check{C}_{\bullet}(\mathcal{U}))$$
is a quasi-isomorphism. Hence $\mathrm{Toc}(C_*\check{C}_{\bullet}(\mathcal{U}))\in \mathbf{D}_b(\mathrm{Ind}(EHM))$ does not depend on $(\mathcal{U},r)$ up to a canonical isomorphism. Here $\mathbf{D}_b(\mathrm{Ind}(EHM))$ of course denotes the derived category of bounded homological complexes in $\mathrm{Ind}(EHM)$.

This defines a functor
$$\fonc{\tilde{M}}{\mathbf{Var}_k} {\mathbf{D}_b(\mathrm{Ind}(EHM))}{X}{\mathrm{Tot}(C_*\check{C}_{\bullet}(\mathcal{U}))}$$
The functor
$$\mathbf{D}_b(EHM)\rightarrow\mathbf{D}_b(\mathrm{Ind}(EHM))$$ is fully faithful and its essential image $\tilde{\mathbf{D}}_b(EHM)$ is precisely the full subcategory consisting in chain complexes with essentially constant homology. Note that $\tilde{\mathbf{D}}_b(EHM)$ is strictly larger than $\mathbf{D}_b(EHM)$ but these two categories are canonically equivalent; we make no distinction between them. Let us denote by $\mathbf{Var}_k$  the category of varieties over $E$. We obtain a functor
$$M:\mathbf{Var}_k\rightarrow \mathbf{D}_b(EHM)$$
such that the image of $M(X)$ in $\mathbf{D}_b(\mathrm{Vec}_\mathbb Q)$ computes the singular homology of $X$. 

\begin{notation}\label{notdansehm}
For any variety $X/k$ we define the object $h_i(X)$ of EHM by 
$$h_i(X)=H_i(M(X)).$$ 
\end{notation}
 
 The tensor product on  $EHM$ extends to a tensor product on $\mathrm{Ind}(EHM)$. Any object of $EHM$ is flat (since $EHM$ has a fiber functor to the category of finite dimensional $\mathbb{Q}$-vector spaces) hence so is any object of $\mathrm{Ind}(EHM)$. Therefore, derived tensor products on $\mathbf{D}_b(EHM)$ and $\mathbf{D}_b(\mathrm{Ind}(EHM))$ are well defined and $\mathbf{D}_b(EHM)\rightarrow\mathbf{D}_b(\mathrm{Ind}(EHM))$ is monoidal.

\begin{thm}
We have a functorial isomorphism in $\mathbf{D}_b(EHM)$ 
$$M(X)\otimes^L M(Y)\stackrel{\sim}{\longrightarrow }M(X\times_kY)$$
inducing the usual K\"unneth isomorphism on singular homology.
\end{thm}

\begin{proof} Since the fully faithful functor $\mathbf{D}_b(EHM)\rightarrow\mathbf{D}_b(\mathrm{Ind}(EHM))$ commutes with $\otimes^L$, it is enough to prove the result in $\mathbf{D}_b(\mathrm{Ind}(EHM))$. Given arbitrary varieties $X$ and $Y$ endowed with rigidified open affine covers $(\mathcal{U},r)$ and $(\mathcal{V},s)$, the complexes $\mathrm{Toc}(C_*\check{C}_{\bullet}(\mathcal{U}))$ and $\mathrm{Toc}(C_*\check{C}_{\bullet}(\mathcal{V}))$ in $\mathbf{Ch}_+(\mathrm{Ind}(EHM))$ represent $M(X)$ and $M(Y)$. Hence $M(X)\otimes^L M(Y)$ is represented by the total complex 
$$\mathrm{Tot}(\mathrm{Toc}(C_*\check{C}_{\bullet}(\mathcal{U}))\otimes \mathrm{Toc}(C_*\check{C}_{\bullet}(\mathcal{V})))\in \mathbf{Ch}_+(\mathrm{Ind}(EHM))$$ associated with the double complex $\mathrm{Toc}(C_*\check{C}_{\bullet}(\mathcal{U}))\otimes \mathrm{Toc}(C_*\check{C}_{\bullet}(\mathcal{V})))$. It is therefore enough to exhibit a quasi-isomorphism
\begin{equation}\label{finalqis}
\mathrm{Tot}(\mathrm{Toc}(C_*\check{C}_{\bullet}(\mathcal{U}))\otimes \mathrm{Toc}(C_*\check{C}_{\bullet}(\mathcal{V})))\stackrel{\sim}{\longrightarrow} \mathrm{Toc}(C_*\check{C}_{\bullet}(\mathcal{U}\times \mathcal{V}))
\end{equation}
which is functorial in $(\mathcal{U},r)$, $(\mathcal{V},s)$, $X$ and $Y$. Here $(\mathcal{U}\times \mathcal{V},r\times s)$ is the obvious rigidified open affine cover of $X\times Y$.

We first consider the affine case. Consider the following diagram of functors:
\[ \xymatrix{
\mathbf{Aff}_k\times \mathbf{Aff}_k\ar[d]^{-\times-}\ar[r]^{C_*\times C_*\hspace{2.3cm}}
&\mathbf{Ch}_+(\mathrm{Ind}(EHM))\times\mathbf{Ch}_+(\mathrm{Ind}(EHM))\ar[d]^{\mathrm{Tot}(-\otimes-)}\\
\mathbf{Aff}_k\ar[r]^{C_*}
&\mathbf{Ch}_+(\mathrm{Ind}(EHM))}
\]
\begin{lem}\label{lemaffine} (Nori, \cite{Nori} Section 4.5)
The square above is quasi-commutative in the sense that there is a natural transformation 
$$\mathrm{Tot}(C_*(-)\otimes C_*(-))\longrightarrow C_*(-\times -)$$
 inducing a quasi-isomorphism
$$\mathrm{Tot}(C_*(X)\otimes C_*(Y))\rightarrow C_*(X\times Y) $$
in $\mathbf{Ch}(\mathrm{Ind}(EHM))$ for every $X, Y\in \mathbf{Aff}_k$.
\end{lem}
\begin{proof}
We have $\mathrm{Tot}(C_*(X)\otimes C_*(Y))\simeq \underrightarrow{lim} \,\mathrm{Tot}(C_*(X_*)\otimes C_*(Y_*))$ where the colimit is taken over the system of good filtrations $X_*$ and $Y_*$ on $X$ and $Y$ respectively. Given such good filtrations $X_*$ and $Y_*$, there exists a good filtration $Z_*$ on $X\times Y$ such that $\cup_{i+j=k}X_i\times Y_j\subseteq Z_k$. 
The obvious map from
$$\mathrm{Tot}_k(C_*(X_*)\otimes C_*(Y_*))=\bigoplus_{i+j=k}h_i(X_i,X_{i-1})\otimes h_j(Y_j,Y_{j-1})\simeq
\bigoplus_{i+j=k}h_i(X_i\times Y_j,X_{i}\times Y_{j-1}\cup X_{i-1}\times Y_j).$$
to $C_k(Z_*)=H_k(Z_k,Z_{k-1})$
gives morphisms of complexes
$$\mathrm{Tot}(C_*(X_*)\otimes C_*(Y_*))\rightarrow C_*(Z_*)\rightarrow C_*(X\times Y)$$
which does not depend on the choice of the filtration $Z_*$. This induces a quasi-isomorphism 
$$\mathrm{Tot}(C_*(X)\otimes C_*(Y))\rightarrow C_*(X\times Y)$$
which is moreover functorial in $X$ and $Y$. This proves the lemma and thus the theorem when $X$ and $Y$ are both affine. 

\end{proof}
Take now arbitrary varieties $X$ and $Y$, and choose rigidified open affine covers $(\mathcal{U},r)$ and $(\mathcal{V},s)$, with $\mathcal{U}=\{U_i\subset X,\,i\in I\}$ and $\mathcal{V}=\{V_j\subset Y,\,j\in J\}$. We denote by $\check{C}_{\bullet}(\mathcal{U}):\Delta^{op}\rightarrow \mathbf{Aff}_k$ and $\check{C}_{\bullet}(\mathcal{V}):\Delta^{op}\rightarrow \mathbf{Aff}_k$ the corresponding simplicial affine schemes. Consider the diagram
\[ \xymatrix{
\Delta^{op}\times\Delta^{op} \ar[d]^{\mathrm{Id}}\ar[r]^{\check{C}_{\bullet}(\mathcal{U})\times \check{C}_{\bullet}(\mathcal{V})}&  \mathbf{Aff}_k\times \mathbf{Aff}_k\ar[d]^{-\times-}\ar[r]^{C_*\times C_*\hspace{2cm}}
&\mathbf{Ch}_+(\mathrm{Ind}(EHM))\times\mathbf{Ch}_+(\mathrm{Ind}(EHM))\ar[d]^{\mathrm{Tot}(-\otimes-)}\\
\Delta^{op}\times\Delta^{op}  \ar[r]& \mathbf{Aff}_k\ar[r]^{C_*}
&\mathbf{Ch}_+(\mathrm{Ind}(EHM))}
\]
where the left square commutes. The bottom line gives a bisimplicial object of $\mathbf{Ch}_+(\mathrm{Ind}(EHM))$, which we denote by
$$C_*(\check{C}_{\bullet}(\mathcal{U})\times \check{C}_{\bullet}(\mathcal{V})):\Delta^{op}\times\Delta^{op}\longrightarrow \mathbf{Ch}_+(\mathrm{Ind}(EHM)).$$
Composing the top line with $\mathrm{Tot}(-\otimes-)$ gives another bisimplicial object of $\mathbf{Ch}_+(\mathrm{Ind}(EHM))$, which we denote by
$$\mathrm{Tot}(C_*\check{C}_{\bullet}(\mathcal{U})\otimes C_*\check{C}_{\bullet}(\mathcal{V}))):\Delta^{op}\times\Delta^{op}\longrightarrow \mathbf{Ch}_+(\mathrm{Ind}(EHM)).$$
We denote by 
\begin{equation}\label{doubles}
\mathrm{T}(C_*(\check{C}_{\bullet}(\mathcal{U})\times \check{C}_{\bullet}(\mathcal{V})))\mbox{ and }\mathrm{T}(\mathrm{Tot}(C_*\check{C}_{\bullet}(\mathcal{U})\otimes C_*\check{C}_{\bullet}(\mathcal{V}))))\in \mathbf{Ch}_+(\mathbf{Ch}_+(\mathrm{Ind}(EHM)))
\end{equation} the complexes of objects of $\mathbf{Ch}_+(\mathrm{Ind}(EHM))$ associated with
these bisimplicial objects of $\mathbf{Ch}_+(\mathrm{Ind}(EHM))$. More precisely, if $A_{\bullet\bullet}$ is a bisimplicial object in an abelian category $\mathcal{A}$, we denote by $\mathrm{C}(A_{\bullet\bullet})$ the double complex associated with $A_{\bullet\bullet}$ and by $\mathrm{T}(A_{\bullet\bullet})$ its total complex, so that $\mathrm{T}(A_{\bullet\bullet})_n=\bigoplus_{p+q=n}A_{p,q}$. In our situation $\mathcal{A}=\mathbf{Ch}_+(\mathrm{Ind}(EHM))$.

We can see (\ref{doubles}) as double complexes and
consider their total complexes
$$\mathrm{Tot}(\mathrm{T}(C_*(\check{C}_{\bullet}(\mathcal{U})\times \check{C}_{\bullet}(\mathcal{V}))))\mbox{ and }\mathrm{Tot}(\mathrm{T}(\mathrm{Tot}(C_*\check{C}_{\bullet}(\mathcal{U})\otimes C_*\check{C}_{\bullet}(\mathcal{V})))))\in \mathbf{Ch}_+(\mathrm{Ind}(EHM)).$$

\begin{lem}
There is a canonical quasi-isomorphism
$$\mathrm{Tot}(\mathrm{T}(\mathrm{Tot}(C_*\check{C}_{\bullet}(\mathcal{U})\otimes C_*\check{C}_{\bullet}(\mathcal{V}))))\stackrel{\mathcal{N}}{\longrightarrow}
\mathrm{Tot}(\mathrm{T}(C_*(\check{C}_{\bullet}(\mathcal{U})\times \check{C}_{\bullet}(\mathcal{V})))).$$
\end{lem}
\begin{proof}
Consider $$\mathrm{T}:\mathrm{Fun}(\Delta^{op}\times\Delta^{op},\mathbf{Ch}_+(\mathrm{Ind}(EHM)))\longrightarrow \mathbf{Ch}_+(\mathbf{Ch}_+(\mathrm{Ind}(EHM)))$$ 
as a functor from bisimplicial objects of $\mathbf{Ch}_+(\mathrm{Ind}(EHM))$ to $\mathbf{Ch}_+(\mathbf{Ch}_+(\mathrm{Ind}(EHM)))$. The natural transformation of Lemma \ref{lemaffine} gives  a morphism of bisimplicial objects: 
$$\mathrm{Tot}(C_*\check{C}_{\bullet}(\mathcal{U})\otimes C_*\check{C}_{\bullet}(\mathcal{V}))\longrightarrow 
C_*(\check{C}_{\bullet}(\mathcal{U})\times \check{C}_{\bullet}(\mathcal{V})).$$
Applying the functor $\mathrm{T}$, we get a morphism
$$D_*=\mathrm{T}(\mathrm{Tot}(C_*\check{C}_{\bullet}(\mathcal{U})\otimes C_*\check{C}_{\bullet}(\mathcal{V})))\longrightarrow 
\mathrm{T}(C_*(\check{C}_{\bullet}(\mathcal{U})\times \check{C}_{\bullet}(\mathcal{V})))=D'_*$$
in $\mathbf{Ch}_+(\mathbf{Ch}_+(\mathrm{Ind}(EHM)))$. By Lemma \ref{lemaffine}, 
the morphism in $\mathbf{Ch}_+(\mathrm{Ind}(EHM))$ 
$$D_n=\bigoplus_{p+q=n}\mathrm{Tot}(C_*\check{C}_{p}(\mathcal{U})\otimes C_*\check{C}_{q}(\mathcal{V}))
\longrightarrow \bigoplus_{p+q=n} C_*(\check{C}_{p}(\mathcal{U})\times \check{C}_{q}(\mathcal{V}))=D'_n$$ 
is a quasi-isomorphism for any $n\geq0$. The result then follows from a standard spectral sequence argument.
\end{proof}
\begin{lem}
There is a canonical isomorphism
$$\mathrm{Tot}(\mathrm{Toc}(C_*\check{C}_{\bullet}(\mathcal{U}))\otimes \mathrm{Toc}(C_*\check{C}_{\bullet}(\mathcal{V})))\simeq\mathrm{Tot}(\mathrm{T}(\mathrm{Tot}(C_*\check{C}_{\bullet}(\mathcal{U})\otimes C_*\check{C}_{\bullet}(\mathcal{V})))).$$
\end{lem}
\begin{proof}
These two complexes can be both identified with the total complex associated with the quadruple complex whose generic term is
$$(C_p\check{C}_{q}(\mathcal{U})\otimes C_u\check{C}_{v}(\mathcal{V}))_{p,q,u,v\in\mathbb{N}^4}.$$
Indeed, we have
\begin{eqnarray*}
\mathrm{Tot}(\mathrm{T}(\mathrm{Tot}(C_*\check{C}_{\bullet}(\mathcal{U})\otimes C_*\check{C}_{\bullet}(\mathcal{V}))))_t
&=&  \bigoplus_{n+m=t}(\mathrm{T}(\mathrm{Tot}(C_*\check{C}_{\bullet}(\mathcal{U})\otimes C_*\check{C}_{\bullet}(\mathcal{V})))_{n})_m\\
&=&\bigoplus_{n+m=t}\bigoplus_{p+q=n} \mathrm{Tot}(C_*\check{C}_{p}(\mathcal{U})\otimes C_*\check{C}_{q}(\mathcal{V}))_m\\
&=&\bigoplus_{n+m=t}\bigoplus_{p+q=n}\bigoplus_{u+v=m} C_u\check{C}_{p}(\mathcal{U})\otimes C_v\check{C}_{q}(\mathcal{V}))\\
&=&\bigoplus_{u+v+p+q=t}C_u\check{C}_{p}(\mathcal{U})\otimes C_v\check{C}_{q}(\mathcal{V}),
\end{eqnarray*}
and
\begin{eqnarray*}
\mathrm{Tot}(\mathrm{Toc}(C_*\check{C}_{\bullet}(\mathcal{U}))\otimes \mathrm{Toc}(C_*\check{C}_{\bullet}(\mathcal{V})))_t&=&
\bigoplus_{n+m=t}\mathrm{Toc}(C_*\check{C}_{\bullet}(\mathcal{U}))\otimes \mathrm{Toc}(C_*\check{C}_{\bullet}(\mathcal{V}))_{n,m}\\
&=& \bigoplus_{n+m=t}\mathrm{Toc}(C_*\check{C}_{\bullet}(\mathcal{U}))_n\otimes \mathrm{Toc}(C_*\check{C}_{\bullet}(\mathcal{V}))_m\\
&=& \bigoplus_{n+m=t}(\bigoplus_{p+q=n} C_p\check{C}_{q}(\mathcal{U}))\otimes (\bigoplus_{u+v=m}C_u\check{C}_{v}(\mathcal{V}))\\
&=& \bigoplus_{n+m=t}\bigoplus_{p+q=n} \bigoplus_{u+v=m} C_p\check{C}_{q}(\mathcal{U})\otimes C_u\check{C}_{v}(\mathcal{V}) \\
&=&\bigoplus_{u+v+p+q=t} C_p\check{C}_{q}(\mathcal{U})\otimes C_u\check{C}_{v}(\mathcal{V}).
\end{eqnarray*}
\end{proof}

Consider now the diagonal simplicial object associated with the  bisimplicial object $C_*(\check{C}_{\bullet}(\mathcal{U})\times \check{C}_{\bullet}(\mathcal{V}))$, which we denote by 
$$\mathrm{diag}(C_*(\check{C}_{\bullet}(\mathcal{U})\times \check{C}_{\bullet}(\mathcal{V})))
:\Delta^{op}\longrightarrow \Delta^{op}\times\Delta^{op}\longrightarrow \mathbf{Ch}_+(\mathrm{Ind}(EHM)).$$
We denote by $\mathrm{D}(C_*(\check{C}_{\bullet}(\mathcal{U})\times \check{C}_{\bullet}(\mathcal{V})))\in \mathbf{Ch}_+(\mathbf{Ch}_+(\mathrm{Ind}(EHM)))$ the associated (unnormalized) homological complex, so that
$$\mathrm{D}(C_*(\check{C}_{\bullet}(\mathcal{U})\times \check{C}_{\bullet}(\mathcal{V})))_n=C_*(\check{C}_{n}(\mathcal{U})\times \check{C}_{n}(\mathcal{V}))).$$

\begin{lem} (Eilenberg-Zilber Theorem) There is a canonical quasi-isomorphism
$$\mathrm{Tot}(\mathrm{T}(C_*(\check{C}_{\bullet}(\mathcal{U})\times \check{C}_{\bullet}(\mathcal{V}))) )\stackrel{\nabla}\longrightarrow
\mathrm{Tot}(\mathrm{D}(C_*(\check{C}_{\bullet}(\mathcal{U})\times \check{C}_{\bullet}(\mathcal{V})))).$$
\end{lem}
\begin{proof}
We consider the bisimplicial object $C_*(\check{C}_{\bullet}(\mathcal{U})\times \check{C}_{\bullet}(\mathcal{V}))$ in the abelian category $\mathcal{A}:=\mathbf{Ch}_+(\mathrm{Ind}(EHM))$. By the Eilenberg-Zilber Theorem (\cite{Weibel94} 8.5.1), the shuffle map (see \cite{Weibel94} 8.5.4) 
$$\mathrm{T}(C_*(\check{C}_{\bullet}(\mathcal{U})\times \check{C}_{\bullet}(\mathcal{V})) )\longrightarrow
\mathrm{D}(C_*(\check{C}_{\bullet}(\mathcal{U})\times \check{C}_{\bullet}(\mathcal{V}))),$$
a map in $\mathbf{Ch}_+(\mathcal{A})=\mathbf{Ch}_+(\mathbf{Ch}_+(\mathrm{Ind}(EHM)))$,
is a quasi-isomorphism. We conclude with a spectral sequence argument.

\end{proof}

We consider the rigidified open affine cover $(\mathcal{U}\times \mathcal{V},r\times s)$  of $X\times Y$, where $\mathcal{U}\times \mathcal{V}:=\{U_i\times V_j\subset X\times Y, (i,j)\in I\times J\}$.
\begin{lem}
There is a canonical isomorphism 
$$\mathrm{Tot}(\mathrm{D}(C_*(\check{C}_{\bullet}(\mathcal{U})\times \check{C}_{\bullet}(\mathcal{V}))))\simeq \mathrm{Toc}(C_*(\check{C}_{\bullet}(\mathcal{U}\times \mathcal{V})))$$
\end{lem}
\begin{proof}
The result follows from the fact that there is a canonical isomorphism of simplicial object of $\mathbf{Ch}_+(\mathrm{Ind}(EHM))$:
\begin{equation}\label{simpliso}
\mathrm{diag}(C_*(\check{C}_{\bullet}(\mathcal{U})\times \check{C}_{\bullet}(\mathcal{V})))= C_*(\check{C}_{\bullet}(\mathcal{U}\times \mathcal{V})).
\end{equation}
Indeed, we have
\begin{eqnarray*}
\mathrm{diag}(C_*(\check{C}_{\bullet}(\mathcal{U})\times \check{C}_{\bullet}(\mathcal{V})))_n
&=&  
C_*(\check{C}_{n}(\mathcal{U})\times \check{C}_{n}(\mathcal{V}))\\
&=&C_*(\bigsqcup_{(i_0,...,i_n)}U_{i_0}\cap...\cap U_{i_n}\times \bigsqcup_{(j_0,...,j_n)}V_{j_0}\cap...\cap V_{j_n})\\
&\simeq&
C_*(\bigsqcup_{((i_0,j_0),...,(i_n,j_n))}(U_{i_0}\times V_{j_0})\cap...\cap (U_{i_n}\times V_{j_n})\\
&=&C_*(\check{C}_{n}(\mathcal{U}\times \mathcal{V})).
\end{eqnarray*}
The desired isomorphism is obtained by applying the functor (\ref{functorToc}) to the isomorphism of simplicial objects (\ref{simpliso}).

\end{proof}

We obtain (\ref{finalqis}) by composing the quasi-isomorphisms of the previous lemmas.  More precisely, we have the quasi-isomorphism
\begin{eqnarray*}
\mathrm{Tot}(\mathrm{Toc}(C_*\check{C}_{\bullet}(\mathcal{U}))\otimes \mathrm{Toc}(C_*\check{C}_{\bullet}(\mathcal{V})))&\simeq&\mathrm{Tot}(\mathrm{T}(\mathrm{Tot}(C_*\check{C}_{\bullet}(\mathcal{U})\otimes C_*\check{C}_{\bullet}(\mathcal{V}))))\\
&\stackrel{\mathcal{N}}{\longrightarrow}&\mathrm{Tot}(\mathrm{T}(C_*(\check{C}_{\bullet}(\mathcal{U})\times \check{C}_{\bullet}(\mathcal{V}))))\\
&\stackrel{\nabla}{\longrightarrow}&\mathrm{Tot}(\mathrm{D}(C_*(\check{C}_{\bullet}(\mathcal{U})\times \check{C}_{\bullet}(\mathcal{V}))))\\
&\simeq& \mathrm{Toc}(C_*(\check{C}_{\bullet}(\mathcal{U}\times \mathcal{V})))
\end{eqnarray*}
Note that $\mathcal{N}$ is induced by Nori's Lemma \ref{lemaffine} (i.e. the K\"unneth formula for affine varieties) and that 
$\nabla$ is induced by the Eilenberg-Zilber shuffle map. This quasi-isomorphism is functorial in $(X,\mathcal{U},r)$ and $(Y,\mathcal{V},s)$ by construction. It induces a canonical isomorphism $$M(X)\otimes^L M(Y)\stackrel{\sim}{\longrightarrow }M(X\times_kY)$$
in $\mathbf{D}_b(EHM)$, inducing the usual K\"unneth formula on singular homology.
\end{proof}

\begin{cor}\label{mm}
For any varieties $X$ and $Y$, there is a canonical isomorphism in $EHM$
$$\bigoplus_{p+q=n}h_{p}(X)\otimes h_q(Y)\stackrel{\sim}{\longrightarrow} h_n(X\times Y).$$
\end{cor}
\begin{proof}
This follows immediately from the previous theorem since any object of $EHM$ is flat.
\end{proof}

Applying Nori's universal construction to the diagram of arbitrary pairs of varieties, or proceeding as we did for homological motives (see Notation \ref{notdansehm}), we may define $h^i(X)\in ECM$ for any variety $X$.

\begin{cor}\label{K}For any varieties $X$ and $Y$, there is a canonical isomorphism in $ECM$.   
$$\bigoplus_{p+q=n}h^{p}(X)\otimes h^q(Y)\stackrel{\sim}{\longleftarrow} h^n(X\times Y).$$
\end{cor}
\begin{proof} Let $C(D, T)$ be the diagram category associated to a diagram $D$ and a representation $T: D\rightarrow {\mathrm{Vec}}_\mathbb Q$. First we exhibit a canonical equivalence of categories
\begin{equation}\label{explicit}  
C(D, T)\simeq C(D^{op}, T^\vee)^{op}.
\end{equation}
To this aim, we consider the  following square 
\[ \xymatrix{
D\ar[d]^{G}\ar[r]^{\tilde T}
&{C(D, T)}\ar[d]^{f_T}\\
C(D^{op}, T^\vee)^{op}\ar[r]^{f_{T^\vee}^\vee}
&{{\mathrm{Vec}}_\mathbb Q}}
\]
where $G$ is the representation of $D$ obtained by composing the functors 
$$D\rightarrow D^{op}\rightarrow C(D^{op}, T^\vee)^{op}. $$ Since the square  is commutative we know from (\cite{Nori}, Theorem 1.6),  that there exists a unique $\mathbb Q$-linear functor $L(G): C(D, T)\rightarrow 
C(D^{op}, T^\vee)^{op}$ such that $$G=L(G)\circ \tilde T\ \mathrm{and}\ f_T=f_{T^\vee}^\vee\circ L(G).$$ The functor $L(G)$ is an equivalence of categories. A quasi inverse for $L(G)$ is constructed by considering the square attached to the pair $(D^{op}, T^\vee)$.

We need to give a module-theoretic interpretation of the  functor $L(G)$. Recall that the category $C(D, T)$ is defined in \cite{Nori} by 
$$C(D, T)=\mathrm{colim}_FC(F, T|F)$$ where $F$ runs through all the finite subdiagrams $F$ of $D$ and the $2$-colimit is taken in the $2$-category of small categories. If $F $ is finite, then $C(F, T|F)$ is the category of finitely generated left $\mathrm{End}(T|F)$-modules where  
$$\mathrm{End}(T|F)=\{u=(u_p)_{p\in F}\in \mathbf{\prod}_{p\in F}\mathrm{End}(T(p))  | u_q\circ T(m)=T(m)\circ u_p, \forall m: p\rightarrow q, p, q \in F\}.$$
For $M$ an 
$\mathrm{End}(T|F)$-module,  we have ring homomorphisms 
\[ \xymatrix{
\mathrm{End}(T|F)\ar[r]^f&\mathrm{End}_{\mathbb Q}(M)\ar[r]^\vee&\mathrm{End}_{\mathbb Q}(M^\vee)^{op}}
\]
where $f$ is given by the action of $\mathrm{End}(T|F)$ on $M$ and where $\vee$ associates to any  $u\in \mathrm{End}(M)$ its transposed $u^\vee$. Moreover the map $u=(u_p)_{p\in F}\rightarrow u^\vee=(u_p^\vee)_{p\in F}$  induces a ring isomorphism $\mathrm{End}(T|F)\simeq\mathrm{End}(T^\vee|F)^{op}$. Therefore we obtain by composition a ring homomorphism 
$$\mathrm{End}(T^\vee|F)\rightarrow \mathrm{End}(T|F)^{op}\rightarrow \mathrm{End}_{\mathbb Q}(M^\vee)$$ and thus a structure of $\mathrm{End}(T^\vee|F)$-module on $M^\vee$. We let $\alpha_F$ be the functor 
$$\appl{C(F,T|F)}{C(F^{op},T^\vee|F)^{op}}{M}{M^\vee}$$
and we denote by $\alpha: C(D, T)\rightarrow C(D^{op},T^\vee)^{op}$ the functor defined by the $(\alpha_F)_{F\subset D}$. It is easy  to check that $G=\alpha\circ \tilde T$ and $f_T=f_{T^\vee}^\vee\circ \alpha$. It follows from the universal property of $C(D, T)$ that $L(G)$ and $\alpha$ are isomorphic.

Applying these observations to the diagram of good pairs,  we obtain an equivalence
$$\alpha: EHM\stackrel{\sim}{\longrightarrow} ECM^{op}.$$
We have $\alpha(h_i(X))=h^i(X)$ for any integer $i$ and any variety $X/k$. Moreover the functor $\alpha$ sends direct sums to direct sums. Thus in order to deduce  Corollary \ref{K} from Corollary \ref{mm}  it suffices to prove that $\alpha$ sends tensor products to tensor products. The categories $EHM$ and $ECM$ are both diagram categories  attached to a graded diagram endowed with a commutative product structure with unit and a graded unital multiplicative representation. Using the definition of the tensor product on such categories as defined in (\cite{Huber-Muller-Stach-12}, Proposition B.16) and in  (\cite{Huber-Muller-Stach-15}, Chapter VII, Section 7.1) and using the description  of $\alpha$ in terms of modules given above, we check that $\alpha$ respects tensor products as required.
 \end{proof}

\subsection{Poincar\'e duality}

We denote by $\mathrm{NMM}_k$ the category of cohomological Nori motives.

\begin{cor}\label{poincaremotiv}
Let $V/k$ be a smooth projective variety of dimension $n$. There is a map in $\mathrm{NMM}_k$
\begin{equation}\label{motivic-poincare}
h^i(V)(j)\otimes h^{2n-i}(V)(n-j)\longrightarrow h^{2n}(V)(n)\stackrel{\mathrm{Tr}}{\longrightarrow}{\bf 1}
\end{equation}
inducing the usual pairings in Betti and de Rham cohomology for any $i,j\in\mathbb{Z}$. In particular, (\ref{motivic-poincare}) yields an canonical isomorphism
$$h^i(V)(j)\stackrel{\sim}{\longrightarrow}h^{2n-i}(V)(n-j)^{\vee}.$$
\end{cor}

\begin{proof}
Corollary \ref{K} gives the canonical map in $\mathrm{NMM}_k$
$$h^i(V)(j)\otimes h^{2n-i}(V)(n-j)\longrightarrow h^{2n}(V\times V)(n)\longrightarrow h^{2n}(V)(n)$$
inducing the same map on Betti cohomology. Here $h^{2n}(V\times V)\rightarrow h^{2n}(V)$ is induced by the diagonal map.

Let us show the existence of the trace map
$$h^{2n}(V)(n)\stackrel{\mathrm{Tr}}{\rightarrow}{\bf 1}.$$
We follow \cite{Huber-Muller-Stach-12} with a slight modification.
Since $h^{i}(V_1\bigsqcup V_2)=h^{i}(V_1)\oplus h^i(V_2)$, we may
assume that $V$ is connected. Let $V\rightarrow \mathbf{P}^N_{k}$ be a closed embedding. We consider the finite extension  $K:=\mathcal{O}_V(V)$ of $k$, so that $V/K$ is geometrically connected. We obtain a closed embedding $V\rightarrow\mathbf{P}^N_{k}\times_{k}K=\mathbf{P}^N_{K}$, inducing a map 
$\varphi:h^{2n}(\mathbf{P}^N_{K})\rightarrow h^{2n}(V)$. But $H^*_B(\varphi)$ is an isomorphism, hence so is $\varphi$. 
By (\cite{Huber-Muller-Stach-12} Lemma 1.12(1)) we have a canonical isomorphism $h^{2n}(\mathbf{P}^N_{k})\simeq {\bf 1}(-n)$.
We obtain $$h^{2n}(\mathbf{P}^N_{K})=h^{2n}(\mathbf{P}^N_{k}\times \mathrm{Spec}(K))\simeq h^{2n}(\mathbf{P}^N_{k})\otimes h^0(\mathrm{Spec}(K))\simeq h^0(\mathrm{Spec}(K))(-n).$$
One is therefore reduced to define the trace map $h^0(\mathrm{Spec}(K))\rightarrow{\bf 1}$ in the category of Artin motives, i.e. in the category of $\mathbb{Q}$-linear representations of $G_{k}$ (see Proposition \ref{prop-Artin}). But $h^0(\mathrm{Spec}(K))$ corresponds to the linear $G_{k}$-action on $\mathbb{Q}^{\mathrm{Hom}_{k}(K,\overline{k})}$ by permutation of the factors, hence the sum map
$$\appl{\mathbb{Q}^{\mathrm{Hom}_{k}(K,\overline{k})}}{\mathbb{Q}}{(v_{\sigma})_{\sigma}}{\sum_{\sigma} v_{\sigma}}$$
is $G_{k}$-equivariant. This defines a map $h^0(\mathrm{Spec}(K))\stackrel{\sum}{\rightarrow}{\bf 1}$. The trace map of the lemma is then given by composition
$$h^{2n}(V)(n)\stackrel{\sim}{\longleftarrow}h^{2n}(\mathbf{P}^N_{K})(n)\simeq h^0(\mathrm{Spec}(K))\stackrel{\sum}{\longrightarrow}{\bf 1}.$$
This map induces the usual trace map in Betti cohomology.

So we have the map $$m:h^i(V)(j)\otimes h^{2n-i}(V)(n-j)\longrightarrow h^{2n}(V)(n)\stackrel{\mathrm{Tr}}{\longrightarrow}{\bf 1}$$ and this map induces the usual pairing on Betti cohomology. 
Applying the tensor functors $H^*_B$ and $H^*_{dR}$, we obtain product maps
$$H^*_B(m):H^i(V(\mathbb{C}),(2i\pi)^j\mathbb{Q})\otimes H^{2n-i}(V(\mathbb{C}),(2i\pi)^{n-j}\mathbb{Q})
\longrightarrow H^{2n}(V(\mathbb{C}),(2i\pi)^n\mathbb{Q})\stackrel{\mathrm{Tr}}{\longrightarrow} \mathbb{Q}$$
and
$$H^*_{dR}(m):H^i(V/k)\otimes H^{2n-i}(V/k)
\longrightarrow H^{2n}(V/k)\stackrel{\mathrm{Tr}}{\longrightarrow} k.$$
On the other hand, we denote by 
$$m_B:H^i(V(\mathbb{C}),(2i\pi)^j\mathbb{Q})\otimes H^{2n-i}(V(\mathbb{C}),(2i\pi)^{n-j}\mathbb{Q})
\longrightarrow H^{2n}(V(\mathbb{C}),(2i\pi)^n\mathbb{Q})\stackrel{\mathrm{Tr}}{\longrightarrow} \mathbb{Q}$$
and
$$m_{dR}:H^i_{dR}(V/k)\otimes H^{2n-i}_{dR}(V/k)
\longrightarrow H^{2n}(V/k)\stackrel{\mathrm{Tr}}{\longrightarrow} k$$
the usual product map in Betti cohomology (given by cup-product) and in de Rham cohomology (given by wedge product) respectively. The usual period isomorphism yields an isomorphism of tensor functors $H_B^*\otimes_{\mathbb{Q}}\mathbb{C}\simeq H_{dR}^*\otimes_k\mathbb{C}$, hence the square
\[ \xymatrix{
H^i(V(\mathbb{C}),(2i\pi)^j\mathbb{Q})_{\mathbb{C}}\otimes H^{2n-i}(V(\mathbb{C}),(2i\pi)^{n-j}\mathbb{Q})_{\mathbb{C}}\ar[r]^{\hspace{3.5cm}H^*_B(m)\otimes \mathbb{C}}\ar[d]^{\simeq}&\mathbb{C}\ar[d]^{\mathrm{Id}}\\
H^i_{dR}(V/k)_{\mathbb{C}}\otimes H^{2n-i}_{dR}(V/k)_{\mathbb{C}}\ar[r]^{\hspace{2.5cm}H^*_{dR}(m)\otimes \mathbb{C}}
&\mathbb{C}
}
\]
commutes, where the vertical isomorphisms are given by the period isomorphism. But the square
\[ \xymatrix{
H^i(V(\mathbb{C}),(2i\pi)^j\mathbb{Q})_{\mathbb{C}}\otimes H^{2n-i}(V(\mathbb{C}),(2i\pi)^{n-j}\mathbb{Q})_{\mathbb{C}}\ar[r]^{\hspace{3.5cm}m_B\otimes \mathbb{C}}\ar[d]^{\simeq}&\mathbb{C}\ar[d]^{\mathrm{Id}}\\
H^i_{dR}(V/k)_{\mathbb{C}}\otimes H^{2n-i}_{dR}(V/k)_{\mathbb{C}}\ar[r]^{\hspace{2.5cm}m_{dR}\otimes \mathbb{C}}
&\mathbb{C}
}
\]
commutes as well. By construction of $m$, we have 
$H^*_B(m)=m_B$. It follows that $H^*_{dR}(m)\otimes_k\mathbb{C}=m_{dR}\otimes_k\mathbb{C}$, hence that 
$H^*_{dR}(m)=m_{dR}$. The first claim of the corollary is proved.

By adjunction, the map $m$ gives a map $d:h^i(V)(j)\rightarrow h^{2n-i}(V)(n-j)^{\vee}$ inducing the usual map
$H^*_{B}(d):H^i(V(\mathbb{C}),(2i\pi)^j\mathbb{Q})\rightarrow H^{2n-i}(V(\mathbb{C}),(2i\pi)^{n-j}\mathbb{Q})^{\vee}$ in Betti cohomology. But $H^*_{B}(d)$ is an isomorphism, hence so is $d$, since $H^*_B$ is conservative.
\end{proof}

We record the following special case of Corollary \ref{poincaremotiv}.

\begin{cor}\label{corqSaito} Let $V/k$ be a smooth projective algebraic variety of even dimension $n$ over the number field $k$. Then  $h^n(V)(n/2)\in\mathrm{NMM}_k$ has a canonical orthogonal structure $q$ such that the forms $q_B$ and $q_{dR}$ 
are the perfect pairings
\begin{equation}\label{qB}
q_B:H^n_{B}(V(\mathbb{C}),(2i\pi)^{n/2}\mathbb{Q})\otimes H^n_{B}(V(\mathbb{C}),(2i\pi)^{n/2}\mathbb{Q})\stackrel{\cup}{\longrightarrow} H^{2n}_{B}(V(\mathbb{C}),(2i\pi)^{n}\mathbb{Q})\stackrel{\mathrm{Tr}}{\longrightarrow} E
\end{equation}
and
\begin{equation}\label{qdR}
q_{dR}:H^n_{dR}(V/k)\otimes H^n_{dR}(V/k)\stackrel{\wedge}{\longrightarrow}H^{2n}_{dR}(V/k)\stackrel{\mathrm{Tr}}{\longrightarrow} k
\end{equation}
given by cup-product and trace maps.
\end{cor}

\subsection{Artin motives} As pointed out to us by A. Huber, it is not clear that the Tannakian subcategory $\mathrm{NMM}^0_k$ of $\mathrm{NMM}_k$ generated by the objects of the form $h^0(\mathrm{Spec}(K))$, where $K/k$ is a finite extension, is the category $\mathrm{Rep}_{\mathbb{Q}}(G_k)$ of classical Artin motives. This is essentially due to the fact that one does not know that the tensor functor $ECM_k\rightarrow\mathrm{NMM}_k$ is fully faithful. The aim of this section is to clarify this point.

Let $\mathcal{T}'$ be a Tannakian category over the field $k$ (i.e. a rigid abelian tensor category such that $\mathrm{End}(\mathbf{1})=k$). A Tannakian subcategory $\mathcal{T}\subset \mathcal{T}'$ is a strictly full subcategory which is stable by finite tensor products (in particular $\mathbf{1}\in \mathcal{T}$), finite sums, duals and subquotients. Let $a:\mathcal{T}\rightarrow \mathcal{T}'$ be a $k$-linear tensor functor between Tannakian categories over $k$. We say that $a$ \emph{identifies $\mathcal{T}$ with a Tannakian subcategory of $\mathcal{T}'$} if $a$ is exact, fully faithful and if its essential image is a Tannakian subcategory of $\mathcal{T}'$. Let $a:\mathcal{T}\rightarrow \mathcal{T}'$ be an exact tensor functor, let $\omega':\mathcal{T}'\rightarrow \mathrm{Vec}_k$ be a neutralizing fiber functor, let $\omega:=\omega'\circ a$ and let $\mathcal{G}_{\omega}:=\mathbf{Aut}^{\otimes}(\omega)$ and $\mathcal{G}_{\omega'}:=\mathbf{Aut}^{\otimes}(\omega')$ be the Tann
 aka duals of $\mathcal{T}$ and $\mathcal{T}'$ respectively. Then $a$ identifies $\mathcal{T}$ with a Tannakian subcategory of $\mathcal{T}'$ if and only if the morphism $\mathcal{G}_{\omega'}\rightarrow \mathcal{G}_{\omega}$ is faithfully flat.

\begin{lem} Let $\mathcal{T}$, $\mathcal{T}'$ and $\mathcal{T}''$ be $k$-linear Tannakian categories and let 
\[ \xymatrix{
\mathcal{T}\ar[r]^{a}\ar[dr]^{c}
&\mathcal{T}'\ar[d]^{b}\ar[rd]^{\omega'}&\\
&\mathcal{T}''\ar[r]^{\omega''}&\mathrm{Vec}_k
}
\]
be a commutative diagram of tensor categories, where $a$, $b$, $c$, $\omega$ and $\omega'$ are $k$-linear tensor functor (in particular, we have specified isomorphisms of tensor functors $c\simeq b\circ a$ and $\omega'\simeq \omega''\circ b$). We assume that 

\begin{itemize}
\item $\omega'$ and $\omega''$  are fiber functors (i.e. exact and faithful);
\item $c$ identifies $\mathcal{T}$ with a Tannakian subcategory of $\mathcal{T}''$;
\end{itemize}
Then $a$ identifies $\mathcal{T}$ with a Tannakian subcategory of $\mathcal{T}'$.
\end{lem}

\begin{proof}
By Tannaka duality, the fiber functor $\omega'$ can be identified with the forgetful functor $\mathrm{Rep}_k(\mathcal{G}_{\omega'})\rightarrow \mathrm{Vec}_k$. Hence $\omega'$ reflects short exact sequences, i.e. a (short) sequence is exact in $\mathcal{T}'$ if and only if it is exact in $\mathrm{Vec}_k$ after applying $\omega'$, and similarly for $\omega''$. Moreover, $\omega'$ is injective on subobjects, in the sense that for any $V'\in\mathcal{T}'$ the map $\mathrm{Sub}_{\mathcal{T}'}(V')\rightarrow \mathrm{Sub}_{\mathrm{Vec}_k}(\omega'(V'))$ is injective, where 
$\mathrm{Sub}_{\mathcal{T}'}(V')$ is the set of subobjects of $V'$ in $\mathcal{T}'$. It follows that $b$ is exact and (hence) faithful, and that $a$ is exact. It also follows that $b$ is injective on subobjects.

The fact that $b$ is faithful and $c$ fully faithful immediately implies that $a$ is fully faithful. Similarly the fact that $b$ is injective on subobjects and $c$ bijective on subobjects implies that $a$ is bijective on subobjects (i.e. the map $\mathrm{Sub}_{\mathcal{T}}(V)\rightarrow \mathrm{Sub}_{\mathcal{T}'}(a(V))$ is bijective for any $V\in\mathcal{T}$). It follows that the essential image of $a$ is stable by subobjects. The set of quotients $\mathrm{Quot}_{\mathcal{T}}(V)$ of $V\in\mathcal{T}$ (resp. $\mathrm{Quot}_{\mathcal{T'}}(a(V))$) is in 1-1 correspondence with $\mathrm{Sub}_{\mathcal{T}}(V)$ (resp. $\mathrm{Sub}_{\mathcal{T'}}(a(V))$). Hence the essential image of $a$ is stable by subquotients. It is also stable by finite tensor products, duals and finite direct sums since $a$ is an exact tensor functor.

\end{proof}

Consider the diagram $\mathrm{Pairs}^0$ of pairs $(X,Y,i)$ where $\mathrm{dim}(X)=0$ endowed with the representation $H^*$ given by Betti cohomology with $\mathbb{Q}$-coefficients, i.e. $H^*(X,Y,i)=H^i(X(\mathbb{C}),Y(\mathbb{C}),\mathbb{Q})$. Let $ECM_k^0:=C(\mathrm{Pairs}^0,H^*)$ be its diagram category. Sending the vertex $(\mathrm{Spec}(A),\emptyset,0)$ to the representation $\mathbb{Q}^{\mathrm{Hom}_k(A,\mathbb{C})}$, where $A/k$ is a finite \'etale algebra, leads to a representation $\mathrm{Pairs}^0\rightarrow \mathrm{Rep}_{\mathbb{Q}}(G_k)$ inducing a canonical tensor equivalence $\alpha:ECM_k^0\stackrel{\sim}{\rightarrow} \mathrm{Rep}_{\mathbb{Q}}(G_k)$, where $G_k$ is the absolute Galois group of $k$ (see \cite{Huber-Muller-Stach-15} Section 12.3). Let $\alpha^{-1}:\mathrm{Rep}_{\mathbb{Q}}(G_k)\stackrel{\sim}{\rightarrow}ECM_k^0$ be a tensor-inverse and consider the tensor functor
$$a:\mathrm{Rep}_{\mathbb{Q}}(G_k)\stackrel{\sim}{\rightarrow} ECM_k^0\rightarrow ECM_k \rightarrow \mathrm{NMM}_k.$$

\begin{prop}\label{prop-Artin}
The tensor functor $a:\mathrm{Rep}_{\mathbb{Q}}(G_k) \rightarrow \mathrm{NMM}_k$
identifies $\mathrm{Rep}_{\mathbb{Q}}(G_k)$ with the Tannakian subcategory $\mathrm{NMM}^0_k$ of $\mathrm{NMM}_k$ generated by the objects of the form $h^0(\mathrm{Spec}(K))$ where $K/k$ is a finite extension.
\end{prop}

\begin{proof}
Let $\mathcal{T}=\mathrm{Rep}_{\mathbb{Q}}(G_k)$, let  $\mathcal{T}'=\mathrm{NMM}_k$ and let $\mathcal{T}''=\mathcal{R}_{l}$ be the $\mathbb{Q}$-linear Tannakian category whose objects are pairs $(V,\rho_l)$ where $V$ is a finite dimensional $\mathbb{Q}$-vector space and $\rho_l$ is a continuous representation of $G_k$ on $V\otimes \mathbb{Q}_l$. The functor $\omega'$ is the Betti realization and $\omega''$ is the functor $(V,\rho_l)\mapsto V$. The functor $b$ is induced by the Betti and the $l$-adic realization together with Artin's comparison isomorphism, and $c:\mathrm{Rep}_{\mathbb{Q}}(G_k)\rightarrow \mathcal{R}_{l}$ is the obvious functor. The lemma above applies and shows that the tensor functor $a$
identifies $\mathrm{Rep}_{\mathbb{Q}}(G_k)$ with a Tannakian subcategory of $\mathrm{NMM}_k$.

Let $\mathrm{Im}(a)$ be the essential image of $a$, and let $\mathrm{NMM}^0_k$ be the Tannakian subcategory  of $\mathrm{NMM}_k$ generated by the objects of the form $h^0(\mathrm{Spec}(K))$ where $K/k$ is a finite extension. We need to check that $\mathrm{Im}(a)=\mathrm{NMM}^0_k$. The functor $a$ maps the representation $\mathbb{Q}^{\mathrm{Hom}_k(K,\mathbb{C})}$ to $h^0(\mathrm{Spec}(K))\in \mathrm{Im}(a)$, for any finite extension $K/k$. Since $\mathrm{Im}(a)$ is a Tannakian subcategory of $\mathrm{NMM}_k$, we have that $\mathrm{NMM}^0_k \subset\mathrm{Im}(a)$ is a Tannakian subcategory. Let $\alpha: \mathrm{Im}(a)\rightarrow \mathrm{Rep}_{\mathbb{Q}}(G_k)$ be a tensor-inverse for the functor $\mathrm{Rep}_{\mathbb{Q}}(G_k)\rightarrow \mathrm{Im}(a)$ induced by $a$. Let $\alpha(\mathrm{NMM}^0_k)$ be the essential image of the functor $\mathrm{NMM}^0_k \hookrightarrow \mathrm{Im}(a)\stackrel{\alpha}{\rightarrow}\mathrm{Rep}_{\mathbb{Q}}(G_k)$. Then $\alpha(\mathrm{NMM}^0_k)\
 subset  \mathrm{Rep}_{\mathbb{Q}}(G_k)$ is a Tannakian subcategory containing the representations of the form $\mathbb{Q}^{\mathrm{Hom}_k(K,\mathbb{C})}$, where $K/k$ is a finite Galois extension (in particular the regular representation of any finite quotient $G$ of $G_k$ belongs to $\alpha(\mathrm{NMM}^0_k)$). It follows that $\alpha(\mathrm{NMM}^0_k)=\mathrm{Rep}_{\mathbb{Q}}(G_k)$ hence that $\mathrm{NMM}^0_k  =\mathrm{Im}(a)$.
\end{proof}

\end{document}